\documentclass[aap,preprint]{imsart}

\usepackage{amsmath, amsthm,amssymb,bbm, bm,multirow}
\usepackage{mathpazo}

 \usepackage[top=1.2in, bottom=1.1in, left=1.3in, right=1.3in]{geometry}

\usepackage{pdfsync}

\RequirePackage[OT1]{fontenc}

\startlocaldefs
\theoremstyle{definition}
\newtheorem{assumption}{Assumption}[part]

\newtheorem{theorem}{Theorem}[section]
\newtheorem{lemma}[theorem]{Lemma}
\newtheorem{proposition}[theorem]{Proposition}
\newtheorem{corollary}[theorem]{Corollary}
\newtheorem{definition}[theorem]{Definition}
\newtheorem{remark}[theorem]{Remark}
\numberwithin{equation}{section}

\endlocaldefs

\def\initial{I_0}

\def\R{{\mathbb R}}
\def\N{{\mathbb N}}
\def\Z{{\mathbb Z}}

\def\C{{\mathbb C}}
\def\AC{\mathbb{AC}}
\def\Z{{\mathbb Z}}
\def\D{\mathbb{D}_{\R}[0,\infty)}

\def\Yn{Y^{(N)}}
\def\Ynt{\widetilde{Y}^{(N)}}
\def\Y{\mathcal{Y}}

\def\phiset{\mathbb{C}^{1,1}_c(\supint\times\R_+)}
\def\tightphiset{\mathbb{C}_b(\supint\times\R_+)}
\def\fset{\mathbb{C}_c^1\supint}

\def\cal{\mathcal}
\def\id{\textbf{Id}}
\def\f1{\mathbf{1}}
\newcommand{\indic}[1]{\mathbb{1}_{\{#1\}}}
\newcommand{\indicil}[1]{\mathbb{1}_{\{#1\}}}

\def\terma{{\mathcal J}_1}
\def\termb{{\mathcal J}_2}
\def\termc{{\mathcal J}_3}
\def\termcb{{\mathcal J}_{3}^\prime}

\newcommand{\Xth}[1]{#1^{\text{th}}}
\newcommand{\Ept}[1]{\mathbb{E}\left[#1\right]}
\newcommand{\Eptil}[1]{\mathbb{E}[#1]}
\newcommand{\Prob}[1]{\mathbb{P}\left\{#1\right\}}
\newcommand{\Probil}[1]{\mathbb{P}\{#1\}}
\newcommand{\deq}[0]{\overset{d}{=}}

\newcommand{\twopartdef}[4]{	\left\{		\begin{array}{ll}	#1 & \mbox{if } #2 \\[3mm]	#3 & \mbox{if } #4	 \end{array}	\right.}

\newcommand{\threepartdef}[6]{	\left\{		\begin{array}{lll}	#1 & \mbox{if } #2  \\[3mm]	#3 & \mbox{if }  #4	\\[3mm]#5 & \mbox{#6}  \end{array}	\right.}
\newcommand{\threepartdefE}[6]{	\left\{		\begin{array}{lll}	#1 & \mbox{if } #2  \\	#3 & \mbox{if }  #4	\\#5 & \mbox{if } #6 \end{array}	\right.}

\def\endsup{L}
\def\endsupzero{\ell_0}
\def\supint{[0,\endsup)}

\def\delayn{u^{(N)}_0}

\def\ren{R_E^{(N)}}
\def\re{R_E}

\def\xnbar{\overline{X}^{(N)}}
\def\xn{X^{(N)}}
\def\xii{X^i}
\def\xni{X^{(N),i}}
\newcommand{\xnX}[1]{X^{(N),#1}}
\newcommand{\xX}[1]{X^{#1}}

\def\arrivaln{E^{(N)}}
\def\arrival{E}
\def\arrivalni{E^{(N),i}}
\def\arrivali{E^{i}}
\newcommand{\arrivalnX}[1]{E^{(N),#1}}
\def\arrivalnbar{\overline{E}^{(N)}}

\def\nun{\nu^{(N)}}
\def\nunbar{\overline{\nu}^{(N)}}
\def\nuln{\nu_\ell^{(N)}}

\def\nuni{\nu^{(N),i}}

\def\nunone{S^{(N)}}
\def\nuone{S}
\def\nunonebar{\overline{S}^{(N)}}
\def\nuonebar{\overline{S}}
\def\kn{K^{(N)}}

\def\d{D}
\def\dn{D^{(N)}}
\def\dni{D^{(N),i}}

\def\di{D^i}
\def\dnbar{\overline{D}^{(N)}}

\def\Uactiven{\mathcal{V}^{(N)}}
\def\Usizeln{\mathcal{U}^{(N)}_\ell}
\def\Usizen{\mathcal{U}^{(N)}}

\def\agen{a^{(N)}}
\def\ageni{a^{(N),i}}
\def\agei{a^{i}}
\newcommand{\agenX}[1]{a^{(N),#1}}
\def\age{a}

\def\tM{\tilde{\mathbb{M}}}
\def\supp{{\rm{supp}}}

\def\j0{j_0}%
\def\jo0{-\xn(0)+1}
\def\ki0{-k_{0,i}^{(N)}}

\def\serviceTime{v}

\def\kninv{\alpha^{(N)}}
\def\kinv{\alpha}

\def\dept{\beta}
\def\deptn{\beta^{(N)}}

\def\arrivalTime{\gamma^{(N)}}
\def\arrivalTimenon{\gamma}

\def\eventtimen{\tau^{(N)}}

\def\eventmarkn{z^{(N)}}
\def\MPn{\mathcal{T}^{(N)}}

\def\eventtime{\tau}

\def\eventmark{z}

\def\stationn{\kappa^{(N)}}
\def\station{\kappa}
\def\rqc{\iota}

\def\custServSize{\chi^{(N)}}
\def\custServSizenon{\chi}
\def\custES{\chi^{E,(N)}}
\def\custKS{\chi^{K,(N)}}
\def\custDS{\chi^{D,(N)}}

\def\custDSnon{\chi^{D}}

\def\Que{\cal D}
\def\Quephinl{\cal D_{\varphi,\ell}^{(N)}}

\def\Quen{\cal D^{(N)}}
\def\Queni{\cal D^{(N),i}}

\def\Quenbar{\overline{\cal D}^{(N)}}
\def\Quephinli{\cal D_{\varphi,\ell}^{(N),i}}
\newcommand{\QuephinlX}[1]{\cal D_{\varphi,\ell}^{(N),#1}}

\def\An{A^{(N)}}
\def\Anbar{\overline{A}^{(N)}}
\def\Mn{\cal M^{(N)}}
\def\Mni{\cal M^{(N),i}}
\def\Mnbar{\overline{\cal M}^{(N)}}

\def\Rphinl{\cal R_{\varphi,\ell}^{(N)}}
\def\Rn{\cal R^{(N)}}
\def\Rni{\cal R^{(N),i}}
\newcommand{\RnX}[1]{\cal R^{(N),#1}}
\def\Rnbar{\overline{\cal R}^{(N)}}

\def\Bphinl{B_{\varphi,\ell}^{(N)}}
\def\Bn{B^{(N)}}
\def\Bnbar{\overline{B}^{(N)}}
\def\Nn{\cal N^{(N)}}
\def\Nni{\cal N^{(N),i}}
\def\Nnbar{\overline{\cal N}^{(N)}}

\def\filt{\mathcal{F}}
\def\filtn{\mathcal{F}^{(N)}}

\def\Gen{G_E^{(N)}}
\def\Ge{G_E}
\def\bGen{\overline G_E^{(N)}}
\def\bGe{\overline G_E}
\def\gen{g_E^{(N)}}
\def\gee{g_E}
\def\hen{h_E^{(N)}}
\def\he{h_E}

\def\Pstar{P^*}
\def\Gstar{G^*}
\def\gstar{G^*}
\def\hstar{h^*}
\def\rstar{r^*}
\def\Ustar{U^*}
\def\ustar{u^*}

\def\s{\mathbb{S}}

\newcommand{\hc}[0] {[0,\infty)}
\newcommand{\ho}[0] {(0,\infty)}

\def\In{\mathcal{I}^{(n)}}

\def\busyn{\mathfrak{B}^{(N)}}
\def\busy{\mathfrak{B}}

\newcommand{\poly}[2]{\mathfrak{P}_d\left(#1,#2\right)}

\def\mset{\mathcal{Z}}

\def\filtData{\mathcal{F}^{\text{Data}}}

\def\tnun{\tilde{\nu}^{(N)}}
\def\tkn{\tilde{K}^{(N)}}
\def\tdn{\tilde{D}^{(N)}}
\def\txn{\tilde{X}^{(N)}}
\def\tV{\tilde{{\mathcal V}}}
\def\tvn{\tilde{\mathcal V}^{(N)}}
\def\ta{\tilde{a}}
\def\tQuen{\tilde{\mathcal{D}}^{(N)}}
\def\tilQuen{\tilde{{\mathcal D}}^{(N)}}
\def\tilAn{\tilde{A}^{(N)}}
\def\tmn{\tilde{{\mathcal M}}^{(N)}}

\def\inten{\varrho}
\def\tRni{\tilde{R}^{(N),i}}

\begin{document}

\begin{frontmatter}

\title{The hydrodynamic limit of a randomized load balancing network}

\begin{aug}
\author{\fnms{Reza} \snm{Aghajani} \ead[label=e1]
{maghajani@ucsd.edu}}
\and
\author{\fnms{Kavita} \snm{Ramanan}\ead[label=e2]
{kavita\_ramanan@brown.edu}}

\affiliation{Department of Mathematics, UCSD,
and Division of Applied Mathematics, Brown University}

\end{aug}

\begin{abstract}
Randomized load balancing networks arise in a variety of applications, and allow for efficient sharing of resources, while being relatively easy to implement.  We consider a network of parallel queues in which incoming jobs with  independent and identically distributed service times are assigned to the shortest queue among a randomly chosen subset of $d$ queues, and leave the network on completion of service.  Prior work on dynamical properties of this model has focused on the case of exponential service distributions. In this work, we analyze the more realistic case of  general service distributions. We first introduce a novel particle representation of the state of the network, and characterize the state dynamics via a sequence of interacting measure-valued stochastic processes. Under mild assumptions, we show that the sequence of scaled state processes converges, as the number of servers goes to infinity, to a hydrodynamic limit that is characterized as the unique solution to a  countable system of coupled deterministic measure-valued equations.  We also establish a propagation of chaos result that shows that finite collections of  queues are asymptotically independent.  The general framework developed here is potentially useful for analyzing a larger class of  models arising in diverse fields including biology and materials science.
\end{abstract}

\end{frontmatter}


\section{Introduction}\label{sec_introduction}

\subsection{Background and  Discussion of Results}

Randomized load balancing is a method for the efficient sharing of resources in networking systems that is relatively easy to implement, and  used in a variety of applications  such as, for example,  hash tables in data switches, parallel computing \cite{Mit01} and wireless networks \cite{Ganetal12}.  In this article, we  introduce a  mathematical framework for the analysis of a class of  large-scale parallel server load balancing networks in the presence of general service times, with the specific goal of obtaining a tractable characterization of the hydrodynamic limit of  randomized join-the-shortest-queue  networks, as the number of servers  goes to infinity. Past work on dynamical properties of this model  has essentially been restricted to the case of exponential service distributions. A key component of our approach that allows us to handle general service distributions is a
characterization of its dynamics via interacting measure-valued stochastic processes.   Our framework can be generalized and we expect similar
representations  to also  be  useful  for the study of other load-balancing models \cite{Mit00} as well as models arising in population
biology and materials science.


In  the randomized join-the-shortest-queue  network model,  also referred to as the supermarket model, jobs with independent and identically distributed (i.i.d.) service times arrive  according to a renewal process with rate $\lambda N$ to a network of $N$ identical  servers in parallel, each with an infinite-capacity queue.  Upon arrival of a job,  $d$ queues are sampled independently and uniformly at random (with replacement) and the job is routed to the shortest queue amongst those sampled, with ties broken uniformly at random.  Each server processes jobs from its queue  in a first-come first-serve (FCFS)  manner,  a server never idles when there is a job in its queue and jobs leave the network on completion of service.   The arrival process and service times are assumed to be mutually independent, and  service times of jobs have finite mean which, without loss of generality,  will be taken to be one. We refer to this model as the ($N$-server) SQ($d$) model.  A positive feature  of this algorithm is that its implementation does not require much system memory.

Several results are known when the arrival process is Poisson with rate $\lambda < 1$ and the service time is exponential (with unit mean). When $d=1$, the model reduces to a system of $N$ independent single-server queues with exponential service times, for which it is a classical result that  the stationary distribution of the length of a typical queue is geometric. When $d = 2$,  the stationary distribution of a typical queue  is not exactly computable, but it was shown in \cite{VveDobKar96} (also see \cite{Mit01} for the extension to $d > 2$)  that as the number of servers goes to infinity,  the limit of the stationary distributions of a typical queue  has a double exponential tail. This shows that  introducing just a little bit of random choice leads to a dramatic improvement in performance in equilibrium, a phenomenon that has been dubbed the ``power of two choices'' and  has led to substantial interest in this class of randomized load balancing schemes.

The analysis in \cite{VveDobKar96} proceeds by  representing the  dynamics  of the $N$-server network by a Markov  chain that keeps track of the fraction of queues that have $\ell$ or more jobs at time $t$, for each positive integer $\ell$, and
then applying the so-called ODE method \cite[Theorem 11.2.1]{KurtzBook} to  show that, as $N \rightarrow \infty$, the sequence of  Markov chains  converges weakly (on finite time intervals)   to the unique solution of a countable system of coupled $[0,1]$-valued  ordinary differential equations (ODEs).  Further tightness estimates are then used to prove  convergence of the stationary distributions to the unique invariant state of the ODE.  This basic approach was subsequently used to analyze various relevant modifications of the supermarket model \cite{FarMoaPra05, Mit00, Kat14}.  Other theoretical results on the SQ($d$) model with exponential service distributions in this asymptotic regime  include \cite{Gra00,LucNor05,MukMaz17}.

However, measurements in different applications have shown that service times are typically not exponentially distributed  \cite{BroEtAl05,CheEtAl14,GuaKoz14,Kol84}. In this case, the ODE method is no longer directly applicable because in order to describe the future evolution of the system, it is not sufficient to keep track of the fraction of queues with $\ell$ jobs at any time. For each job in service, one has also to keep track of additional information such as its age (the amount of time the job has spent in service)  or its residual service time. In the system with $N$ servers, this requires keeping track of $N$ additional nonnegative random variables, and thus the dimension of the Markovian state representation grows with $N$, which is not  conducive to obtaining a limit theorem.  Our goal is to develop a general framework for the analysis of this model and related models, that in particular enables
   an intuitive and tractable description of the hydrodynamic limit.


With this in mind,  we  introduce a novel  interacting particle representation of the state of the network that  allows for a description of the dynamics of all $N$-server systems on  a common (infinite-dimensional) state space. Specifically, we represent the state  of an $N$-server SQ($d$) network at any time $t$ in terms of a countable collection of finite measures $\nun(t) = \{ \nun_{\ell}(t)\}_{\ell \in \N}$, where $\nun_{\ell}(t)$ is the measure that has a ``particle''  or unit delta mass at the age of each  job that is in service at a queue of length greater than or equal to $\ell$ at time $t$, where  length denotes the number of jobs either waiting or in service.
 We then characterize  the dynamics of the $N$-server SQ($d$) model in terms of a  coupled system of interacting measure-valued stochastic processes (see Proposition \ref{prelim_dynamics_prop}). The main result of this article,  Theorem \ref{thm_convergence}, shows that under general conditions on the service time distribution and arrival processes,
as $N \rightarrow \infty$,  the  sequence of scaled  state processes $\nun/N$ converges weakly to the unique solution of a coupled system of deterministic measure-valued equations, which  we refer to as  the hydrodynamic equations (see Definition \ref{Def_fluid}).

The state-dependent routing structure makes the SQ($d$) model  substantially more challenging to analyze than certain other many-server models such as the GI/GI/N model studied in \cite{KasRam11,KanRam10}.
As a result,  the proof of convergence requires several new ingredients.  First, we use a marked point process representation of the dynamics (see Section \ref{sec_mpp}) that allows us to prove certain conditional independence properties that are used to identify compensators (with respect to a suitable filtration) of various auxiliary processes that govern the dynamics, such as the cumulative routing and departure processes (see Propositions \ref{prop_prelim_Rcomp} and \ref{prop_prelim_Qcomp}). Next,  we establish certain  renewal estimates to characterize the limit of the scaled compensators (in Sections \ref{sec_tightnessAux} and \ref{sec_limitRouting}).  Furthermore (in Section \ref{sec_limit}) we obain  an alternative dynamical characterization of the solution to the hydrodynamic equations (see Proposition \ref{prop_SolutionForm}). This is used to prove relative compactness of the sequence of scaled state processes in Theorem \ref{thm_tightness} and  show that any  subsequential limit satisfies the hydrodynamic equations. To complete the proof, we establish uniqueness of the solution to the hydrodynamic equations  (see Theorem \ref{thm_uniqueness}).  The hydrodynamic equations consist of a countable collection of coupled nonlinear measure-valued equations subject to non-standard boundary conditions that appear to fall outside the class considered in the literature.
 Two new ingredients that we introduce to  facilitate the uniqueness analysis is  a non-standard norm on the space of measure-valued paths and a characterization of  the evolution of this norm in terms of a certain renewal equation.

The hydrodynamic equations are useful for two purposes. First, characterization of dynamical behavior is important given the presence of non-stationary effects  in real networks. As shown in \cite{AghRam15d,AghThesis17}, under additional assumptions on the service distribution, the dynamics of relevant functionals of the state process such as the queue length and virtual waiting time can be captured by a simpler system of coupled classical partial differential equations (PDE). Furthermore, in \cite{AghRam15d}, this PDE was numerically solved to  show good agreement with simulations for moderately sized networks, and also to uncover non-intuitive behavior of load-balancing networks, such as, for example, the effect of  heavy-tailed service distributions. Thus, our approach yields a PDE method for analyzing randomized load balancing networks, which generalizes the more classical ODE  method that is valid only in the presence of exponential service distributions, and  can be adapted to study many other models. Second, the hydrodynamic equations can also be used to characterize  equilibrium behavior. To the best of our knowledge, the only prior work on the SQ($d$) model, $d \geq 2$, for a general class of non-exponential service distributions seems to be the work of Bramson, Lu and Prabhakar \cite{Bra11,BraLuPra10, BraLuPra12,BraLuPra13},  which   focuses on equilibrium behavior rather than dynamical behavior. In particular, under the assumption that  the arrival process is Poisson with rate $\lambda < 1$ and  the service distribution has a decreasing hazard rate, they show  that the stationary distribution of a typical queue in the $N$-server model converges to a limit, and uncover the interesting phenomenon that when the service distribution is power law,  its tail does not always have  double-exponential decay. The hydrodynamic equations introduced in this article pave the way for an alternative approach to analyzing the equilibrium behavior for a larger class of service distributions and  more general, renewal arrivals.  Indeed, the hydrodynamic equations are shown in \cite{AgaRam17} to  have a unique invariant state,  which serves as a candidate limit for the sequence of stationary state processes. Establishing convergence of the stationary distributions to this invariant state for the class of service distributions considered here is an interesting problem for future work.

The rest of the paper is organized as follows. Section \ref{sec_introNotation} lists some common notation.  Section \ref{sec_main} first introduces the basic assumptions of the model, the  state representation, and the definition of the hydrodynamic equations, and then states the main results.  Section \ref{sec_fluid} is devoted to the analysis of the hydrodynamic equations,  with uniqueness of the solution established in Section  \ref{sec_fluid_unique} and an alternative dynamical characterization of the solution obtained in Section \ref{sec_fluid_dynamics}. Section \ref{sec_model} contains a detailed description of the state dynamics in the $N$-server system. Martingale decompositions for the routing and departure processes are stated in  Section \ref{sec_prelimit_summary}. The proofs build on a marked point process representation and some conditional independence results established in Section \ref{sec_mpp}.   Finally, the main convergence results are established in Section \ref{sec_convergence}, with the proof of Theorem \ref{thm_convergence} presented in Section \ref{sec_limitCharacterization}.  Proofs of some technical lemmas are relegated to the Appendices.

\subsection{Common Notation}\label{sec_introNotation}

The following notation will be used throughout the paper. We  use $\Z$,  $\Z_+$ and $\N$ to denote the sets of integers, nonnegative integers and positive integers, respectively. Also, $\R$ is the set of real numbers and $\R_+$ the set of nonnegative real numbers. For $a, b \in \R$, $a \wedge b$ and $a \vee b$ denote the minimum and maximum of $a$ and $b$, respectively.  For a set $B$, $\indic{B,\cdot}$ is the indicator function of the set B (i.e., $\indic{B,x} = 1$ if $x \in B$ and $\indic{B,x} = 0$ otherwise). When $B$ is a measurable subset of a probability space $(\Omega,\mathbb{F})$, we usually omit the explicit dependence on $\omega$ and write $\indic{B,\omega}$ as $\indic{B}$. Moreover, with a slight abuse of notation, on every domain $V$, $\f1$ denotes the constant function equal to $1$ on $V$. Also, $\id$ is the identity function on $\hc$, that is, $\id(t)=t,$ for all $t\geq0$.

For a topological space $V$, we let $\C(V)$, $\C_b(V)$ and $\C_c(V)$ be, respectively, the space of continuous functions, bounded continuous functions, and continuous functions with compact support on $V$. For $f\in\C_b(V),$ $\|f\|_\infty$ denotes  $\sup_{s \in V} |f(s)|$. When $V=\hc$, for $T\geq0$,  $\|f\|_T$ denotes  $\sup_{s \in [0,T]} |f(s)|$, and recall that  $w_f(\delta, T)\doteq\sup\{|f(t)-f(s)|; s,t\in[0,T], |s-t|\leq \delta\},$ $\delta >0$, is the modulus of continuity of $f$ on the interval $[0,T]$.   For $V = [0, L)$, $L \in [0,\infty]$,  $\C_b^1(V)$ is the set of functions $f\in\C_b(V)$ for which the first derivative, denoted by $f^\prime$, exists and is bounded and continuous on $V$. Similarly, when $V \subset \R^2$ is the product of two intervals in $\R$, $\C_b^{1,1}(V) $ (respy, $\C_c^{1,1}(V)$) is the set of functions $(x,s) \mapsto \varphi (x,s)$ in $\C_b(V)$ ( resp. $\C_c(V)$) for which the first order partial derivatives $\varphi_x$ and $\varphi_s$ exist and are bounded and continuous (resp. continuous with compact support) on $V$.  Also, let $\AC (V)$ denote the space of real-valued functions that are absolutely continuous on every bounded subset of $V$.

For a metric space $\mathbb{X}$, $\mathbb{D}_{\mathbb{X}}\hc$ is the set of $\mathbb{X}$-valued functions on $\hc$ that  are right continuous and have finite left limits on $(0,\infty)$, and $\mathbb{C}_{\mathbb{X}}\hc$ is the subset of continuous functions on $\hc$.  For every function $f\in\mathbb{D}_{\mathbb{X}}\hc$ and $T\geq0,$ $w^\prime(f,\delta,T)$ is the modulus of continuity of $f$ in $\mathbb{D}_{\mathbb{X}}\hc$; see \cite[(3.6.2]{KurtzBook} for a precise definition of $w^\prime.$  Furthermore, for every function $f\in\D,$  we define
\[  [f]_t \doteq \lim_{|\pi|\to 0} \sum_{k=1}^n \left(f(t_k)-f(t_{k-1})\right)^2, \]
where the limit is taken over all partitions $\pi=\{t_0=0,t_1,...,t_n=t\}$  $[0,t]$ with $|\pi|\doteq\max_{k=1,...,n}|t_k-t_{k-1}|.$  When $f$ is a c\`{a}dl\`{a}g stochastic process, the limit is  defined in the sense of convergence in probability.

Finally, $\mathbb{L}^1\ho,$ $\mathbb{L}^2\ho$ and $\mathbb{L}^\infty\ho,$ denote, respectively, the spaces of integrable, square-integrable and essentially bounded functions on $\ho$, equipped  with their corresponding standard norms. Also, $\mathbb{L}^1_{\text{loc}}\ho$  denotes the space of locally integrable functions on $\hc.$  For any $f\in\mathbb{L}^1_{\text{loc}}\ho$ and a function $g$ that is bounded on finite intervals,  $g*f$ denotes the (one-sided) convolution of the two functions, defined as $f*g(t)\doteq\int_0^t f(t-s)g(s)ds$,  $t \geq0.$

For every subset $V$ of $\R$ or $\R^2$ endowed with  the Borel sigma-algebra, let $\mathbb{M}_{F}(V)$ (resp. $\mathbb{M}_{\leq 1}(V)$) be the space of finite positive (resp.\ sup-probability) measures on $V$. For $\mu \in \mathbb {M}_F(V)$ and any bounded Borel-measurable function $f$ on $V$, we denote the integral of $f$ with respect to $\mu$  by
\[   \langle f, \mu \rangle \doteq \int_{V}  f(x) \mu (dx).\]
Given $\mu \in \mathbb {M}_F(V)$ and  a function $f$ defined on a larger set $\tilde{V} \supseteq V$, by some abuse of notation, we will write $\langle f, \mu\rangle$ to denote $\langle f_{|V}, \mu \rangle = \int_{V} f(x) \mu (dx)$, where $f_{|V}$ denotes the restriction of $f$ to $V$. For every measure $\mu$ with representation $\mu=\mu^+-\mu^-;\mu^+,\mu^-\in \mathbb{M}_f\hc$,  we  extend the bracket notation by setting $\langle f,\mu\rangle\doteq\langle f,\mu^+\rangle- \langle f,\mu^-\rangle$. We  equip $\mathbb{M}_F (V)$ and $\mathbb{M}_{\leq 1}(V)$ with the weak topology: $\mu_n\Rightarrow \mu$ if and only if $\langle f,\mu_n\rangle \to \langle f,\mu\rangle$ for all $f\in\mathbb{C}_b\hc.$   Recall that the Prohorov metric $d_P$ on $\mathbb{M}_{F}(V)$ \cite[p.\ 72]{BillingsleyBook}  induces the same topology \cite[Theorem 6.8 of Chap. 1]{BillingsleyBook}.

Also, we denote by $\mathbb{M}(V)$ the space of Radon measures on $V$, that is, the space of measures on $V$  that assign finite  mass to every relatively compact subset of $V$.    Alternatively, one can identify  a Radon measure $\mu \in \mathbb{M}(V)$ with the  space of  linear functionals $\varphi \mapsto \mu(\varphi) \doteq\int_V \varphi(x) \mu (dx) $ on $\C_c(V)$  such that for every compact set ${\cal K}\subset V$, there exists a finite $C_{\cal K}$ such that
\[ \mu(\varphi)\leq C_{\cal K} \|\varphi\|_\infty, \forall \varphi \in \C_c(V)  \quad \text{ with supp}(\varphi)\subset{\cal K}.\]


\section{Main Results}\label{sec_main}

\subsection{Basic Assumptions}\label{sec_mianAsm}
Consider the SQ($d$) model with $N$ servers described in the introduction.  For  $t\geq0,$ let $\arrivaln(t)$ denote the number of jobs that arrived to the network in the interval $[0,t]$.  We start by stating our assumptions on the cumulative arrival process $\arrivaln$. Let $\widetilde E$ be a delayed renewal process with inter-arrival times $\widetilde u_n,n\geq1,$ whose cumulative distribution function $G_{\widetilde E}$   has a density $g_{\widetilde E}$ and mean $\lambda^{-1}$, for some $\lambda > 0$,  and delay $\widetilde{u_0}$ that satisfies $\Prob{\widetilde u_0>r}=\overline{G}_{\widetilde E}(\tilde R+r)/\overline{G}_{\widetilde E}(\tilde R)$, for some $\tilde R\geq0$,  where $\overline{G}_{\widetilde E} \doteq  1 - G_{\widetilde{E}}$.

\begin{assumption}\label{asm_E}
The arrival process satisfies $\arrivaln(t)=\widetilde E(Nt),t\geq0$.
\end{assumption}

Note that Assumption \ref{asm_E} implies that  $\arrivaln$ is a delayed renewal process with  delay $\delayn\doteq\widetilde u_0/N$, and inter-arrival times $u^{(N)}_n\doteq\widetilde u_n/N,n\geq1,$ with common distribution $\Gen(x)\doteq G_{\widetilde E}(Nx)$, $x\geq0$, and probability density function $\gen(\cdot)  = N g_E^{(N)} (N \cdot).$ Moreover,  setting
$R^{(N)}\doteq\widetilde R/N$,
 we have
\begin{equation}\label{delay_dist}
 \Prob{\delayn>r}=\frac{\bGen(R^{(N)}+r)}{\bGen(R^{(N)})}, \quad r \geq 0.
\end{equation}
For future purposes, we also  define the backward recurrence time of $E^{(N)}$:
\begin{equation}\label{def_Ren}
  \ren(t)\doteq\twopartdef{R^{(N)}+t}{0\leq t< u_0,}{t-\sup\{s\geq 0,\arrivaln(s)<\arrivaln(t)\}}{t\geq u_0,}
\end{equation}
where in this particular definition, the supremum of an empty set should be interpreted as zero. Note that $\ren(0)=R^{(N)}$.

Next, let $G$ denote the cumulative distribution function of the i.i.d.\  service times $\{v_j;j\in\Z\}$, and let $\overline G \doteq 1 - G$. We impose the following conditions on  $G$.
\begin{assumption}\label{asm_G}
The service time distribution $G$ has the following properties:
\renewcommand{\theenumi}{\alph{enumi}}
\begin{enumerate}
    \item $G$ has a density $g$ and finite mean which can (and will) be set to $1$.  \label{asm_mean}
    \item \label{hazard}There exists $\endsupzero<\endsup$ such that the hazard rate function
\begin{equation}\label{def_h}
  h(x)\doteq\frac{g(x)}{{\overline G}(x)},\quad\quad x\in\supint,
\end{equation}
where $\endsup\doteq\sup\{x\in[0,\infty):G(x)<1\}$, is either bounded or lower semi-continuous on $(\endsupzero,\endsup)$. \label{asm_h}
\item \label{loc_b}  The density $g$ is bounded on every finite interval of $[0,\infty)$.
\end{enumerate}
\end{assumption}

Note that both Assumptions \ref{asm_G}.\ref{hazard} and \ref{asm_G}.\ref{loc_b} hold  if either $g$ is continuous or $h$ is bounded on $[0,\infty)$.

\subsection{State Representation}
\label{sec_state}

Recall from the introduction that $\nu^{(N)}_\ell(t)$ is a (random) finite measure on $[0,\infty)$ that has a unit delta mass at the age (i.e., amount of time spent in service) of each job that, at time $t$, is in service at a queue of length no less than $\ell$. Since the maximum number of jobs in service at any time is $N$,   $\nun_\ell(t)/N$ takes values in the space $\mathbb{M}_{\leq1}\supint$ of sub-probability measures on $\supint$.  The state of the system at time $t$ will be represented by  ${\nu}^{(N)}(t)\doteq(\nun_\ell(t);\ell\geq1)$.  The scaled state $\nun(t)/N$ takes values in the space
\begin{equation}
  \s \doteq \left\{(\mu_\ell;\ell\geq 1)\in  \mathbb{M}_{\leq1}\supint^\N \; ; \; \langle f, \mu_{\ell}\rangle \geq\langle f, \mu_{\ell+1}\rangle, \forall \ell\geq 1,f\in\C_b\hc, f\geq0 \right\}
\end{equation}
of ordered sequences of  sub-probability measures.  We equip  $\s$ with the metric
\begin{equation}\label{def_dS}
  d_{\s}(\mu,\tilde \mu) \doteq \sup_{\ell\geq 1} \frac{d_P(\mu_\ell,\tilde \mu_\ell)}{\ell},
\end{equation}
where $d_P$ is the Prohorov metric. Thus,  a sequence $\{\mu^{n}\}$ converges to $\mu$  in $\s$ if and only if for every $\ell\geq1$, $ \{\mu_\ell^n\}$ converges weakly to $\mu_\ell$.

Recall that $\f1$ denotes the function that is identically one, and note that
\begin{equation}
\label{nunone}
\nunone_\ell(t)\doteq\langle \f1, \nu_\ell^{(N)}(t) \rangle, \quad  t \geq 0, \ell \geq 1,
\end{equation}
is the number of queues with length at least $\ell$ at time $t$. Moreover, let $\xn(t)$ be the total number of jobs in the system at time $t$ (including those in service and those waiting in queue). Since  $\nunone_\ell(t)-\nunone_{\ell+1}(t)$ is the number of queues with length exactly $\ell$, we have
\begin{equation}
  X^{(N)} (t) = \sum_{ \ell \geq 1} [ \ell (\nunone_\ell(t)-\nunone_{\ell+1}(t))  ]  = \sum_{ \ell \geq 1} \nunone_\ell(t) = \sum_{ \ell \geq 1} \langle \f1, \nun_{\ell}(t)\rangle.
\end{equation}
Finally, for  $t\geq0$  and  $\ell \in \N$, let  $D^{(N)}_\ell(t)$ denote the total number of jobs that completed service  in the interval $[0,t]$ at a queue that had length  $\ell$ just prior to service completion.
All these random elements are assumed to be supported on a common probability space $(\Omega,\cal F, \mathbb{P})$.

\subsection{Hydrodynamic Equations}\label{sec_hydro}

We now introduce the hydrodynamic equations, which will be shown to characterize the ``functional law of large numbers'' or ``fluid'' limit of the state of the network.   The terminology
refers to the fact that we are looking at the limiting dynamics of the empirical measure of an interacting particle system.
\begin{equation}\label{def_poly}
  \mathfrak{P}_d(x,y)\doteq \frac{x^d-y^d}{x-y}= \sum_{m=0}^{d-1} x^m \;y^{d-1-m}.
\end{equation}
When $d=2$, we have the simple form $\mathfrak{P}_2(x,y)=x+y$, and in general,
 for $x,y,\tilde x, \tilde y\leq 1,$
\begin{equation}\label{polyDerivative}
 \mathfrak{P}_d(x,y)\leq d \qquad \mbox{ and } \qquad  \mathfrak{P}_d(x,y)-\mathfrak{P}_d(\tilde x,\tilde y) \leq d^2 \Big((x-\tilde x)+(y-\tilde y) \Big).
\end{equation}

\begin{definition}\textbf{(Hydrodynamic Equations) }\label{Def_fluid}
Given $\lambda>0$ and $\nu(0)\in\s$, $\{\nu(t)=(\nu_\ell(t);\ell\geq1);t\geq0\}$ in $ \mathbb{C}_{\s}\hc$ is said to solve the \textit{hydrodynamic equations} associated with $(\lambda,\nu(0))$ if and only if for every $t\in\hc$,
\begin{equation}\label{Fluid_bound}
        \int_0^t\langle h,\nu_1(s) \rangle ds<\infty,
\end{equation}
and for every  $\ell\geq1$,
\begin{equation}\label{Fluid_Balance}
  \langle\f1,\nu_\ell(t)\rangle - \langle\f1,\nu_\ell(0)\rangle = D_{\ell+1}(t)+\int_0^t\langle\f1,\eta_\ell(s)\rangle ds -D_\ell(t),
\end{equation}
where
\begin{equation}\label{Fluid_D}
    D_\ell(t)\doteq\int_0^t\langle h,\nu_\ell(s) \rangle ds,\quad\quad \forall \ell\geq 1,
  \end{equation}
and for every  $f\in\C_b\hc$,
\begin{align}\label{Fluid_f}
    \langle f, \nu_\ell(t)\rangle = &\langle f(\cdot+t)\frac{\overline G(\cdot+t)}{\overline G(\cdot)},\nu_\ell(0)\rangle +\int_{[0,t]} f(t-s)\overline G(t-s)dD_{\ell+1}(s)\\
    & + \int_0^t \langle f(\cdot+t-s)\frac{\overline G(\cdot+t-s)}{\overline G(\cdot)},\eta_\ell(s)\rangle ds,\notag
\end{align}
with
\begin{equation}\label{Fluid_R}
  \eta_\ell(t)\doteq\twopartdef{\lambda\left(1-\langle\f1,\nu_1(t)\rangle^d\right)\delta_0}{\ell=1,} {\lambda\mathfrak{P}_d\Big(\langle \f1,\nu_{\ell-1}(t)\rangle,\langle \f1,\nu_\ell(t)\rangle\Big)(\nu_{\ell-1}(t)-\nu_\ell(t))}{\ell\geq 2.}
\end{equation}
\end{definition}

Given a solution $\nu$ to the hydrodynamic equations, we define
\begin{equation}
\label{Fluid_S}
S_\ell(t)\doteq \langle \f1,\nu_\ell(t) \rangle, \quad  t\geq0, \ell\geq1.
\end{equation}

\begin{remark}
The bound \eqref{Fluid_bound} implies that for every $\ell\geq 1$, the process $D_\ell$ is well-defined.
\end{remark}

\begin{remark}
Definition \ref{Def_fluid} of the hydrodynamic equations and the corresponding uniqueness result in Theorem \ref{thm_uniqueness} can be generalized to time-varying rates by simply replacing the constant arrival $\lambda$ everywhere with a non-negative locally integrable function $\lambda (\cdot)$.
\end{remark}

We now state  our first main result, which is proved in  Section \ref{sec_fluid_unique}.

\begin{theorem}\label{thm_uniqueness}
Suppose Assumptions \ref{asm_E} and \ref{asm_G}.\ref{asm_mean} hold. Then for every $\lambda>0$ and $\nu(0)\in\s$, the hydrodynamic equations associated with $(\lambda,\nu(0))$ have at most one solution.
\end{theorem}

We now provide some   intuition into the form of the hydrodynamic equations. Given $a(s)$, the age  of a job in service at time $s$, the mean conditional probability that this job will complete service in the time interval $(s,s+ds)$ is roughly $h(a(s)) ds$.
 Summing over the ages of all jobs in service at queues of length no less than $\ell$, we see that the conditional mean departure rate from such queues at time $s$ is $\langle h, \nu_{\ell} (s) \rangle$.  In the large $N$ limit, the scaled departure process coincides with its mean, thus giving rise to the  equality in \eqref{Fluid_D}. Next, to understand the mass balance equation  \eqref{Fluid_Balance}, fix $ \ell \geq 1$ and  note that in analogy with \eqref{nunone}, $\langle \f1, \nu_\ell(t)\rangle$ represents the limit fraction of queues of length no less than $\ell$ at time $t$. Over the interval $[0,t]$, this quantity decreases due to departures from queues of length precisely $\ell$, which is given by $D_{\ell}(t) - D_{\ell+1}(t)$, and increases due to exogeneous arrivals to queues of length $\ell - 1$. To quantify the latter, note that   $\lambda$  is the scaled mean arrival rate of jobs to the network and the probability that an arriving job is routed to a queue of length $\ell-1$ at time $s$ is approximately equal to $(\langle \f1, {\nu}_{\ell-1}(s) \rangle)^d - (\langle \f1, {\nu}_{\ell}(s) \rangle)^d$, which is equal to $\mathfrak{P}_d(\langle \f1,\nu_{\ell-1}(s)\rangle,\langle \f1,\nu_\ell(s)\rangle) (\langle \f1, \nu_{\ell-1}(s) \rangle -\langle \f1, \nu_\ell(s)\rangle)$, with the convention $\langle \f1, {\nu}_0(s) \rangle\doteq1$. Thus,  with $\eta_\ell$ defined by \eqref{Fluid_R}, $\langle \f1, \eta_\ell (s)\rangle $ represents the scaled arrival rate at time $s$ of jobs to queues of length $\ell-1$,  and  $\int_0^t \langle \f1, \eta_\ell(s) \rangle ds$ represents the total exogeneous arrivals to such queues over the interval $[0,t]$.    These observations, when combined, justify the form of \eqref{Fluid_Balance}.

The equation \eqref{Fluid_f} is a more involved mass balance equation, whose right-hand side consists of three terms that contribute to the measure $\nu_\ell (t)$.  The first term on the right-hand side accounts for  jobs already in service at time $0$.  Any such  job, conditioned on having  initial age $a(0)$,  would still be in service at time $t$, with age $a(0) + t$, with probability $\overline{G} (a(0) + t)/\overline{G}(a(0))$.   The second term represents the contribution to $\nu_\ell(t)$  due to jobs that entered  service at some time $s \in (0,t]$ at a queue of length no less than $\ell$ at time  $s$ and that are still in service at time $t$. Such service entries occur due to departures of jobs at time $s$ from a queue no less than $\ell+1$ prior to departure, which would happen  at rate $dD_{\ell+1}(s)$. Further,  the job entering service would have age  $0$ at time $s$ and  so would still be in  service at time $t$ (with age $t-s$) with  probability $\overline{G}(t-s)$.   Finally, the last term captures the contribution due to jobs that were in service at a queue of length $\ell - 1$ at some time $s \in [0,t]$ when its length increased by one due to the routing of a job  to that queue.  If $\ell > 1,$ and $a(s)$ was the age of  the job in service at that queue at time $s$, the  job would still be in service at time $t$ only if its service time were greater than $a(s)  + t-s$ (given that it was clearly greater than $a(s)$), which has probability $\overline{G} (a(s) + t- s)/\overline{G}(a(s))$. Now, the (limit) distribution of ages in service at queues of length $\ell-1$ at time $s$ is $\nu_{\ell-1} (s) - \nu_{\ell}(s)$.   When multiplied by  the (limit) rate at which  jobs are routed to a random queue of length $\ell-1$, which  is $\lambda \mathfrak{P}_d(\langle \f1,\nu_{\ell-1}(s)\rangle,\langle \f1,\nu_\ell(s)\rangle)$,  yields $\eta_{\ell}(s)$. The case $\ell = 1$ can be argued similarly. This explains the form of the third term on the right-hand side of \eqref{Fluid_f}.  The above discussion also suggests why the limit of a more general class of routing algorithms could be characterized similarly, but with a suitably modified definition of $\eta_\ell$.

\subsection{Convergence Result}\label{subs-conv}

For $H = E, D, \nu_\ell, \nu, S_\ell$,  we define the scaled version of $H^{(N)}$ as follows:
\begin{equation}
\label{scaling}
\overline{H}^{(N)} (t) = \frac{H^{(N)}(t)}{N}, \quad N \in \N, t \geq 0.
\end{equation}

The following condition is imposed on the initial state of the network.

\renewcommand{\theenumi}{\alph{enumi}}
\begin{assumption}\label{asm_initial}
The sequence of initial conditions satisfies the following.
  \begin{enumerate}
    \item For every $N\in\N$, $\xn(0)=\sum_{\ell\geq1}\langle\f1,\nun_\ell(0)\rangle<\infty$ almost surely, $E^{(N)}$ and  the random queue choices in the load balancing algorithm are independent of $\nun(0)$, and for each job $j$ that is in service at time $0$,  its service time $v_j$ is conditionally independent of $\nun(0)$ given its initial age $a_j(0)$; see \eqref{vInitial} for further details.  \label{asm_initial_ind}
    \item there exists $\nu(0)=(\nu_\ell(0);\ell\geq1)\in \s$ such that $\nunbar(0)\to\nu(0)$ in $\s$, $\mathbb{P}$-almost surely, as $N\to\infty.$\label{asm_initial_nu}
    \item \label{asm_initial_X} $\limsup_N\Eptil{ \xnbar(0) }<\infty$, and $\xnbar(0)\to X(0)$ as $N\to\infty$, where
    $X(0)\doteq\sum_{\ell\geq1} \langle\f1,\nu_\ell(0)\rangle$.
  \end{enumerate}
\end{assumption}
\renewcommand{\theenumi}{\arabic{enumi}}

We now state some immediate consequences of Assumptions \ref{asm_E} and \ref{asm_initial}.\ref{asm_initial_X}.

\begin{lemma}\label{asm_initial_remark}
    Suppose Assumption \ref{asm_E} holds. Then, as $N\to\infty$, $\arrivalnbar\to\lambda \id$ in $\D$, $\mathbb{P}$-almost surely, Moreover, for all $t\geq0$, $\Eptil{\arrivalnbar(t)}\to  \lambda t$ as $N\to\infty$, and hence,
    \begin{equation}\label{asm_initial_arrivalmoment}
      \limsup_{N\to\infty} \Ept{\arrivalnbar(t)}<\infty.
    \end{equation}
In addition, if Assumption \ref{asm_initial}.\ref{asm_initial_X} holds, then for every $t\geq0$,
  \begin{equation}\label{asm_initial_Xarrivalmoment}
    \limsup_{N\to\infty}\Ept{\xnbar(0) + \arrivalnbar(t)}<\infty.
  \end{equation}
\end{lemma}

\begin{proof}
The  almost sure convergence of $\arrivalnbar$ to $\lambda\id$ in $\D$ follows from Assumption \ref{asm_E}
and the functional law of large numbers for renewal processes (e.g., see \cite[Theorem 5.10]{CheYao2001}).
  Also, for $t\geq0$, by the elementary renewal theorem (e.g., see \cite[Proposition V.1.4]{Asm03}),     $\lim_{N\to\infty}\Ept{\arrivalnbar(t)}$ $=\lim_{N\to\infty}\widetilde E(N t)/N$ $= \lambda t,$  where $\tilde{E}$ is the delayed renewal process of Assumption \ref{asm_E}.  This implies  \eqref{asm_initial_arrivalmoment}, which along with
Assumption \ref{asm_initial}.\ref{asm_initial_X}, implies  \eqref{asm_initial_Xarrivalmoment}.
\end{proof}

We now state the second main result, whose proof is given in Section \ref{sec_limitCharacterization}.

\begin{theorem}\label{thm_convergence}
Suppose Assumptions \ref{asm_E}-\ref{asm_initial} hold. Then there exists a unique solution $\nu \in \C_{\s}$ to the hydrodynamic equations associated  with $(\lambda,\nu(0))$, and the sequence $\{\nunbar\}$ converges in distribution to  $\nu$.
\end{theorem}

As a corollary, we establish a ``propagation of chaos'' result, whose proof is also deferred to Section  \ref{sec_limitCharacterization}.  Let $\xni(t)$  be the length of the $i$th queue at time $t$, and if the queue is initially non-empty,  let $\ageni(0)$ be the initial age of the job receiving service at the $i^{\text{th}}$ queue.

\begin{corollary}\label{cor_PropofChaos}
Suppose Assumptions \ref{asm_E}-\ref{asm_initial} hold, and the initial conditions are exchangeable, that is, for every $N$ and any permutation $\pi$ of the queue indices $\{1,...,N\}$, the random vector
\[\left( \xnX{\pi(i)}(0), \agenX{\pi(i)}(0) \indicil{\xnX{\pi(i)}(0)>0};i=1,...,N \right),\]
has the same distribution.   Let $\nu$ be the solution to the hydrodynamic equations associated with
$(\lambda, \nu(0))$ and let $\{S_\ell, \ell \geq 1\}$ be as defined in \eqref{Fluid_S}.
Then,  for every $\ell\geq 1$ and $t\geq0$,
\begin{equation}\label{propOfChaos_marg}
  \lim_{N\to\infty} \Prob{\xnX{1}(t)\geq \ell} = S_\ell(t),
\end{equation}
and for ever fixed $k\geq 0$ and $\ell_1,...,\ell_k\in\N$,
\begin{equation}\label{propOfChaos}
  \lim_{N\to\infty} \Prob{\xnX{1}(t)\geq \ell_1,...,\xnX{k}(t)\geq \ell_k} = \prod_{m=1}^k S_{\ell_m}(t).
\end{equation}
\end{corollary}


\section{Analysis of the Hydrodynamic Equations}\label{sec_fluid}
In Section \ref{sec_fluid_unique} we prove Theorem \ref{thm_uniqueness}.  In Section \ref{sec_fluid_dynamics}, we  obtain a dynamical characterization of the hydrodynamic equations in terms of a measure-valued PDE,  which is  used in Section \ref{sec_limitCharacterization}  to prove Theorem \ref{thm_convergence}.

\subsection{Proof of Uniqueness of the Solution to the Hydrodynamic Equations}\label{sec_fluid_unique}

\begin{proof}[Proof of Theorem \ref{thm_uniqueness}]
Fix $\nu(0)\in\s$, $\lambda > 0$, and let $\nu$ and $\tilde\nu$ both be solutions to the hydrodynamic equations  associated with $(\lambda,\nu(0))$, and let $\tilde{D}, \tilde{\eta}, \tilde{S}$ be defined as in \eqref{Fluid_D}--\eqref{Fluid_S}, but with $\nu$ replaced by $\tilde{\nu}$. For $\ell\geq1$, define $\Delta H_\ell\doteq H_\ell -\tilde H_\ell$ for $H=\nu,D,\eta,S$.  Consider the parameterized family of continuous bounded functions
\[
\mathbb{F}\doteq\left\{\vartheta^r\doteq\frac{\overline G(\cdot+r)}{\overline G(\cdot)};r\geq0\right\} \subset \C_b\supint
 \cap \mathbb{AC}\hc,
\]
where the last inclusion holds by Assumption \ref{asm_G}.
Note that $\f1=\vartheta^0\in\mathbb{F}$, and for every $\ell\geq1$ and $t\geq0$, define
\begin{equation}
\label{def-vell}
  V_\ell(t) \doteq \sup_{f\in \mathbb{F}} |\langle f, \Delta\nu_\ell(t) \rangle|.
\end{equation}

By \eqref{Fluid_R} with $\ell=1$, \eqref{Fluid_S} and \eqref{def_poly}, for $s\geq0$ and  every $f\in\C_b\supint$ we have
\begin{align}\label{uniq_proof_generalf1}
  \langle f,\Delta \eta_1(s)\rangle =&\lambda f(0)\left((\tilde S_1(s))^d -( S_1(s))^d\right) \notag\\
  =& -\lambda f(0)\mathfrak{P}_d\Big(\tilde S_1(s),S_1(s)\Big) \langle \f1,\Delta\nu_1(s) \rangle.
\end{align}
Therefore, for every $f\in \mathbb{F}$, since $\mathfrak{P}_d(x,y)\leq d$ and $\|f\|_\infty\leq 1$,
\begin{equation}\label{uniqueproof_nubound_part2}
  \left|\langle f,\Delta \eta_1(s)\rangle\right| \leq  \lambda d | \langle \f1,\Delta\nu_1(s)\rangle| \leq \lambda d V_1(s),\quad\quad\quad f\in\mathbb{F}.
\end{equation}
Next, for $\ell\geq2,$ $s\geq0$ and $f\in\C_b\supint$, again invoking \eqref{Fluid_R}, \eqref{Fluid_S} and \eqref{def_poly}, we have
\begin{align}\label{uniq_proof_generalfl}
  \langle f,\Delta \eta_\ell(s)\rangle =&\lambda \mathfrak{P}_d\big(S_{\ell-1}(s),S_\ell(s)\big) \langle f,\Delta\nu_{\ell-1}(s)-\Delta\nu_\ell(s)\rangle\\
   &+ \lambda \Big(\mathfrak{P}_d\big(S_{\ell-1}(s),S_\ell(s)\big) -\mathfrak{P}_d\big(\tilde S_{\ell-1}(s),\tilde S_\ell(s)\big)\Big) \langle f,\tilde \nu_{\ell-1}(s)-\tilde \nu_\ell(s)\rangle.\notag
\end{align}
Hence, for  $f\in \mathbb{F}$, given $\|f\|_\infty\leq 1$, $\nu_{\ell-1} \geq \nu_\ell$, $\tilde \nu_{\ell-1} \geq \tilde \nu_\ell$, $\nu_\ell, \tilde{\nu}_\ell \in {\mathcal M}_{\leq 1} \supint$, $\mathfrak{P}_d(x,y)\leq d$  and inequality \eqref{polyDerivative},  we have
\begin{equation}\label{uniqueproof_nubound_part3}
  \left|\langle f,\Delta \eta_\ell(s)\rangle\right| \leq \lambda(d+d^2)\big( V_{\ell-1}(s)+V_{\ell}(s) \big),\quad\quad f\in\mathbb{F}.
\end{equation}

Now, for   $f\in \mathbb{C}_b\hc\cap \mathbb{AC}\hc$ applying integration by parts to \eqref{Fluid_f}, we obtain
\begin{equation}
\label{Fluid_f_alt}
\begin{array}{rcl}
\displaystyle    \langle f, \nu_\ell(t)\rangle  & = & \displaystyle \langle f(\cdot+t)\frac{\overline G(\cdot+t)}{\overline G(\cdot)},\nu_\ell(0)\rangle +f(0)D_{\ell+1}(t)+\int_{[0,t]} (f\overline G)^\prime(t-s)D_{\ell+1}(s)ds\\
    & & \displaystyle + \int_0^t \langle f(\cdot+t-s)\frac{\overline G(\cdot+t-s)}{\overline G(\cdot)},\eta_\ell(s)\rangle ds.
\end{array}
\end{equation}
Also,  for every $t\geq s\geq0$, $r \geq 0$,  we have $\vartheta^r(0)=\overline G(r)$,
$(\vartheta^r \overline{G})^\prime (t-s) = -  g(t-s+r),$ and
\[\vartheta^r(x+t-s)\frac{\overline G(x+t-s)}{\overline G(x)}= \frac{\overline G(x+t-s+r)}{\overline G(x)} = \vartheta^{t-s+r}(x),\quad\quad x\in\supint. \]
Thus, substituting $f=\vartheta^r$ in  \eqref{Fluid_f_alt}, both as  is and when $\nu_\ell,D_\ell,\eta_\ell$ are replaced by $\tilde\nu_\ell, \tilde D_\ell$, $\tilde\eta_\ell$, respectively, and recalling $\Delta \nu_\ell (0) = 0$, we have
\begin{equation}\label{uniquenessproof_deltanu}
    \langle \vartheta^r, \Delta\nu_\ell(t)\rangle = \overline G(r)\Delta D_{\ell+1}(t)-\int_0^tg(t-s+r)\Delta D_{\ell+1}(s)ds + \int_0^t \langle \vartheta^{t-s+r},\Delta \eta_\ell(s)\rangle ds.
\end{equation}
Since $\vartheta^0=\f1$, equation \eqref{uniquenessproof_deltanu} for $r=0$ gives
\begin{equation}\label{uniquenessproof_deltanu1}
      \langle \f1, \Delta\nu_\ell(t)\rangle =  \Delta D_{\ell+1}(t)-\int_0^tg(t-s)\Delta D_{\ell+1}(s)ds + \int_0^t \langle \vartheta^{t-s},\Delta \eta_\ell(s)\rangle ds.
\end{equation}
Since $\{\nu_\ell\}_{\ell\in\N}$ and $\{\tilde\nu_\ell\}_{\ell\in\N}$ satisfy \eqref{Fluid_Balance} and its analog,  and that $\Delta\nu_\ell(0)=0$, we  have
\begin{equation}\label{temp_unique1}
  \Delta D_{\ell+1}(t)=   \langle\f1,\Delta\nu_\ell(t)\rangle +\Delta D_\ell(t) - \int_0^t\langle\f1,\Delta \eta_\ell(s)\rangle ds.
\end{equation}
Combining \eqref{temp_unique1} and  \eqref{uniquenessproof_deltanu1},  it follows that for $\ell\geq 2$, $\Delta D_\ell$ satisfies the renewal equation
\begin{equation*}
  \Delta D_\ell(t) = g*\Delta D_\ell(t) + F_\ell(t),
\end{equation*}
with
\begin{equation*}
  F_\ell(t)\doteq \int_0^t\langle\f1,\Delta \eta_\ell(s)\rangle ds + g*\langle\f1,\Delta \nu_\ell\rangle(t) - (g* \int_0^\cdot\langle\f1,\Delta \eta_\ell(s)\rangle ds)(t) - \int_0^t \langle \vartheta^{t-s},\Delta \eta_\ell(s)\rangle ds.
\end{equation*}
Since $\Delta\nu_\ell$ is the difference of two measures in $\mathbb{D}_{\mathbb{M}_F\supint}\hc$,
$\langle\f1,\Delta\nu_\ell(\cdot)\rangle$ and $\langle f,\Delta \eta_\ell(\cdot)\rangle,$ $f\in\C_b\hc$,  are also locally  integrable, and hence $F_\ell$ is uniformly bounded on finite intervals (i.e., $\|F\|_t<\infty$ for all $t\geq 0$). Moreover, \eqref{Fluid_bound} ensures that $\Delta D_\ell$ is also bounded on finite intervals.  Therefore, by \cite[Theorem 2.4, Chapter V]{Asm03},
\begin{equation}
    \label{new_dl} \Delta D_\ell(t)= F_\ell(t) + u*F_\ell(t),
\end{equation}
where $u$ is the renewal density of $G$ on  $\supint$.  Note that since $G$ has density $g$,  by \cite[Proposition 2.7, Chapter V]{Asm03}, $u$ exists and satisfies the equation $u=u*g +g$. Moreover, since $g$ is   locally bounded due to Assumption \ref{asm_G}.\ref{loc_b},  $u$ is bounded on every finite interval  of $[0,L)$ by another application of \cite[Theorem 2.4, Chapter V]{Asm03}.  Substituting the definition of $F_\ell$ into equation \eqref{new_dl} and using the relation $u*g + g=u$, we have
\begin{align}\label{uniq_Dbound_eq}
  \Delta D_\ell(t)=  \int_0^t\langle\f1,\Delta \eta_\ell(s)\rangle ds+ u*\langle\f1,\Delta \nu_\ell\rangle(t) -  \int_0^t \langle \vartheta^{t-s},\Delta \eta_\ell(s)\rangle ds -  (u* \int_0^\cdot \langle \vartheta^{\cdot-s},\Delta \eta_\ell(s)\rangle ds)(t).
\end{align}

Next, we bound each term on the right-hand side of \eqref{uniq_Dbound_eq}.  Fix $T\geq 0$.  Then  \eqref{uniqueproof_nubound_part3} implies
\begin{equation}\label{uniq_Dbound_1}
  \left|\int_0^t\langle\f1,\Delta \eta_\ell(s)\rangle ds\right|\leq \lambda(d+d^2)\int_0^t \left( V_{\ell-1}(s) + V_\ell(s) \right) ds,\quad t\leq T.
\end{equation}
Moreover, recalling the notation $\|f\|_T = \sup_{t\in[0,T]}|f(t)|$, we also have
\begin{equation}\label{uniq_Dbound_2}
    \left|u*\langle\f1,\Delta \nu_\ell\rangle(t)\right|
    \leq \int_0^t u(t-s) \left|\langle\f1,\Delta \nu_\ell(s)\rangle \right| ds
    \leq \|u\|_T\int_0^t V_\ell(s)ds,\quad  t\leq T.
\end{equation}
Furthermore, for the function $\zeta_\ell(t)\doteq\int_0^t \langle \vartheta^{t-s},\Delta \eta_\ell(s)\rangle ds$, the bound  \eqref{uniqueproof_nubound_part3} with $f=\vartheta^{t-s}$ implies
\begin{equation}\label{uniq_Dbound_3}
  |\zeta_\ell(s)|\leq \lambda(d+d^2)\int_0^s \left( V_{\ell-1}(v) + V_\ell(v) \right) dv, \quad\quad s\geq 0,
\end{equation}
and hence, applying Tonelli's theorem in the third inequality below, we obtain
\begin{align}\label{uniq_Dbound_4}
  |u*\zeta_\ell(t)|  \leq  & \int_0^tu(t-s)|\zeta_\ell(s)|ds \\
                \leq  & \lambda(d+d^2)\int_0^t \int_0^s u(t-s) \left( V_{\ell-1}(v) + V_\ell(v) \right) dv\; ds\notag\\
                \leq  & \lambda(d+d^2)\int_0^t  \left( V_{\ell-1}(v) + V_\ell(v) \right) \left(\int_v^tu(t-s)ds\right)dv \notag\\
                \leq  & \lambda(d+d^2)U(T) \int_0^t  \left( V_{\ell-1}(s) + V_\ell(s) \right) ds,\quad\quad\forall t\leq T,\notag
\end{align}
where $U(\cdot)=1+\int_0^\cdot u(s)ds$ is the renewal function associated with $G$. From  equation \eqref{uniq_Dbound_eq} and the bounds \eqref{uniq_Dbound_1}-\eqref{uniq_Dbound_4}, for $\ell\geq 2$ we then have
\begin{equation}
\label{uniqueproof_Dbound_final}
  \|\Delta D_\ell\|_t \leq C(T)\int_0^t  \left( V_{\ell-1}(s) + V_\ell(s) \right) ds,\quad\forall t\leq T,
\end{equation}
with $C(T)\doteq\|u\|_T + \lambda(d+d^2) (2+U(T))<\infty$. Finally, incorporating \eqref{uniqueproof_nubound_part2}, \eqref{uniqueproof_nubound_part3} and \eqref{uniqueproof_Dbound_final}, with $\ell$ replaced by $\ell+1$, into  \eqref{uniquenessproof_deltanu}, we have
\begin{equation*}
  |\langle f,\Delta\nu_1(t)\rangle| \leq 3 C(T)\int_0^t (V_1(s)+V_2(s)) ds,\quad\quad\quad \forall f\in\mathbb{F},
\end{equation*}
and for every $\ell \geq 2$,
\begin{equation*}
  |\langle f,\Delta\nu_\ell(t)\rangle| \leq 3 C(T) \int_0^t  (V_{\ell-1}(s)+V_\ell(s)+V_{\ell+1}(s)) ds,\quad\quad\quad \forall f\in\mathbb{F}.
\end{equation*}
Taking the supremum over  $f\in \mathbb{F}$ in the last two  inequalities, for all $t\leq T$ we obtain
\begin{equation}\label{uniqueproof_nubound_final}
  V_\ell(t)\leq\twopartdef{3 C(T)\int_0^t  (V_1(s)+V_2(s)) ds,}{\ell=1,}{3 C(T) \int_0^t  (V_{\ell-1}(s)+V_\ell(s)+V_{\ell+1}(s)) ds,}{\ell\geq 2.}
\end{equation}

Define $V(t)\doteq\sum_{\ell=1}^\infty 2^{-\ell}V_\ell(t)$. Then  \eqref{uniqueproof_nubound_final}   implies
\begin{align}
  V(t)\leq 12C(T)\int_0^t V(s) ds.
\end{align}
Also, since $|\langle f,\nu_\ell(t)\rangle|\leq 1$ for $f\in\mathbb{F},$ \eqref{def-vell} implies $V_\ell(t)\leq 2$ for all  $\ell\geq 1$, and hence $V(t)\leq 2$.
Since $V(0)=0$, an application of Gronwall's inequality then shows that $V(t)=0$ for all $t\geq 0$, and hence,  $V_\ell(t)=0$ for all $t\geq 0$ and $\ell \geq 1$. In particular, this shows that
\begin{equation}\label{temp_unique2}
  \langle \f1,\Delta\nu_k \rangle \equiv 0,\quad\quad\quad k \geq1.
\end{equation}
Moreover, by  \eqref{uniqueproof_Dbound_final}, $\Delta D_{\ell} \equiv 0$ for all $\ell \geq 2$.
Taking the difference between equation \eqref{Fluid_f_alt}, and the same equation, but  with $\nu_\ell$, $D_{\ell+1}$ and $\eta_{\ell}$ replaced by  $\tilde{\nu}_\ell$, $\tilde{D}_{\ell+1}$ and $\tilde{\eta}_{\ell}$, respectively, and using the identities  $\Delta D_{\ell+1} \equiv 0$ and $\Delta\nu_\ell(0)=0$, we see that for $\ell \geq 1$ and
 $f\in \mathbb{C}_b\hc\cap \mathbb{AC}_{\text{loc}}\hc$,
\begin{equation}\label{uniq_newnubound}
  \langle f, \Delta\nu_\ell(t)\rangle =\int_0^t \langle f(\cdot+t-s)\vartheta^{t-s}(\cdot),\Delta \eta_\ell(s)\rangle ds.
\end{equation}

To finish the proof, we  use induction on $\ell$ to show that $\Delta\nu_\ell\equiv0$ for $\ell\geq1$. Let $\AC^\prime\supint \doteq \{f \in \AC \supint: ||f||_\infty \leq 1\}$.  Since $G$ has a density by Assumption \ref{asm_G}, $f(\cdot+t-s)\vartheta^{t-s}(\cdot)\in \AC^\prime \supint$ for all $f\in \AC^\prime\supint$ and $t,s\geq0$. For $\ell=1$, \eqref{uniq_proof_generalf1} and  \eqref{temp_unique2} with $k=1$ imply  $\langle f(\cdot+t-s)\vartheta^{t-s},\Delta \eta_1(s) \rangle =0$, and  therefore, $\Delta\nu_1\equiv0$ by  \eqref{uniq_newnubound}. Furthermore, if $\Delta \nu_{\ell-1}\equiv0$ for some $\ell\geq2$, it follows from \eqref{uniq_proof_generalfl}, \eqref{polyDerivative} and \eqref{temp_unique2}, both with $k=\ell-1$ and $k=\ell$, that
\begin{align*}
    \left|\langle f(\cdot+t-s)\vartheta^{t-s}(\cdot),\Delta \eta_\ell(s)\rangle\right| \leq & \lambda \mathfrak{P}_d\big(\langle \f1,\nu_{\ell-1}(s)\rangle,\langle \f1,\nu_\ell(s)\rangle\big) \langle f(\cdot+t-s)\vartheta^{t-s}(\cdot),\Delta\nu_\ell(s)\rangle\\
    \leq &\lambda  d \sup_{f \in \AC^\prime \supint} \left|\langle f, \Delta\nu_\ell (s)\rangle \right|.
\end{align*}
Together with \eqref{uniq_newnubound}  this implies
\begin{equation}
   \sup_{f \in  \AC^\prime \supint} \left|\langle f, \Delta\nu_\ell(t) \rangle \right|\leq d\int_0^t\lambda
    \sup_{f \in \AC^\prime \supint}\left|\langle f, \Delta\nu_\ell (s)\rangle \right| ds,\quad\quad \forall t \geq0.
\end{equation}
Since $\Delta\nu_\ell(0)=0,$ Gronwall's inequality shows $\left|\langle f, \Delta\nu_\ell (t)\rangle \right| = 0$ for $f \in \AC^\prime \supint$ and hence, by linearity and a density argument, for  $f \in \C_b \supint$. This proves $\Delta\nu_\ell(t)=0$ for all $t\geq0$.
\end{proof}

\subsection{A Measure-valued PDE associated with the hydrodynamic limit}\label{sec_fluid_dynamics}

The following is the main result of this section.

\begin{proposition}\label{prop_SolutionForm}
Given $\nu(0)=(\nu_\ell(0);\ell\geq1)\in\s$ and $\lambda>0$,
suppose  $\nu=(\nu_\ell)_{\ell\geq1}\in\mathbb{D}_{\s}\hc$ satisfies the following: \eqref{Fluid_bound} holds and for every $\ell\geq 1$ and $t\geq0$, \eqref{Fluid_Balance} holds, with $D_\ell$ and $\eta_\ell$ defined as in \eqref{Fluid_D}  and \eqref{Fluid_R}, respectively, and for $\varphi\in\phiset,$
\begin{align}\label{GAE_dynamic}
  \langle \varphi(\cdot, t),\nu_\ell(t)\rangle  = & \langle \varphi(\cdot, 0),\nu_\ell(0)\rangle + \int_0^t \langle \varphi_x(\cdot, s)+\varphi_s(\cdot, s)-\varphi(\cdot, s)h(\cdot),\nu_\ell(s)\rangle ds \notag\\
  &+ \int_0^t\varphi(0,s)dD_{\ell+1}(s)+ \int_0^t \langle \varphi(\cdot, s),\eta_\ell(s)\rangle ds.
\end{align}
Then, $\nu$ is a solution to the hydrodynamic equations associated with $(\lambda,\nu(0)).$
\end{proposition}

The proof of Proposition \ref{prop_SolutionForm}, relies on the following lemma. Denote by $\tM$  the space of Radon measures  on $\R^2$ whose supports lie in $\supint\times\R_+$, and denote the integral with respect to any Radon measure $\Theta$ on $\R^2$ by
\[\Theta(\varphi)=\int_{\R^2}\varphi(x,s)\Theta(dx,ds),  \quad  \varphi \in \C_c ([0,\endsup) \times \R_+). \]

\begin{lemma}\label{lem_aae}
Given a Radon measure $\Theta\in\tM$, suppose $ \mu=\{\mu(t);t\geq0\}$ in $ \mathbb{D}_{\mathbb{M}_F\supint}\hc$ satisfies $ \int_0^t\langle h,\mu(s) \rangle ds <\infty,$ and
and for every  $\varphi\in\phiset$,
\begin{equation}\label{GAE_dynamic_abstract}
  -  \int_0^\infty \langle \varphi_x(\cdot, s)+\varphi_s(\cdot, s)-\varphi(\cdot, s)h(\cdot),\mu(s)\rangle ds = \Theta(\varphi).
\end{equation}
Then for every $f\in\mathbb{C}_c\supint$ and $t\geq0$,
\begin{equation}\label{gae_solution}
    \langle f,\mu(t)\rangle =\int_{\supint \times [0,t]} f(x+t-s)\frac{\overline G(x+t-s)}{\overline G(x)}\Theta(dx,ds).
\end{equation}
\end{lemma}

Equation \eqref{GAE_dynamic_abstract} is called the  \textit{abstract age equation} and was studied extensively in \cite[Section 4.3]{KasRam11}.  Lemma \ref{lem_aae} essentially  follows from Corollary 4.17 and equations (4.24),  (4.45), (4.46) and (4.55)  in \cite{KasRam11}; however, we provide a proof here for completeness.

\begin{proof}[Proof of Lemma \ref{lem_aae}]
Define the measure $h\mu$ on $\R^2$ as
\[(h\mu)(\varphi)=\int_0^\infty \langle \varphi(\cdot,s)h(\cdot),\mu(s)\rangle ds,\quad\quad\varphi\in\C_c(\R_+^2).\]
Using this definition, the equation \eqref{GAE_dynamic_abstract} can be written as
\begin{equation}\label{temp_aae}
  -  \int_0^\infty \langle \varphi_x(\cdot, s)+\varphi_s(\cdot, s),\mu(s)\rangle ds= -(h\mu)(\varphi) + \Theta(\varphi).
\end{equation}
With a change of variable, \eqref{temp_aae} can be transformed into a simpler linear measure-valued transport equation.   Define $\psi_h(x,t)\doteq\exp(r_h(x,t))$, where
\begin{equation*}
  r_h(x,t) = \threepartdef{\displaystyle -\int_{x-t}^xh(u)du}{0\leq t\leq x<  L,}{\displaystyle -\int_{0}^xh(u)du}{  0\leq x\leq t, x<L,}{0}{otherwise,}
\end{equation*}
and $\tilde\Theta(\varphi)\doteq\Theta(\psi_h^{-1}\varphi)$. Then, the measure-valued process $\{\tilde \mu_t;t\geq0\}$ defined as
\begin{equation}\label{temp_sae0}
\langle f, \tilde \mu_t \rangle \doteq \langle f \psi_{h}^{-1} (\cdot,t), \mu_t \rangle \end{equation}
can be easily shown to satisfy (see \cite[Proposition 4.16]{KasRam11})
\begin{equation}\label{temp_sae}
 -  \int_0^\infty \langle \varphi_x(\cdot, s)+\varphi_s(\cdot, s),\tilde\mu(s)\rangle ds= \tilde\Theta(\varphi).
\end{equation}
The equation \eqref{temp_sae} is essentially a measure-valued linear inhomogeneous transport equation, which has a unique solution that is given explicitly by (see \cite[Lemma 4.13]{KasRam11})
\begin{equation}\label{temp_sae2}
  \langle f, \tilde\mu_t\rangle = \tilde\Theta(\Lambda_f^t),
\end{equation}
where $\Lambda_f^t(x,s)=f(x+t-s)\mathbb{1}_{[0,t]}(s)$. Combining \eqref{temp_sae0} and \eqref{temp_sae2}, it follows that the equation \eqref{GAE_dynamic_abstract} has the unique solution given by (see \cite[Corollary 4.17]{KasRam11})
\begin{equation}
 \label{gae_temp}
  \langle f,\mu(t)\rangle =\Theta\big(  \psi_{h}^{-1} \Lambda^t_{f(\cdot)\psi_h(\cdot,t)} \big).
\end{equation}
Note that $\psi_h$ can be simplified as
\begin{equation*}
  \psi_h(x,t) = \threepartdef{\displaystyle \frac{\overline G(x)}{\overline G(x-t)}}{0\leq t\leq x<  L,}{\overline G(x)}{0\leq x\leq t<\infty,}{0}{otherwise.}
\end{equation*}
and hence,
\[  (\psi_{h})^{-1} (x,s) \Lambda^t_{f(\cdot)\psi_h(\cdot,t)}(x,s)=f(x+t-s) \frac{\overline{G}(x+t-s)}{\overline{G}(x)}\f1_{[0,t]}(s).
\]
Equation \eqref{gae_solution} then follows on substituting $(\psi_{h})^{-1} \Lambda^t_{f(\cdot)\psi_h(\cdot,t)}$ from the equation above into  \eqref{gae_temp}.
\end{proof}

\begin{proof}[Proof of Proposition \ref{prop_SolutionForm}]
Clearly, we only need to show that \eqref{Fluid_f} holds for all $t\geq0,$ $\ell\geq1$ and $f\in\mathbb{C}_b\supint.$ Fix $\ell\geq 1$, and consider the linear functional $\xi_\ell$ on $\mathbb{C}_c(  \R^2)$ defined by
\begin{equation}\label{CAE_proof_xi}
    \xi_\ell(\varphi)\doteq \langle \varphi(\cdot,0), \nu_\ell(0) \rangle + \int_{\hc} \varphi(0,s)dD_{\ell+1}(s) + \int_0^\infty \langle\varphi(\cdot,s),\eta_\ell(s)\rangle ds.
\end{equation}
By  \eqref{Fluid_R} and \eqref{polyDerivative},  for $m\in\supint$, $T\in\hc$ and every $\varphi\in\C_c(\R^2)$ with $\supp(\varphi)\subset [0,m]\times[0,T]$,
\begin{equation}
   \big|\langle \varphi(\cdot,s), \eta_\ell(s) \rangle\big|  \leq \|\varphi\|_\infty \lambda d.
\end{equation}
Hence, since $D_{\ell+1}$ is non-decreasing,
\begin{equation}
  |\xi_\ell(\varphi)| \leq \|\varphi\|_\infty \big(|\nu_\ell(0)|_{TV} + D_{\ell+1}(T) + C(T)\big),
\end{equation}
with $C(T)\doteq \lambda d T<\infty$, and  $D_{\ell+1}(T)<\infty$ by \eqref{Fluid_bound}. Moreover, $\xi_\ell(\varphi)$=0 for all $\varphi$ such that $\supp(\varphi)\cap\supint\times\R_+ =\emptyset$. Hence, $\xi$ is a Radon measure on $\R^2$ with support in $\supint\times\R_+$, i.e., $\xi\in\tM$. Moreover, since $\varphi$ has compact support, the left-hand side of \eqref{GAE_dynamic} is equal to zero for sufficiently large $t$. Therefore, sending $t\to\infty$ in \eqref{GAE_dynamic}, we have  for all $\varphi \in \phiset$,
\begin{equation}
  - \int_0^t \langle \varphi_x(\cdot, s)+\varphi_s(\cdot, s)-\varphi(\cdot, s)h(\cdot),\nu_\ell(s)\rangle ds =\xi_\ell(\varphi).
\end{equation}
Thus, $\nu_\ell$ satisfies the abstract  age equation associated with $\xi_\ell\in\tM$ and $h$.  Therefore, by Lemma \ref{lem_aae} and \eqref{CAE_proof_xi}, for all $f\in\mathbb{C}_c\supint$ and $t\geq0,$
\begin{align}\label{gae_solved}
  \langle f,\nu_\ell(t)\rangle =& \int_{\supint \times [0,t]} f(x+t-s)\frac{\overline G(x+t-s)}{\overline G(x)}\xi_\ell(dx,ds) \notag\\
   =&\langle f(\cdot+t)\frac{\overline G(\cdot+t)}{\overline G(\cdot)},\nu_\ell(0)\rangle +\int_{[0,t]} f(t-s)\overline G(t-s)dD_{\ell+1}(s)\notag\\
    &\hspace{2.5cm} + \int_0^t \langle f(\cdot+t-s)\frac{\overline G(\cdot+t-s)}{\overline G(\cdot)},\eta_\ell(s)\rangle ds.
\end{align}
Since for $t\geq0$, the right-hand side of \eqref{gae_solved} is finite for every $f\in\mathbb{C}_b\supint,$ the relation \eqref{gae_solved} can be extended to all $f\in\mathbb{C}_b\supint$ by an application of the dominated convergence theorem, and \eqref{Fluid_f} follows.
This completes the proof.
\end{proof}


\section{State Dynamics}\label{sec_model}

In Section \ref{sec_StateVariable} we express the state descriptor $\nun$  in terms of primitives of the networks.   This will be required to justify the existence of compensators of certain processes in Section \ref{sec_mgale}. We also  introduce some auxiliary processes in Section \ref{sec_AuxProcess}, and use them to derive dynamical equations for the state variables in Section \ref{sec_prelimit_prelimit}.

\begin{remark}
  To make it easy to follow the notation, throughout the paper, we use the superscript $i$ to denote queue indices and  subscript $j$ to denote jobs.
\end{remark}

\subsection{State Variables}\label{sec_StateVariable}

For each job $j$, let $\arrivalTime_j, \kninv_j$ and $\deptn_j$, respectively, represent the time at which job $j$ arrives into the system, enters service, and departs the queue on completing service. Note that  $\arrivalTime_j=\kninv_j$ if the job is routed to an empty queue, and  $\deptn_j=\kninv_j+\serviceTime_j$ where $\{v_j\}$ is the i.i.d.\ sequence of service times. We use the convention that jobs initially in the network are indexed by non-positive numbers $j=\j0,...,0$ with $\j0\doteq-\xn(0)+1$ being the smallest job index, where recall that $\xn(t)$ represents the total number of jobs in system at time $t$.   We also assume that jobs that entered service earlier get smaller indices. Jobs that arrive after time $0$ are given indices $j\geq1$ in the order of their arrival time ($0<\arrivalTime_j<\arrivalTime_{j+1}$ for $j\geq1$, almost surely).   Then the age $\agen_j(t)$ of job $j$ at time $t$ takes the form:
\begin{equation}\label{def_agen}
  \agen_j(t)\doteq\threepartdefE{0}{t<\kninv_j,}{t-\kninv_j}{\kninv_j\leq t <\deptn_j,}{v_j}{t\geq \deptn_j.}
\end{equation}

We also assign to each queue an index $i \in \{1, \ldots, N\}$.   To implement the $SQ($d$)$ routing algorithm, upon arrival, each job $j\ge 1$ chooses a vector  $\rqc_j=(\rqc_{j}(1),...,\rqc_{j}($d$))$ of $d$ indices,  each chosen independently and uniformly at random from the set $\{1, \ldots, N\}$ (in practice, it  would  be more natural to  sample $d$ queues at random without replacement, but we choose the former routing procedure for simplicity;  the effect of the  difference vanishes in the hydrodynamic limit). The job is then  routed to the queue with the shortest  length amongst the chosen indices, where if there are multiple queues of minimal length, then one of them is chosen uniformly at random. We denote the index of the queue to which job $j$ is routed  by $\stationn_j$:
\begin{equation}\label{def_station}
  \stationn_j\sim \text{Unif} \left( \text{argmin}\left\{ \xnX{\rqc_{j}(1)}(\arrivalTime_j-),...,\xnX{\rqc_{j}(d)}(\arrivalTime_j-)\right\} \right),
\end{equation}
where recall $\xni(t)$ is the number of jobs in the $i$th queue at time $t$. With a slight abuse of notation, we also use $\stationn_j(t)\doteq\indicil{t\geq \arrivalTime_j}\stationn_j$ to denote the queue index process.

For a Markovian description, the initial state of the network is completely determined by $\ren(0)$ from \eqref{def_Ren} and $\nun(0)$. However, it will prove convenient to  also refer to a more detailed description,  in which  $\stationn_j$ is specified for each job initially in the system. According to  our indexing convention, each job initially in the network has an index $j\in\{\j0 = -\xn(0)+1,...,N\},$ and any job initially in service has an index $j\in\{-\j0, -\xn(0)+1,...,-\xn(0)+\langle\f1,\nun_1(0)\rangle\}$.  Now, let
\begin{equation}\label{def_I0}
    I_0 \doteq \left(\ren(0),\agen_j(0),\stationn_j;j=\jo0,...,0\right).
\end{equation}
Given our indexing convention, $I_0$ and $(\ren(0),\nun(0))$ can be recovered from each other. Also,
 Assumption \ref{asm_initial}.\ref{asm_initial_ind} can be expressed in terms of the above notation as follows: for every finite subset $\cal K \subset \{-\xn(0)+1, \ldots, ...,0\}$ of  jobs initially in system,
\begin{equation}\label{vInitial}
    \Prob{v_{j}>b_j; j\in \cal K|I_0} =\prod_{j\in\cal K}\frac{\overline G(\agen_j(0)+b_j)}{\overline G(\agen_j(0))}, \quad b_j \geq 0.
\end{equation}
In other words, for every job $j$  not initially in service, $v_j$ is independent of $I_0$.

We now express the measures $\nun_\ell$ in terms of the primitives defined above.  A job $j$  receives service during the interval $[\kninv_j,\deptn_j)$, and hence,
\begin{equation}\label{def_Uactive}
  \Uactiven(t)\doteq\left\{j\geq \j0 \; :\;  \kninv_j\leq t<\deptn_j\right\}
\end{equation}
is the set of indices of  jobs receiving service at time $t$. Also, for $t \in  [\arrivalTime_j,\infty)$ let $\custServSize_j(t)$
 denote the length at time $t$ of the queue to which job $j$ was routed. In other words,
\begin{equation}\label{def-custservsize}
  \custServSize_j(t)= \xnX{\stationn_j}(t),\quad t\geq \arrivalTime_j.
\end{equation}
The value of $\custServSize_j(t)$ for $t<\arrivalTime_j$ is irrelevant. Therefore, for $\ell\geq1$,
\begin{equation}\label{def_Usize}
  \Usizeln(t)\doteq\left\{j\geq \j0: \indic{\arrivalTime_j\leq t} \custServSize_j(t)\geq \ell\right\}
\end{equation}
is the set of jobs in queues with length at least $\ell$ at time $t$. Using this notation, we can write
\begin{align}
    \nun_\ell(t) & \doteq \sum_{j=\j0}^{\infty}\indic{\kninv_j\leq t}\indic{\deptn_j>t}\indic{\custServSize_j(t)\geq \ell}\delta_{\agen_j(t)} \label{nuln_alternative}\\
    & = \sum_{j=\j0}^{\infty}\indic{j\in\Uactiven(t)}\indic{j\in\Usizeln(t)}\delta_{\agen_j(t)}. \label{def_nuln}
\end{align}

The pair $(\ren,\nun)$ with $\nun=(\nun_\ell;\ell\geq1)$ is the state descriptor of the $N$-server network.

\subsection{Auxiliary Processes and Filtration}\label{sec_AuxProcess}

To describe the dynamics of  $\nun$, it will be convenient to introduce a number of auxiliary processes. For every queue $i\in\{1,...,N\}$, let $\arrivalni$ denote the cumulative arrival process to queue $i$, defined as
\begin{equation}\label{def_Ei}
  \arrivalni(t)\doteq\sum_{j=1}^\infty \indic{\arrivalTime_j\leq t}\indic{\stationn_j=i},\quad\quad t\geq 0.
\end{equation}
For $\ell\geq1$, let $\nuni_\ell (t)$ denote the measure that has a Dirac delta mass at the age of the job in service at queue $i$ (if any) at time $t$: that is, for $t\geq0$,
\begin{equation}\label{nuni}
    \nuni_\ell(t)\doteq\sum_{j=\j0}^{\infty}\indic{\kninv_j\leq t}\indic{\deptn_j>t}\indic{\custServSize_j(t)\geq \ell}\indic{\stationn_j=i}\delta_{\agen_j(t)}.
\end{equation}
Note that $\nuni_\ell(t)$ always has mass either zero or 1, and clearly,
\begin{equation}\label{nun_nuni}
    \nun_\ell=\sum_{i=1}^N\nuni_\ell.
\end{equation}

Fix $\ell\geq1$.  For $\varphi\in\tightphiset$, let $\Rn_{\varphi,\ell}$ be the cumulative $\varphi$-weighted routing measure process to queues with length exactly $\ell-1$, defined as follows for all $t\geq 0:$
\begin{equation}\label{def_Rphi}
    \Rn_{\varphi,\ell}(t)\doteq \twopartdef{ \displaystyle \sum_{i=1}^N \int_{(0,t]} \varphi(0,s) (1 - \langle \f1, \nuni_1 (s-)\rangle )d\arrivalni(s)}{\ell=1,}{\displaystyle\sum_{i=1}^{N}\int_{(0,t]} \langle \varphi(\cdot,s),\nuni_{\ell-1}(s-) - \nuni_{\ell}(s-) \rangle d\arrivalni(s),}{\ell \geq 2.}
\end{equation}
Roughly speaking, $\Rn_{\varphi,\ell}(t)$ captures the cumulative effect on the measure $\nu_\ell^{(N)}$ due to jobs routed in the interval $[0,t]$.  Indeed, in
 both cases, $\varphi (\cdot, s)$  in the integral is  evaluated at the age of the job in service at queue $i$.
 Note that  when $\ell = 1$,
since $\langle \f1, \nuni(s-)\rangle = X^{(N),i}(s-)$, we can also write
\begin{equation}\label{def_Rphi1}
  \Rn_{\varphi,1}(t)=\sum_{i=1}^N \int_{(0,t]} \varphi(0,s) \indic{\xni(s-)=0} d\arrivalni(s).
\end{equation}

Next, we turn to the counting process $\dn_\ell=\{\dn_\ell(t);t\geq0\}$ of departures from queues with length  at least $\ell$ right before departure.   For conciseness, we use the following notation for values of the queue length of job $j$ right after its arrival time or service entry  and right before its departure time:
\begin{align}\label{def_CustXS}
  \custES_j\doteq\custServSize_j(\arrivalTime_j),\quad
  \custKS_j\doteq\custServSize_j(\kninv_j),\quad
  \custDS_j\doteq\custServSize_j(\deptn_j-),
\end{align}
where $\custServSize_j(\cdot)$ is the queue length process defined in \eqref{def-custservsize}. Then we have
\begin{equation}\label{def_Dl}
  \dn_\ell(t) = \sum_{j=\j0}^\infty \indic{\deptn_j\leq t}\indic{\custDS_j\geq \ell},\quad\quad t\geq0.
\end{equation}
Note that  $\dn\doteq\dn_1$ is the total cumulative departure process.

\begin{remark}\label{remark_custX}
Since a queue is never empty just prior to a departure or right after a service entry, we have  $\custDS_j\geq1$ and $\custKS_j\geq1$. Also, a simple mass balance shows that
\begin{equation}\label{apx_massTotal}
   \dn(t) + \langle \f1, \nun_1 (t) \rangle \leq \xn(0)+\arrivaln(t).
\end{equation}
\end{remark}

For  $\varphi\in\tightphiset$, let $\Quen_{\varphi,\ell}$ be the cumulative $\varphi$-weighted departure process from queues of length at least $\ell$, defined by
\begin{equation}\label{def_Qphi}
  \Quen_{\varphi,\ell}(t) \doteq \sum_{j=\j0}^{\infty} \varphi\left(\serviceTime_j,\deptn_{j}\right) \indic{\deptn_j\leq t}\indic{\custDS_j\geq \ell},\quad\quad  t\geq 0.
\end{equation}
Clearly, $\Quen_{\f1,\ell}=\dn_\ell,$ and hence by \eqref{apx_massTotal},
\begin{equation}\label{apx_Qbound}
 |\Quen_{\varphi,\ell}(t)|\leq \|\varphi\|_\infty \left( \xn(0)+\arrivaln(t) \right), \quad t \geq 0.
\end{equation}
For $i\in\{1,...,N\}$, let $\dni$ denote the departure process from queue $i$. Then
\begin{equation}\label{def_Di}
\dni(t) = \sum_{j=\j0}^\infty\indic{\deptn_j\leq t}\indic{\stationn_j=i},\quad t\geq0,
\end{equation}
\begin{equation}\label{mass_balance_i}
  \xni(t)=\xni(0)+\arrivalni(t)-\dni(t), \quad t \geq 0.
\end{equation}

Finally, we define the filtration $\{\filtn_t;t\geq0\}$ generated by the initial conditions of the network, see \eqref{def_I0},  plus the filtrations $\mathcal{F}_t^{\arrivalni}$ and $\mathcal{F}_t^{\dni}$  generated by $\arrivalni$ and $\dni$, respectively, $i = 1, \ldots, N$. In other words,
\begin{equation}\label{def_filt}
  \filtn_t\doteq\bigvee_{i=1}^N \left(\mathcal{F}_t^{\arrivalni}\vee\mathcal{F}_t^{\dni}\right)\vee\sigma(I_0),
\end{equation}
We show in Proposition \ref{prop_adapted} that all state variables and auxiliary processes are $\{\filtn_t\}$-adapted.

\begin{remark}
It is possible to show that $\{\filtn_t\}$  is also equal to  the filtration generated by the age and queue index processes $\agen_j(\cdot),\stationn_j(\cdot);j\geq1$. However, our definition allows us to exploit results from \cite{BremaudBook} in Section \ref{sec_mgale} to identify compensators of certain processes.
\end{remark}

\begin{remark}
One can also show that $\{(\ren(t),\nun(t));t\geq0\}$ is a Markov process with respect to the filtration $\{\filtn_t;t\geq0\}$.  But, since we do not use this property,  we do not prove it.
\end{remark}

\subsection{Equations governing the dynamics of the N-Server Network}\label{sec_prelimit_prelimit}
We now describe the dynamics of the state descriptor $\nun$ for a fixed $N$. Our main result (Proposition \ref{prelim_dynamics_prop}) does not require all our assumptions on the arrival process and service time distribution, but instead holds for a very general class of networks and load balancing algorithms, as long as all arrival and departure times are distinct, almost surely. To make this notion precise,  denote by $\Omega_{s,\delta}$ the set of realizations for which at most one arrival or one departure occurs during $(s, s+\delta]$. Also, define $\Omega_t$ to be the set of realizations for which there exists a partition $\{(\frac{k}{n},\frac{k+1}{n}];k=0,...,\lfloor nt\rfloor\}$ of $(0,t]$ such that at most one arrival or one departure occurs in each subinterval, that is,
\begin{equation}\label{def_Omegat}
  \Omega_t\doteq \bigcup_{n=1}^\infty\bigcap_{k=0}^{\lfloor nt\rfloor}\Omega_{\frac{k}{n},\frac{1}{n}}.
\end{equation}
The result in Proposition \ref{prelim_dynamics_prop} holds for a much larger class of load-balancing algorithms. In fact,  as long as for every $t\geq0,$ $\Omega_t$ has full measure, the equation \eqref{prelim_dymamics_eq} holds for the routing process $\Rn$ associated with the corresponding algorithm.
\begin{proposition}\label{prelim_dynamics_prop}
Consider an $N$-server network with any arrival process, service times and load balancing algorithm such that  $\Probil{\Omega_t}=1$ for every $t\geq0$. Then, for $\varphi\in\phiset$, almost surely, for $\ell\geq 1$ and  $t\geq 0$,
\begin{align}\label{prelim_dymamics_eq}
  \langle\varphi(\cdot,t),\nun_\ell(t)\rangle =& \langle\varphi(\cdot,0),\nun_\ell(0)\rangle +\int_0^t \langle \varphi_s(\cdot,s)+\varphi_x(\cdot,s),\nun_\ell(s)\rangle ds - \Quen_{\varphi,\ell}(t)\notag\\
  &+ \int_{[0,t]}\varphi(0,s)d\dn_{\ell+1}(s)+ \Rn_{\varphi,\ell}(t);
\end{align}
\begin{equation}
    \label{prelim_dymamics_balance}
  \langle \f1,\nun_\ell(t)\rangle = \langle\f1,\nun_\ell(0)\rangle - \dn_\ell(t)+ \dn_{\ell+1}(t)+ \Rn_{\f1,\ell}(t).
\end{equation}
\end{proposition}

The rest of this section is devoted to the proof of this proposition. Throughout this section, for ease of notation,  $\sum_j$ is used to denote the sum over all job indices $j\in\{\j0,...,-1,0,1,...\}.$

Fix $\ell\geq 1$, $t\geq0$. For $\varphi\in\mathbb{C}_b(\supint\times\hc),$ since $s \mapsto \langle\varphi(\cdot,s),\nun_\ell(s)\rangle$ is right-continuous, we can write
\begin{align}\label{prelim_discretization_eq}
  \langle\varphi(\cdot,t),\nun_\ell(t)\rangle -  \langle\varphi(\cdot,0),\nun_\ell(0)\rangle =& \lim_{n\to\infty}\sum_{k=0}^{\lfloor nt\rfloor} \left[\langle \varphi(\cdot,\frac{k+1}{n}),\nun_\ell(\frac{k+1}{n}) \rangle  -\langle \varphi(\cdot,\frac{k}{n}),\nun_\ell(\frac{k}{n}) \rangle\right]\notag\\
  =& \lim_{n\to\infty}\left( \In_1+\In_2\right),
\end{align}
where
\begin{equation}\label{prelim_I1def}
  \In_1(t) \doteq \sum_{k=0}^{\lfloor nt\rfloor} \langle \varphi(\cdot,\frac{k+1}{n})-\varphi(\cdot,\frac{k}{n}) ,\nun_\ell(\frac{k+1}{n}) \rangle
\end{equation}
and
\begin{equation}\label{prelim_I2def}
\In_2(t) \doteq \sum_{k=0}^{\lfloor nt\rfloor} \langle \varphi(\cdot,\frac{k}{n}),\nun_\ell(\frac{k+1}{n}) \rangle  -\langle \varphi(\cdot,\frac{k}{n}),\nun_\ell(\frac{k}{n}) \rangle.
\end{equation}

We first compute $\In_1(t)$ for $\varphi\in\phiset$. For $x\in\supint$, by the mean value theorem for $\varphi(x,\cdot)$,
\begin{align*}
\varphi(x,\frac{k+1}{n})-\varphi(x,\frac{k}{n})=& \frac{1}{n}\varphi_s(x,\frac{k+1}{n}-\delta(x,n,k))
\end{align*}
for some $\delta(x,n,k)\in[0, \frac{1}{n}]$. Therefore, we have
\begin{align}\label{prelim_temp1}
   &\lim_{n\to\infty}\left|\In_1 (t) -\int_0^t \langle \varphi_s(\cdot,s),\nun_\ell(s)\rangle ds\right| \notag\\
   &\hspace{2cm}  \leq  \lim_{n\to\infty}\left|\frac{1}{n}\sum_{k=0}^{\lfloor nt\rfloor} \langle \varphi_s(\cdot,\frac{k+1}{n}), \nun_\ell(\frac{k+1}{n}) \rangle-\int_0^t \langle \varphi_s(\cdot,s),\nun_\ell(s)\rangle ds\right|\notag\\
  &\hspace{2.5cm}+\lim_{n\to\infty}\left|\frac{1}{n}\sum_{k=0}^{\lfloor nt\rfloor} \langle \varphi_s(\cdot,\frac{k+1}{n}-\delta(\cdot,n,k))-\varphi_s(\cdot,\frac{k+1}{n}),\nun_\ell(\frac{k+1}{n}) \rangle\right|
\end{align}
The first term on the right-hand side above is equal to zero because $\nuln$ is right-continuous and $\varphi_s$ is continuous, and hence $s\mapsto\langle \varphi_s(\cdot,s),\nun_\ell(s)\rangle$ is Riemann integrable in $[0,t]$. Also, since $\varphi$ is a function in  $\mathbb{C}^{1,1}$  with support in $[0,m]\times[0,T]$ for some $m<\endsup$ and $T<\infty$, $\varphi_s$ is uniformly continuous, and hence given any $\epsilon>0$, for all sufficiently large $n$, we have
\[\left|\varphi_s(x,\frac{k+1}{n}-\delta(x,n,k))-\varphi_s(x,\frac{k+1}{n})\right|\leq \epsilon,\quad\quad \forall x\in[0,m],k\geq 1.\]
Therefore, using the bound $\langle1,\nun_\ell(t)\rangle\leq N$,
\[\left|\frac{1}{n}\sum_{k=0}^{\lfloor nt\rfloor} \langle \varphi_s(\cdot,\frac{k+1}{n}-\delta(\cdot,n,k))-\varphi_s(\cdot,\frac{k+1}{n}),\nun_\ell(\frac{k+1}{n}) \rangle\right| \leq (t+1)\epsilon N.\]
Since $\epsilon$ is arbitrary, the second term on the right-hand side of \eqref{prelim_temp1} is also equal to zero. Therefore, we have shown that
\begin{equation}\label{prelim_I1}
  \lim_{n\to\infty}\In_1 (t) = \int_0^t \langle \varphi_s(\cdot,s),\nun_\ell(s)\rangle ds.
\end{equation}

To compute the limit of $\In_2(t)$, first fix  $s,\delta\geq 0$ and use the expressions for $\nun_\ell$, $\agen_j$
$\Uactiven$ and $\Usizeln$ in \eqref{def_nuln}, \eqref{def_agen}, \eqref{def_Uactive} and \eqref{def_Usize} to write

\begin{align}\label{nulsd1}
  \nuln(s+\delta)=&\sum_{j}\indic{j\in\Uactiven(s+\delta)\cap\Usizeln(s+\delta)}\delta_{\agen_j(s+\delta)}
=  \terma + \termb + \termc
\end{align}
where
\begin{align*}
\terma    \doteq &\sum_{j}\indic{j\in\Uactiven(s+\delta)\cap\Usizeln(s+\delta)\backslash \Uactiven(s)}\delta_{\agen_j(s+\delta)}\\
\termb  \doteq    &\sum_{j}\indic{j\in\Uactiven(s+\delta)\cap\Uactiven(s)\cap\Usizeln(s+\delta)\backslash\Usizeln(s)}\delta_{\agen_j(s)+\delta}\notag\\
 \termc  \doteq     &\sum_{j}\indic{j\in\Uactiven(s+\delta)\cap\Uactiven(s)\cap\Usizeln(s+\delta)\cap\Usizeln(s)}\delta_{\agen_j(s)+\delta}. \notag\:\:.
\end{align*}
Next,  applying the identity $C \cap F\cap \tilde F =  \tilde F \backslash \left\{ [ \tilde F\backslash C] \cup [ (\tilde F \cap  C)\backslash F]\right\}$ with $C=\Uactiven(s+\delta)$, $F=\Usizeln(s+\delta)$ and  $\tilde F=\Uactiven(s) \cap  \Usizeln(s)$, we have
\begin{align}\label{nulsd2}
  \termc   &= \sum_j\indic{j \in \Uactiven(s)\cap\Usizeln(s)}\delta_{\agen_j(s)+\delta}
- \termcb \\
  &\qquad  -\sum_j\indic{j \in [\Uactiven(s)\cap\Usizeln(s)\cap\Uactiven(s+\delta)]\backslash\Usizeln(s+\delta)}\delta_{\agen_j(s)+\delta},\notag
\end{align}
where
\[ \termcb  \doteq \sum_j
\indic{j \in [\Uactiven(s)\cap\Usizeln(s)]\backslash\Uactiven(s+\delta)}\delta_{\agen_j(s)+\delta}.
\]
Note we have not used any property of $\Usizeln$ in this calculation. We now restrict to realizations $\Omega_{s,\delta}$ when there is only a single arrival or departure during $(s,s+\delta]$.

\begin{lemma}\label{prelim_deltanuABS_lemma}
For $\ell\geq 1$ and $\varphi\in\mathbb{C}_b(\supint\times\hc)$,  $s\geq 0$ and $\delta>0$, on $\Omega_{s,\delta}$,
\begin{align}\label{prelim_deltanuABS_eq}
&\langle \varphi(\cdot,s),\nun_\ell(s+\delta)\rangle  = \langle \varphi(\cdot+\delta,s),\nun_\ell(s)\rangle\\
&\hspace{2cm}-\sum_{j} \varphi\left(\agen_j(s)+\delta,s\right) \indic{\deptn_j\in(s,s+\delta]} \indic{\custDS_j\geq \ell} \notag\\
&\hspace{2cm}+\sum_{j}\varphi\left(\agen_j(s+\delta),s\right)\indic{\kninv_j\in(s,s+\delta]} \indic{\custKS_j\geq \ell} \notag\\
&\hspace{2cm}+\indic{\ell\geq 2}\sum_{j}\varphi\left(\agen_j(s+\delta),s\right)\indic{\kninv_j\leq s}\indic{\deptn_j>s}\notag\\
&\hspace{4.7cm}\indic{\custServSize_j(s)=\ell-1}\sum_{j^\prime\geq1}\indic{\arrivalTime_{j^\prime}\in(s,s+\delta]} \indic{\stationn_{j^\prime}=\stationn_{j}}.\notag
\end{align}
\end{lemma}
\begin{remark}
Since $\custDS_j\geq1$ and $\custKS_j\geq1$, for $\ell=1$ \eqref{prelim_deltanuABS_eq} reduces to
\begin{align}\label{prelim_deltanuABS1_eq}
\langle \varphi(\cdot,s),\nun_1(s+\delta)\rangle  =& \langle \varphi\left(\cdot+\delta,s\right),\nun_1(s)\rangle -\sum_{j} \varphi\left(\agen_j(s)+\delta,s\right) \indic{\deptn_j\in(s,s+\delta]} \notag\\
&+\sum_{j} \varphi\left(\agen_j(s+\delta),s\right)\indic{\kninv_j\in(s,s+\delta]}.
\end{align}
\end{remark}

\begin{proof}[Proof of Lemma \ref{prelim_deltanuABS_lemma}]
We first simplify the terms $\terma, \termb$ and $\termc$ in  \eqref{nulsd1}. A job $j$ receives service at time $s+\delta$, but not at $s$ if and only if it entered service during $(s,s+\delta]$. Moreover, on $\Omega_{s,\delta}$, there could have been no arrivals in $(s,s+\delta]$ and so, the length of the job's queue is constant from the service entry time to $s+\delta$.  This implies  $j\in\Usizeln(s+\delta)$ if and only if $\custKS_j\geq \ell$. Thus,
\begin{align}\label{nulsd_temp1}
\terma &=\sum_{j}\indic{\kninv_j\in(s,s+\delta]} \indic{\custKS_j\geq \ell} \delta_{\agen_j(s+\delta)}.
\end{align}

We now analyze the term $\termb$. If a job $j$ received service throughout the period $(s,s+\delta]$, that is $j\in\Uactiven(s)\cap\Uactiven(s+\delta)$, then the corresponding queue could not have been empty at time $s$, and on $\Omega_{s,\delta}$, the difference between the queue length at time $s+\delta$ and time $s$ is either zero (if there were no arrivals to that queue) or one (if there was precisely one arrival to that queue.) Therefore, when $\ell=1$,
\begin{equation}
\Uactiven(s)\cap\Uactiven(s+\delta)\cap\Usizen_1(s+\delta)\backslash\Usizen_1(s)=\emptyset,
\end{equation}
and $\termb = 0$, whereas
 for $\ell=2,$ using the representation \eqref{def_Ei} for $\arrivalni$,
\begin{align}
\termb   &\hspace{.5cm}=\sum_{j}\indic{\kninv_j\leq s}\indic{\deptn_j>s}\indic{\custServSize_j(s)= \ell-1}\arrivalnX{\stationn_j}(s,s+\delta] \delta_{\agen_j(s+\delta)}\notag\\
  &\hspace{.5cm}=\sum_{j}\indic{\kninv_j\leq s}\indic{\deptn_j>s}\indic{\custServSize_j(s)=\ell-1}\delta_{\agen_j(s+\delta)}\sum_{j^\prime\geq1} \indic{\arrivalTime_{j^\prime}\in(s,s+\delta],\stationn_{j^\prime}=\stationn_{j}}.\notag
\end{align}

For the third term $\termc$, we use \eqref{nulsd2}. First, note that by the form \eqref{def_nuln} of $\nuln$,
\begin{align}
  \langle \varphi(\cdot,s),\sum_{j}\indic{j\in\Uactiven(s)\cap \Usizeln(s)}\delta_{\agen_j(s)+\delta}\rangle = \langle \varphi(\cdot+\delta),s),\nuln(s)\rangle.
\end{align}

Next, note that a job $j$  departed  a queue during $(s,s+\delta]$ if and only if  $j\in\Uactiven(s)\backslash\Uactiven(s+\delta)$. Moreover, on $\Omega_{s,\delta}$ there were no arrivals during $(s,s+\delta]$, and hence, the queue length was constant on $(s,\deptn_j-)$. Therefore, $j\in\Usizeln(s)$ if and only if $\custDS_j\geq\ell$. Hence,
\begin{align}
 \termcb &=  \sum_{j}\indic{\deptn_j\in(s,s+\delta]} \indic{\custDS_j\geq \ell} \delta_{\agen_j(s)+\delta}.
\end{align}

Finally,  for the last term on the right-hand side of \eqref{nulsd2}, note that if a job $j$ receives service at a queue during $(s,s+\delta]$, then that queue length is non-decreasing on that interval.  Therefore,
\begin{equation}\label{nulsd_temp2}
  \left[\Uactiven(s)\cap\Uactiven(s+\delta)\cap\Usizeln(s)\right]\backslash\Usizeln(s+\delta)=\emptyset.
\end{equation}
The result follows from \eqref{nulsd1}-\eqref{nulsd2} and \eqref{nulsd_temp1}-\eqref{nulsd_temp2}.
\end{proof}

We continue with the identification of the limit of $\In_2(t)$. Since $\Probil{\Omega_t}=1$  by  assumption, there exists $n_0\in \N$ such that almost surely, the identity \eqref{prelim_deltanuABS_eq} holds with $\delta =1/n$ and $s=k/n$ simultaneously for every $n\geq n_0$ and $ k=0,1,...,\lfloor nt \rfloor$. Substituting \eqref{prelim_deltanuABS_eq}, with $\delta =1/n$ and $s=k/n,$  into \eqref{prelim_I2def} we have almost surely,
\begin{align}\label{prelim_I2}
   \In_2(t) = &\sum_{k=0}^{\lfloor nt\rfloor} \left[\left\langle \varphi(\cdot+\frac{1}{n},\frac{k}{n})-\varphi(\cdot,\frac{k}{n}) ,\nun_\ell(\frac{k}{n}) \right\rangle\right]\\
  &-\sum_{k=0}^{\lfloor nt \rfloor}\sum_{j} \varphi\left(\agen_j(\frac{k}{n})+\frac{1}{n},\frac{k}{n}\right)\indic{\deptn_j\in(\frac{k}{n},\frac{k+1}{n}]} \indic{\custDS_j\geq \ell} \notag\\
  &+\sum_{k=0}^{\lfloor nt \rfloor} \;\; \sum_{j} \varphi\left(\agen_j(\frac{k+1}{n}),\frac{k}{n}\right) \indic{\kninv_j\in(\frac{k}{n},\frac{k+1}{n}]} \indic{\custKS_j\geq \ell} \notag\\
  &+\indic{\ell\geq 2}\sum_{k=0}^{\lfloor nt \rfloor}\;\;\sum_{j}\;\;\sum_{j^\prime\geq1} \varphi\left(\agen_j(\frac{k+1}{n}),\frac{k}{n}\right)\indic{\kninv_j\leq\frac{k}{n}}\indic{\deptn_j>\frac{k}{n}}\notag\\ &\hspace{4.5cm}\indic{\custServSize_j(\frac{k}{n})= \ell-1} \indic{\arrivalTime_{j^\prime}\in(\frac{k}{n},\frac{k+1}{n}]}
     \indic{\stationn_{j^\prime}=\stationn_{j}}\notag.
\end{align}

For $\varphi\in\phiset$, using computations analogous to the derivation of the limit of  $\In_1$ in \eqref{prelim_I1def},  the limit of the first term on the right-hand side of \eqref{prelim_I2} is
\begin{equation}\label{prelim_deltanu1_eq}
  \lim_{n\to\infty}\sum_{k=0}^{\lfloor nt\rfloor} \left[\left\langle \varphi(\cdot+\frac{1}{n},\frac{k}{n})-\varphi(\cdot,\frac{k}{n}) ,\nun_\ell(\frac{k}{n}) \right\rangle\right] = \int_0^t \langle \varphi_x(\cdot,s),\nuln(s)\rangle ds.
\end{equation}

For the second term, setting $\deptn_{j,n}\doteq\frac{1}{n}\lfloor n\deptn_j\rfloor$,  noting that $\deptn_{j,n}\uparrow \deptn_j$ as $n\to\infty$,  using the continuity of  $\varphi$ and $\agen_j$,  and the identity $\agen_j(\deptn_j)=\serviceTime_j$,  we have
\begin{align}\label{prelim_deltanu2_eq}
    &\lim_{n\to\infty}\sum_{k=0}^{\lfloor nt \rfloor}\sum_{j} \varphi\left(\agen_j(\frac{k}{n})+\frac{1}{n},\frac{k}{n}\right)\indic{\deptn_j\in(\frac{k}{n},\frac{k+1}{n}]} \indic{\custDS_j\geq \ell} \notag\\
    & \hspace{1cm}=\lim_{n\to\infty}\sum_{j}\;\sum_{k=0}^{\lfloor nt \rfloor} \varphi\left(\agen_j(\frac{k}{n})+\frac{1}{n},\frac{k}{n}\right)\indic{\deptn_j\in(\frac{k}{n},\frac{k+1}{n}]} \indic{\custDS_j\geq \ell}\notag\\
    & \hspace{1cm}=\lim_{n\to\infty}\sum_{j} \varphi\left(\agen_j(\deptn_{j,n})+\frac{1}{n},\deptn_{j,n}\right) \indic{\deptn_j\leq \frac{\lceil nt\rceil}{n}} \indic{\custDS_j\geq \ell} \notag\\
    & \hspace{1cm}=\sum_{j} \varphi\left(\serviceTime_j,\deptn_{j}\right) \indic{\deptn_j\leq t}\indic{\custDS_j\geq \ell} \notag\\
    & \hspace{1cm}=\Quephinl(t).
\end{align}
where the last equality follows from \eqref{def_Qphi}.

Likewise, for the third term, setting $\kninv_{j,n}\doteq\frac{1}{n}\lfloor n\kninv_j\rfloor$, and noting that $\agen_j(\kninv_j)=0$ and $\kninv_{j,n}\uparrow \kninv_j$ as $n\to\infty$,  we obtain
\begin{align} \label{prelim_deltanu3_eq}
  &\lim_{n\to\infty}\sum_{k=0}^{\lfloor nt \rfloor} \;\; \sum_{j} \varphi\left(\agen_j(\frac{k+1}{n}),\frac{k}{n}\right) \indic{\kninv_j\in(\frac{k}{n},\frac{k+1}{n}]} \indic{\custKS_j\geq \ell} \notag\\
 &\hspace{1cm}=\lim_{n\to\infty} \sum_{j} \varphi\left(\agen_j(\kninv_{j,n}+\frac{1}{n}),\kninv_{j,n}\right) \indic{0\leq \kninv_j\leq \frac{\lceil nt\rceil}{n}} \indic{\custKS_j\geq \ell} \notag\\
 &\hspace{1cm}= \sum_{j} \varphi\left(0,\kninv_{j}\right) \indic{0\leq \kninv_j\leq t}\indic{\custKS_j\geq \ell}.
\end{align}
Further simplification of \eqref{prelim_deltanu3_eq} is slightly different for $\ell=1$ and $\ell\geq2$. When $\ell\geq 2,$ due to the non-idling assumption, the service entry time of any job to a queue of length at least $\ell$ just after service entry coincides with the departure time of another job from the same queue, which had length at least $\ell+1$ just before the departure. Therefore, for $\ell\geq 2$, by definition \eqref{def_Dl} of $\dn_\ell$ we have
\begin{align}\label{prelim_deltanu32_eq}
  \sum_{j} \varphi\left(0,\kninv_{j}\right) \indic{0\leq \kninv_j\leq t}\indic{\custKS_j\geq \ell}\notag
  &=\sum_{j} \varphi\left(0,\deptn_{j}\right) \indic{ \deptn_j\leq t}\indic{\custDS_j\geq \ell+1}\notag\\
  &=\int_{[0,t]}\varphi(0,s)d\dn_{\ell+1}(s).
\end{align}
To simplify \eqref{prelim_deltanu3_eq} for $\ell=1$, we first define the cumulative service entry process:
\begin{equation}\label{def_K}
  \kn(t)\doteq\sum_{j=\j0}^\infty\indic{0\leq\kninv_j\leq t}, \quad t \geq 0.
\end{equation}

\begin{lemma}\label{lem_K}
Given $\dn_2$ and $\arrivalni$ defined in \eqref{def_Dl} and \eqref{def_Ei}, respectively,
\begin{equation}\label{KDERelation}
  \kn(t)=\dn_2(t)+\sum_{i=1}^N\;\int_0^t\indic{\xni(u-)=0}d\arrivalni(u), \quad t \geq 0.
\end{equation}
\end{lemma}
\begin{proof}
Service entries can be classified into two types, based on whether or not the queue was empty right before service entry. Thus, we can expand \eqref{def_K} as
\begin{align*}
  \kn(t) =&\sum_{j}  \indic{0\leq \kninv_j\leq t}\indic{\custServSize_j(\kninv_j-)\geq 1}
+\sum_{j} \indic{0\leq \kninv_j\leq t}\indic{\custServSize_j(\kninv_j-)= 0}.
\end{align*}
Due to the non-idling assumption, the service entry time of the first type coincides with the departure time of another job from the same queue, which had length of at least $2$ just before departure.  On the other hand, a service entry time of the second type coincides with the arrival time of the same job to an empty  queue. Recalling that  $\stationn_j$ is the queue index of job $j$, we can then write
\begin{align}
  \kn(t)=&\sum_{j}  \indic{\deptn_j\leq t}\indic{\custDS_j\geq 2}
         +\sum_{j\geq1} \indic{\arrivalTime_j\leq t}\indic{\custES_j=1}\notag\\
        =&\sum_{j}  \indic{\deptn_j\leq t}\indic{\custDS_j\geq 2}+\sum_{i=1}^N\;\sum_{j=1}^{\infty}  \indic{\xni(\arrivalTime_j-)=0}\indic{\arrivalTime_j\leq t}\indic{\stationn_j=i}\notag.
\end{align}
Equation \eqref{KDERelation} then follows from \eqref{def_Dl} and \eqref{def_Ei}.
\end{proof}
By \eqref{def_K}-\eqref{KDERelation},  the fact that $\indicil{\custKS_j\geq 1}=1$ (see Remark \ref{remark_custX}), and
\eqref{def_Rphi1},
we have
\begin{align}\label{prelim_deltanu31_eq}
&\sum_{j} \varphi\left(0,\kninv_{j}\right) \indic{0\leq \kninv_j\leq t}\indic{\custKS_j\geq 1} \notag\\
&\hspace{2cm}=\int_{[0,t]}\varphi(0,u)d\dn_2(u)+  \sum_{i=1}^N\;\int_0^t\varphi(0,u)\indic{\xni(u-)=0}d\arrivalni(u)\notag\\
&\hspace{2cm}=\int_{[0,t]}\varphi(0,u)d\dn_2(u)+  \Rn_{\varphi,1} (t).
\end{align}

Finally, the last term on the right-hand side of \eqref{prelim_I2} is zero for $\ell=1$. For $\ell\geq 2$, changing the order of summation, setting $\arrivalTime_{j^\prime,n}=\frac{1}{n}\lfloor n\arrivalTime_{j^\prime}\rfloor$,  noting that on $\Omega_t$, the arrival of $j^\prime$ is the only event taking place in the interval $(k/n, (k+1)/n]$,   the limit of the last term on the right-hand side of \eqref{prelim_I2} is equal to
\begin{align}
  &\hspace{1cm}\lim_{n\to\infty} \sum_{j^\prime\geq1}\;\sum_{j} \varphi\left(\agen_j(\arrivalTime_{j^\prime,n}+\frac{1}{n}),\arrivalTime_{j^\prime,n}\right) \indic{\kninv_j\leq\arrivalTime_{j^\prime,n}}\indic{\deptn_j>\arrivalTime_{j^\prime,n}}\notag\\
  &\hspace{2.7cm} \indic{\custServSize_j(\arrivalTime_{j^\prime,n})= \ell-1} \indic{\arrivalTime_{j^\prime}\leq \frac{\lceil nt\rceil}{n}} \indic{\stationn_{j^\prime}=\stationn_{j}}\notag.
\end{align}
Since  $\arrivalTime_{j^\prime,n}\uparrow \arrivalTime_{j^\prime}$ as $n\to\infty$, the fact that on $\Omega_t$, $\kninv_j,\deptn_j\neq\arrivalTime_{j^\prime}$ for all $j\neq j^\prime$, and by the continuity of  $\varphi$ and $\agen_j$, the last display is equal to
\begin{align*}
&\sum_{j^\prime\geq1}\;\sum_{j}
 \varphi\left(\agen_j(\arrivalTime_{j^\prime}),\arrivalTime_{j^\prime}\right)  \indic{\kninv_j<\arrivalTime_{j^\prime} \leq t \wedge \deptn_j}
 \indic{\custServSize_j(\arrivalTime_{j^\prime}-)= \ell-1}\indic{\stationn_{j^\prime}=\stationn_{j}}.
\end{align*}
Partitioning jobs in terms of their queues, and using  \eqref{nuni}, the last display equals
\begin{align}
  &\hspace{1cm}\sum_{i=1}^N\;\;\sum_{j^\prime\geq1}  \indic{\stationn_{j^\prime}=i}\indic{\arrivalTime_{j^\prime}\leq t} \sum_{j} \varphi\left(\agen_j(\arrivalTime_{j^\prime}),\arrivalTime_{j^\prime}\right) \indic{\kninv_j<\arrivalTime_{j^\prime}}\notag\\
  &\hspace{2.7cm} \indic{\deptn_j\geq\arrivalTime_{j^\prime}}\indic{\custServSize_j(\arrivalTime_{j^\prime}-)= \ell-1}  \indic{\stationn_{j}=i}\notag  \\
  &\hspace{1cm}=\sum_{i=1}^N\sum_{j^\prime\geq1} \indic{\stationn_{j^\prime}=i} \indic{\arrivalTime_{j^\prime}\leq t}\langle \varphi(\cdot,\arrivalTime_{j^\prime}),\nuni_{\ell-1}(\arrivalTime_{j^\prime}-)- \nuni_{\ell}(\arrivalTime_{j^\prime}-)  \rangle\notag\\
  &\hspace{1cm}=\sum_{i=1}^{N}\int_{(0,t]} \langle \varphi(\cdot,s),\nuni_{\ell-1}(s-) - \nuni_{\ell}(s-) \rangle d\arrivalni(s)\notag\\
  &\hspace{1cm}=\Rn_{\varphi,\ell}(t), \label{prelim_deltanu4_eq}
\end{align}
where $\arrivalni$ and $\Rn_{\varphi, \ell} (t)$ are defined in \eqref{def_Ei} and \eqref{def_Rphi}, respectively.

We now combine the above observations to conclude the proof.

\begin{proof}[Proof of Proposition \ref{prelim_dynamics_prop}]
Equation \eqref{prelim_dymamics_eq}  follows from \eqref{prelim_discretization_eq}, \eqref{prelim_I1}, \eqref{prelim_I2}-\eqref{prelim_deltanu32_eq},  \eqref{prelim_deltanu31_eq} and \eqref{prelim_deltanu4_eq}. To establish \eqref{prelim_dymamics_balance}, note that for $\varphi =\f1$, $\In_1(t)$ and the first term on the right-hand side of \eqref{prelim_I2} are zero for all $t\geq0.$ Since Lemma \ref{prelim_deltanuABS_lemma} and the calculation of other terms on the right-hand side of \eqref{prelim_I2} are valid for all $\varphi\in\mathbb{C}_b(\supint\times\hc),$ \eqref{prelim_dymamics_balance} follows on setting $\varphi=\f1$ in \eqref{prelim_dymamics_eq}.
\end{proof}

\begin{remark}\label{rem_dynamic_extension}
Equation \eqref{prelim_dymamics_eq} clearly remains valid for functions $\varphi$ on $\supint\times\R_+$ of the form $\varphi(x,s)=f(x)$ for some $f\in\mathbb{C}_c^1\supint$.
\end{remark}


\section{Martingale Decomposition for Routing and Departure Processes}\label{sec_mgale}
Fix $N \in \mathbb{N}$.
In Section \ref{sec_prelimit_summary} we state a martingale decomposition result for the $\varphi$-weighted routing process $\Rn_{\varphi,\ell}$ and departure process $\Quen_{\varphi,\ell}$ defined in \eqref{def_Rphi} and \eqref{def_Qphi}, respectively. The proofs are given in Section \ref{sec_Rcomp} and Appendix \ref{sec_Qcomp}, and rely on an alternative characterization of the  filtration $\{\filt_t^{(N)}\}$ in terms of a marked point process introduced in Section \ref{sec_mpp}. Unlike Proposition \ref{prelim_dynamics_prop}, these results are specific to our assumptions on the arrival process and load balancing algorithm, although the general method can be adapted to analyze  other models.

\subsection{The Form of Compensators}\label{sec_prelimit_summary}

For $\varphi\in\tightphiset$, define
\begin{equation}\label{def_Rcomp1}
 B_{\varphi,1}^{(N)}(t) \doteq \int_0^t\hen\left(\ren(u)\right)\varphi(0,u)  (1- (\nunonebar_1(u))^d)du,
\end{equation}
where  $\hen$ is the hazard rate  of the inter-arrival distribution, and
 for $\ell \geq 2$, set
\begin{equation}\label{def_Rcomp}
  B_{\varphi,\ell}^{(N)}(t) \doteq \int_0^t \hen\left(\ren(u)\right) \poly{\nunonebar_{\ell-1}(u)}{\nunonebar_{\ell}(u)} \left\langle \varphi(\cdot, u)\;, \nunbar_{\ell-1}(u)-\nunbar_\ell(u)\right\rangle du.
\end{equation}

\begin{proposition}\label{prop_prelim_Rcomp}
  Suppose Assumptions \ref{asm_E}, \ref{asm_G}.\ref{asm_mean}, and \ref{asm_initial}.\ref{asm_initial_ind} hold. Then, for $\ell\geq 1$ and $\varphi\in\tightphiset$, the process
  \begin{equation}\label{def_Rmgale}
    \Nn_{\varphi,\ell} \doteq \Rphinl-\Bphinl,
  \end{equation}
  is a local $\{\filtn_t\}$-martingale, with   quadratic variation
\begin{equation}\label{QuadvarN}
  [\Nn_{\varphi,\ell}](t) = \Rn_{\varphi^2,\ell}(t),\quad\quad t\geq0.
\end{equation}
\end{proposition}
\noindent  Proposition \ref{prop_prelim_Rcomp} is a key result and its proof is given in Section \ref{sec_Rcomp}. Next, for $\varphi\in\tightphiset$, define
\begin{equation}\label{Def_Qcomp}
  A_{\varphi,\ell}^{(N)}(t) \doteq \int_0^t \langle \varphi(\cdot,s)h(\cdot) ,\nuln(s)\rangle ds,\quad\quad \forall t\geq 0.
\end{equation}

\begin{proposition}\label{prop_prelim_Qcomp}
Suppose Assumptions \ref{asm_E}, \ref{asm_G}.\ref{asm_mean}, and  \ref{asm_initial}.\ref{asm_initial_ind} hold. Then, for  $\ell \geq 1$ and $\varphi\in\tightphiset$,  the process
\begin{equation}\label{def_Mn}
  \Mn_{\varphi,\ell}\doteq\Quen_{\varphi,\ell} - \An_{\varphi,\ell},
\end{equation}
is a local $\{\filtn_t\}$-martingale,  with quadratic variation
\begin{equation}\label{QuadvarM}
  [\Mn_{\varphi,\ell}](t)=\Quen_{\varphi^2,\ell}(t),\quad\quad t\geq0.
\end{equation}
\end{proposition}
Since the proof of Proposition \ref{prop_prelim_Qcomp} is similar to (in fact much simpler than) that of Proposition \ref{prop_prelim_Rcomp}, it is relegated to Appendix \ref{sec_Qcomp}. A similar result  for a different model and  filtration, can also be found in \cite[(5.24), (5.25) and Lemma 5.4]{KasRam11}.

\begin{remark}
Substituting \eqref{def_Rmgale}, \eqref{def_Mn}, and  \eqref{Def_Qcomp} into   \eqref{prelim_dymamics_eq}, we have
\begin{align}\label{prelimDynamicAlt}
  \langle\varphi(\cdot,t),\nun_\ell(t)\rangle =& \langle\varphi(\cdot,0),\nun_\ell(0)\rangle +\int_0^t \langle \varphi_s(\cdot,s)+\varphi_x(\cdot,s)-\varphi(\cdot,s)h(\cdot),\nun_\ell(s)\rangle ds \notag\\
  &+ \int_{[0,t]}\varphi(0,s)d\dn_{\ell+1}(s)+ \Bn_{\varphi,\ell}(t) -\Mn_{\varphi,\ell}(t)+\Nn_{\varphi,\ell}(t).
\end{align}
\end{remark}

We now state an elementary lemma used in  the proof of Proposition \ref{prop_prelim_Rcomp}. Suppose $(\Omega,\mathcal{G},\{\mathcal{G}_t\},\mathbb{P})$ is a filtered probability space that satisfies the usual conditions, and let $\xi = \{\xi(t);t\geq0\}$ be a point process adapted to $\{\cal G_t\}$.  Recall that a non-negative $\{\cal G_t\}$-progressive process $\{\inten (t);t\geq0\}$ is called a $\{\cal G_t\}$-intensity of $\xi$ if for all $t\geq0$, $\int_0^t \inten(s)ds<\infty$ almost surely, and for every non-negative $\{\cal G_t\}$-predictable processes $H$, $ \Eptil{\int_0^\infty H(t)d\xi(t)}=\Eptil{\int_0^\infty H(s) \inten(s)ds}.$ The next result follows from  \cite[Lemma II.L3]{BremaudBook} and see \cite[(18.1), Chapter IV]{RogWilbook2}, but a proof is provided here for completeness.

\begin{lemma}\label{remark_integral}
Let $\inten$  be a $\{\mathcal{G}_t\}$-intensity of a point process $\xi$ on $(\Omega,\mathcal{G},\{\mathcal{G}_t\},\mathbb{P})$, and given a locally bounded, $\{\cal G_t\}$-predictable process $\theta$, define $\zeta(t)\doteq\int_0^t \theta(s)d\xi(s).$ Then  $\zeta(t)-\int_0^t \theta(s) \inten(s)ds, t\geq 0,$ is a local $\{\cal G_t\}$-martingale, with  quadratic variation
\begin{equation}\label{remark_Quad}
  [\zeta](t)=\int_0^t \theta^2(s)d\xi(s), \quad\quad t\geq0.
\end{equation}
\end{lemma}

\begin{proof}
First note that by local boundedness of $\theta$, for all $t\geq0$,
\[\int_0^t|\theta(s)|\inten (s)ds\leq \sup_{s\in[0,t]}|\theta(s)|\int_0^t\inten(s)ds<\infty,\quad\quad a.s,\]
where the last inequality is by definition of stochastic intensity. The first claim then follows from \cite[Lemma II.L3]{BremaudBook}. For the second claim, since $\xi$ is a pure jump process with unit jumps,  $\zeta$ is also a pure jump process with jumps
\begin{equation}\label{temp_intensity}
  \zeta(s)-\zeta(s-)=\theta(s)(\xi(s)-\xi(s-)).
\end{equation}
Therefore, the quadratic variation of $\zeta$ is given by (see \cite[(18.1), Chapter IV]{RogWilbook2})
\begin{equation}\label{temp_quadlemma}
   [\zeta](t)= \sum_{0<s\leq t}\left(\zeta(s)-\zeta(s-)\right)^2 = \sum_{0<s\leq t}\theta^2(s)(\xi(s)-\xi(s-))= \int_0^t \theta^2(s)d\xi(s).
 \end{equation}
\end{proof}

\subsection{A Marked Point Process Representation}\label{sec_mpp}

In this section, we construct a point process $\MPn$  consisting of all arrival and departure times, marked  by their type and their corresponding queue index.   This point process has the property that its natural filtration, together with the $\sigma$-algebra generated by initial conditions, is equivalent to the filtration $\{\filtn_t; t\geq0\}$ defined in \eqref{def_filt}.  Moreover, each auxiliary process defined in Section \ref{sec_AuxProcess} can be represented as an integral with respect to $\MPn$, which allows us to more easily identify its compensator.

Consider the set
\[
    \mathbb{T}^{(N)}\doteq \left\{(\arrivalTime_j,(\mathfrak{E},\stationn_j));j\geq1\right\} \cup\left\{(\deptn_j,(\mathfrak{D},\stationn_j));j\geq \j0\right\},
\]
which is the union of all arrival times $\arrivalTime_j$, marked by the tag $\mathfrak{E}$ (indicating that it is an arrival time) and the index of the queue to which job $j$ is routed, and all departure times $\deptn_j$, marked by the tag $\mathfrak{D}$ (indicating that it is a departure time) and the index of the queue from which job $j$ departed.   Since the inter-arrival and service distributions $\Gen$ and $G$ are absolutely continuous with respect to Lebesgue measure, by Assumptions \ref{asm_E} and  \ref{asm_G}.\ref{asm_mean},  almost surely at most one arrival to and at most one departure from each queue can occur at any given time.   Let $\eventtime_0\doteq0$ and  $\eventmark_0$ be a constant (whose value is irrelevant), and define the sequence of \textit{events} $\{(\eventtimen_k,\eventmarkn_k);k\geq1\}$, each composed of an \textit{event time} $\eventtimen_k$ and an \textit{event mark} $\eventmarkn_k$, to be the relabeling (i.e., a one-to-one correspondence) of $\mathbb{T}^{N}$ sorted by lexicographic order, assuming $\mathfrak{D}<\mathfrak{E}$. That is, events are ordered first by event times ($\eventtimen_k\leq\eventtimen_{k+1}$), then by event type (departure first, then arrival) and finally by queue index (with smaller indices first).  Let $\MPn=\{\MPn(t);t\geq0\}$ be the corresponding marked point process. Clearly, for every index $i\in\{1,...,N\}$,
\begin{equation}\label{Ei_mpp_rep}
  \MPn(\mathfrak{E},i;t)\doteq\sum_{k\geq1}\indic{\eventtimen_k\leq t} \indic{\eventmarkn_k=(\mathfrak{E},i)}=\arrivalni(t),
\end{equation}
and
\begin{equation}\label{Di_mpp_rep}
  \MPn(\mathfrak{D},i;t)\doteq\sum_{k\geq1}\indic{\eventtimen_k\leq t} \indic{\eventmarkn_k=(\mathfrak{D},i)}=\dni(t).
\end{equation}
 These relations show that the filtration $\{\filtn_t\}$ in  \eqref{def_filt}
has the representation
\begin{equation}\label{def_filt_alt}
  \filtn_t=\sigma(I_0)\vee\mathcal{F}^{\mathcal{T},(N)}_t,\quad\quad t\geq0,
\end{equation}
where $\{\mathcal{F}^{\mathcal{T},(N)}_t\}$ is the filtration generated by the marked point process $\MPn$ \cite[p.\ 57, eqn.\ (1.2)]{BremaudBook}.

At any time $t$,  server $i$ is called \textit{busy} if $\xni(t)\geq1$, and is called \textit{idle} otherwise.  By the non-idling assumption, there is a job receiving service at any busy server $i$ at time $t$, and  we define $\ageni(t)$ to be the age of that job.  Using this notation, for $ \ell \geq 1$, we can rewrite the definition \eqref{nuni} of $\nuni_\ell$ as
\begin{equation}\label{nuni_alt}
  \nuni_\ell(t)=\indic{\xni(t)\ge \ell}\delta_{\ageni(t)},
\end{equation}
which when combined with \eqref{nun_nuni}, yields
\begin{equation}\label{nun_nuni_alt}
  \nun_\ell(t)=\sum_{i=1}^N\indic{\xni(t)\ge \ell}\delta_{\ageni(t)}.
\end{equation}
For $k\geq0$,  define
\begin{equation}\label{def_busy}
  \busyn_k\doteq \left\{i: \xni(\eventtimen_k)\geq1 \right\}
\end{equation}
to be the set of busy servers at time $\eventtimen_k$, and note that for $i\in\busyn_k$, $\ageni(\eventtimen_k)$ is well-defined. Define $\xi^{(N)}_{k+1}$ to be the next arrival time strictly after $\eventtimen_k$, and for $i=1,...,N$, define $\sigma_{k+1}^{(N),i}$ to be the  next time strictly after $\eventtimen_k$ when there is a departure from queue $i$ if $i\in\busyn_{k}$, and $\sigma_{k+1}^{(N),i}=\infty$, otherwise. When the event time $\eventtimen_k$ is distinct, that is $\eventtimen_k\neq\eventtimen_{k'}$ for all $k'\neq k$, the next event time will be the minimum among the first arrival time after $\eventtime_k$ and the next departure time from  queues that  are busy at time $\eventtime_k$. Therefore, defining
\begin{equation}\label{def_Omegak}
      \tilde\Omega_k\doteq\{\omega\in\Omega; \eventtime_{k}\neq\eventtime_{k'},\quad \forall k'\neq k\},
\end{equation}
to be the set of realizations on which the event time $\eventtimen_k$ is distinct, we have
\begin{equation}\label{next_event}
  \eventtimen_{k+1}= \min\left(\xi^{(N)}_{k+1},\sigma^{(N),i}_{k+1};i=1,...,N\right)\quad\quad\text{ on }\tilde\Omega_k.
\end{equation}

The next lemma identifies the joint distribution of the next arrival and departure times given $\filtn_{\eventtimen_k}$.

\begin{lemma}\label{lem_condInd}
Suppose Assumptions \ref{asm_E}, \ref{asm_G}.\ref{asm_mean}, and \ref{asm_initial}.\ref{asm_initial_ind} hold. Then, for $k \geq 0$,
$\Probil{\tilde\Omega_k}=1$,  $\xi^{(N)}_{k+1}$ and $\sigma^{(N),i}_{k+1}$, $i=1,...,N,$ are conditionally independent given $\filtn_{\eventtimen_k}$,  and
\begin{equation}\label{cond_density_arrival}
  \Prob{\xi^{(N)}_{k+1}-\eventtimen_k>b\big|\filtn_{\eventtimen_k}} = \frac{\bGen(\ren(\eventtimen_k)+b)} {\bGen(\ren(\eventtimen_k))}, \quad  b > 0,
\end{equation}
and for $i=1,...,N,$
\begin{equation}\label{cond_density_departure}
    \indic{i\in\busyn_k}\Prob{\sigma^{(N),i}_{k+1}-\eventtimen_k>b\big|\filtn_{\eventtimen_k}} = \indic{i\in\busyn_k}\frac{\overline G(\ageni(\eventtimen_k)+b)}{\overline G(\ageni(\eventtimen_k))}, \quad  b > 0.
\end{equation}
\end{lemma}
The result in Lemma \ref{lem_condInd} is intuitive and  follows from the independence of the interarrival and service times. However, a completely rigorous proof is rather involved and technical, although involving fairly routine calculations. Hence, we defer the proof to Section \ref{secapx_condind}. Using Lemma \ref{lem_condInd}, we can rewrite \eqref{next_event} as
\begin{equation}\label{next_event_new}
  \eventtimen_{k+1}= \min\left(\xi^{(N)}_{k+1},\sigma^{(N),i}_{k+1};i=1,...,N\right),\quad\quad\text{ a.s.}
\end{equation}
We now state a consequence of Lemma \ref{lem_condInd}. Recall that a sequence $\{t_n;n\in\N\}$ is called non-explosive if for every $T<\infty,$ there are finitely many $n$ with $t_n\leq T.$

\begin{corollary}\label{cor_distinctness}
Suppose Assumptions \ref{asm_E},  \ref{asm_G}.\ref{asm_mean}, and  \ref{asm_initial}.\ref{asm_initial_ind} hold. Then, almost surely, the sequence of event times $\{\eventtimen_k;k\geq0\}$ is non-explosive, and  $\Probil{\Omega_t}=1$,  $t \geq 0$.
\end{corollary}
\begin{proof}
Fix $t\ge 0$, define $\hat \Omega_t=\hat{\Omega}^{(N)}_t = \{\omega\;: \xn(0)<\infty,\arrivaln(t)<\infty\}.$ Since $\xn(0)<\infty$ almost surely by Assumption \ref{asm_initial}.\ref{asm_initial_ind} and  $\arrivaln$ is a renewal process with non-degenerate interarrival time distribution by Assumption \ref{asm_E}, $\mathbb{P}\{\hat \Omega_t\}=1$. Moreover $\dn(t)$ is also finite on $\hat \Omega_t$ by \eqref{apx_massTotal}. Thus, on $\hat \Omega_t$, and hence, almost surely, the total number of events up to $t$ is finite and $\{\eventtimen_k\}$ is  non-explosive.

Moreover, the set $\tilde \Omega\doteq\cup_{k\geq0}\tilde \Omega_k$ of realizations on which all events are distinct has full measure by Lemma \ref{lem_condInd}. For every $\omega\in\hat\Omega_t\cap\tilde\Omega$,  the quantity $\Delta(\omega) = \Delta^{(N)}(\omega)\doteq \inf_{k: \eventtimen_k \leq t}\left(\eventtimen_{k+1}-\eventtimen_k\right),$ is strictly positive because it is the infimum of  finitely many positive numbers.  This means that for $n> 1/\Delta$, the distance between any two events prior to time $t$ exceeds $1/n$. Therefore,  $\omega\in\Omega_{\frac{k}{n},\frac{1}{n}}$ for all $k=0,...,\lfloor nt\rfloor$, and hence $\omega\in\Omega_t$. This implies $\hat\Omega_t\cap\tilde\Omega  \subseteq \Omega_t$ and hence,  $\Probil{\Omega_t}=1$.
\end{proof}

\subsection{Compensator for the Weighted Routing Measure}\label{sec_Rcomp}

We state our first result.
\begin{lemma}\label{prelim_Ecomp_lemma}
Suppose Assumptions \ref{asm_E},  \ref{asm_G}.\ref{asm_mean}, and  \ref{asm_initial}.\ref{asm_initial_ind} hold. Then, for $i=1,...,N$,  the process
$\arrivalni$ defined in \eqref{def_Ei} has  the following $\{\filtn_t\}$-intensity process:
\begin{equation}\label{Ecomp}
    \left\{\frac{1}{N}\hen\left(\ren(t-)\right)\sum_{\ell=1}^\infty  \indic{\xni(t-)=\ell-1} \poly{\nunonebar_{\ell-1}(t-)}{\nunonebar_\ell(t-)};t\geq0\right\}.
\end{equation}
\end{lemma}
\begin{proof}
Throughout this proof, we omit  the superscript $(N)$ for  ease of notation. Suppose that for  $k\geq0$ and $i=1, \ldots, N$,  the conditional density$f^{\mathfrak{E},i}_{k+1}$ defined by
\begin{equation}\label{prelim_gek_eq}
\Prob{\eventtime_{k+1}-\eventtime_{k}\in A,\eventmark_{k+1}=(\mathfrak{E},i)\big|\filt_{\eventtime_k}}= \int_A f^{\mathfrak{E},i}_{k+1}(\omega,r) dr,\quad \omega\in\Omega,\; A\in\mathcal{B}\hc,
\end{equation}
exists. Then, by the representation  in \eqref{Ei_mpp_rep} of $\arrivali$ in terms of $\MPn$,
and the fact that the sequence $\{\eventtime_k\}$ is non-explosive by Corollary \ref{cor_distinctness}, it follows from \cite[Theorem III.T7, comment ($\beta$) and (2.10)]{BremaudBook} that the process
\begin{equation}\label{prelim_AintensityForm_EQ}
  \sum_{k=0}^\infty \frac{ f^{{\mathfrak{E}},i}_{k+1}(\omega, t-\eventtime_k)}{\Prob{\eventtime_{k+1}>t|\filt_{\eventtime_k}}}\indic{\eventtime_k<t\leq \eventtime_{k+1}}.
\end{equation}
is an $\{\filt_t\}$-intensity of $\arrivali$.

We now show that $f^{\mathfrak{E},i}_{k+1}$ in \eqref{prelim_gek_eq} exists. First, note that by  \eqref{next_event_new}, the next event after $\eventtime_k$ is an arrival to queue $i$, that is,  $\eventmark_{k+1}=(\mathfrak{E},i)$, if the next arrival occurs before the next departure from any queue and the arriving job, which has index $E(\eventtime_{k+1})$, is routed to queue $i$. Hence,  defining $\sigma_{k+1}\doteq\min(\sigma_{k+1}^{i'};i'=1,...,N),$ we have
\begin{align}\label{prelim_gtemp1}
&\Prob{\eventtime_{k+1}-\eventtime_{k}>t,\eventmark_{k+1}=(\mathfrak{E},i)\big|\filt_{\eventtime_k}} \notag\\ &\hspace{1cm}=\Prob{\xi_{k+1}>\eventtime_k+t,\sigma_{k+1}>\xi_{k+1},\station_{E(\eventtime_{k+1})}=i \big|\filt_{\eventtime_k}}\notag\\
&\hspace{1cm}= \Ept{\indic{\xi_{k+1}>\eventtime_k+t,\sigma_{k+1}>\xi_{k+1}} \Prob{\station_{E(\eventtime_{k+1})}=i|\filt_{\eventtime_k},\xi_{k+1},\sigma_{k+1}} \Big|\filt_{\eventtime_{k}}}.
\end{align}
The queue to which  job $j=E(\eventtime_{k+1})$ is routed is given by \eqref{def_station}, which is a (deterministic) function of the random queue choices vector $\rqc_j$ and queue lengths at time $\eventtime_k$. According to \eqref{def_station}, this job is routed to a queue of length exactly $\ell-1$, if and only if all selected queue indices $\rqc_j$ have lengths at least $\ell-1$, and at least one of them has length exactly $\ell-1$.  Since $\rqc_j$ is independent of all other  random variables, the conditional probability (given $\filt_{\eventtime_k},\sigma_{k+1}$ and $\xi_{k+1})$ that the job is routed to a queue of length $\ell-1$ is $(\nuonebar_{\ell-1}(\eventtime_k))^d-(\nuonebar_\ell(\eventtime_k))^d$. Moreover, the job is equally  likely to be routed to any queue of length $\ell-1$. Since there are  $\nuone_{\ell-1}(\eventtime_k)-\nuone_\ell(\eventtime_k)$ such queues,  on the event ${\xii(\eventtime_k)=\ell-1}$, we have
 \begin{align*}
\Prob{\station_{E(\eventtime_{k+1})}=i|\filt_{\eventtime_k},\xi_{k+1},\sigma_{k+1}} & =
\frac{\left(\nuonebar_{\ell-1}(\eventtime_k)\right)^d-\left(\nuonebar_\ell(\eventtime_k)\right)^d}{\nuone_{\ell-1}(\eventtime_k)-\nuone_\ell(\eventtime_k)} \\
   & =
\frac{1}{N}\poly{\nuonebar_{\ell-1}(\eventtime_k)}{ \nuonebar_\ell(\eventtime_k)},
 \end{align*}
where the polynomial $\mathfrak{P}_d$ is defined in \eqref{def_poly}. Therefore,
\begin{align}\label{prelim_gtemp2}
\Prob{\station_{E(\eventtime_{k+1})}=i|\filt_{\eventtime_k},\xi_{k+1},\sigma_{k+1}}
&
= \sum_{\ell=1}^\infty \indic{\xii(\eventtime_k)=\ell-1} \Prob{\station_{E(\eventtime_{k+1})}=i|\filt_{\eventtime_k},\xi_{k+1},\sigma_{k+1}} \notag \\
&
=\frac{1}{N}\sum_{\ell=1}^\infty \indic{\xii(\eventtime_k)=\ell-1}\poly{\nuonebar_{\ell-1}(\eventtime_k)}{ \nuonebar_\ell(\eventtime_k)}.
\end{align}
Moreover, using \eqref{cond_density_arrival} and Lemma \ref{lem_condInd}, we have
\begin{equation}\label{prelim_gtemp3}
\Prob{\xi_{k+1}>\eventtime_k+t,\sigma_{k+1}>\xi_{k+1}|\filt_{\eventtime_k}}  =
\frac{1}{\overline G_E(\re(\eventtime_k))}\int_t^\infty\Prob{\sigma_{k+1}-\eventtime_k>s|\filt_{\eventtime_k}} \gee(s+\re(\eventtime_k))ds.
\end{equation}
Therefore, by  \eqref{prelim_gtemp1}-\eqref{prelim_gtemp3} and the fact that the right-hand side of \eqref{prelim_gtemp2} is $\filt_{\eventtime_{k}}$-measurable, for $k\in\Z_+$, $i\in\{1,...,N\}$, the density $f^{\mathfrak{E},i}_{k+1}$ exists, and for $t\geq \eventtime_k$,
\begin{align}\label{prelim_gE}
  &f^{\mathfrak{E},i}_{k+1}(t-\eventtime_k)= \Prob{\sigma_{k+1}>t|\filt_{\eventtime_k}} \frac{\gee(t-\eventtime_k+\re(\eventtime_k))}{N\;\overline G_E(\re(\eventtime_k))}\sum_{\ell=1}^\infty \indic{\xii(\eventtime_k)=\ell-1} \poly{\nuonebar_{\ell-1}(\eventtime_k)}{ \nuonebar_\ell(\eventtime_k)}.
\end{align}
Similarly, by \eqref{next_event_new}, Lemma \ref{lem_condInd} and \eqref{cond_density_arrival}, for every $k\ge0$ and $t\geq \eventtime_k$ we have
\begin{align}\label{prelim_Gbar}
  \Prob{\eventtime_{k+1}>t|\filt_{\eventtime_k}} = \Prob{\xi_{k+1}\wedge \sigma_{k+1}>t |\filt_{\eventtime_k}} \notag
  &= \Prob{ \sigma_{k+1}>t |\filt_{\eventtime_k}}\Prob{ \xi_{k+1}>t |\filt_{\eventtime_k}}\notag\\
  &=   \Prob{\sigma_{k+1}>t|\filt_{\eventtime_k}}\frac{\overline G_E(t-\eventtime_k+\re(\eventtime_k))}{\overline G_E(\re(\eventtime_k))}.
\end{align}

Combining \eqref{prelim_gE} and \eqref{prelim_Gbar}  with \eqref{prelim_AintensityForm_EQ}, recalling that $\he=\gee/\bGe$ is the hazard rate function of the interarrival times, and noting that on  $(\eventtime_k,\eventtime_{k+1}]$,
$\xii$, $\nuone_\ell$  are constant and $R_E$ has a unit slope, we see that
\begin{align}\label{prelim_SItemp}
 \frac{1}{N}\sum_{k=0}^\infty & \he(t-\eventtime_k+\re(\eventtime_k)) \sum_{\ell=1}^\infty \indic{\xii(\eventtime_k)=\ell-1} \poly{\nuonebar_{\ell-1}(\eventtime_k)}{ \nuonebar_\ell(\eventtime_k)} \indic{\eventtime_k < t \leq \eventtime_{k+1}}\notag\\
 &= \frac{1}{N}\sum_{k=0}^\infty \he(\re(t-)) \sum_{\ell=1}^\infty \indic{\xii(t-)=\ell-1} \poly{\nuonebar_{\ell-1}(t-)}{ \nuonebar_\ell(t-)} \indic{\eventtime_k < t \leq \eventtime_{k+1}}\notag\\
 &= \frac{1}{N}\he(\re(t-)) \sum_{\ell=1}^\infty \indic{\xii(t-)=\ell-1} \poly{\nuonebar_{\ell-1}(t-)}{ \nuonebar_\ell(t-)},\quad\quad t\geq0,
\end{align}
is an $\{\filt_t\}$-intensity of $\arrivalni$.
\end{proof}

We now use Lemma \ref{prelim_Ecomp_lemma} to  prove the martingale decomposition for $\Rn_{\varphi,\ell}$.

\begin{proof}[Proof of Proposition \ref{prop_prelim_Rcomp}]
First, note  Note from \eqref{def_Rphi} that
\begin{equation}\label{def_Rni_decomp}
\Rn_{\varphi,\ell}=\sum_{i=1}^N \Rni_{\varphi,\ell},
\end{equation}
where
\begin{equation}\label{def_Rni}
  \Rni_{\varphi,\ell}(t)\doteq \twopartdef{  \displaystyle \int_{(0,t]} \varphi(0,s)\indic{\xni(s-)=0}d\arrivalni(s)}{\ell=1,}{\displaystyle \int_{(0,t]} \langle \varphi(\cdot,s),\nuni_{\ell-1}(s-) - \nuni_{\ell}(s-) \rangle d\arrivalni(s)}{\ell \geq 2.}
\end{equation}

We first prove the result for the case $\ell=1$. Consider the setup of Lemma \ref{remark_integral} with $\xi = \arrivalni$, $\cal G_t = \filtn_t$, $\theta(s) =$  $\varphi(0,s)\indicil{\xni(s-)=0}$, and hence, $\zeta = \Rni_{\varphi,1}$. Since $\varphi$ is bounded, continuous and $\xni$ is right-continuous and $\{\filtn_t\}$-adapted (see Proposition \ref{prop_adapted}), $\theta$ is bounded, $\{\filtn_t\}$-adapted , and left-continuous, and thus, $\{\filtn_t\}$-predictable. Therefore, using the fact that  the expression in \eqref{Ecomp} is an intensity of $\arrivalni,$ and by the identity
\begin{align}\label{temp_Rcomp}
  &\frac{1}{N}\varphi(0,t)\indic{\xni(t-)=0}\hen(\ren(t-))\sum_{\ell^\prime=1}^\infty \indic{\xni(t-)=\ell^\prime-1} \poly{\nunonebar_{\ell^\prime-1}(t-)}{\nunonebar_{\ell^\prime}(t-)}\notag\\
 &\hspace{1cm} =\frac{1}{N}\varphi(0,t)\indic{\xni(t-)=0}\hen(\ren(t-)) \poly{1}{\nunonebar_1(t-)},
\end{align}
which holds since  $\mathbb{1}(\xni(t-)=0)\mathbb{1} (\xni(t-)=\ell^\prime-1) \neq 0$ if and only if  $\ell^\prime =1$, and $\overline{S}_0^{(N)} \equiv 1$, it follows from Lemma \ref{remark_integral} that the process $\{\Nni_{\varphi,1}(t);t\geq0\}$ defined as
 $\Nni_{\varphi,1} = \Rni_{\varphi, 1} - \tRni_{\varphi,1}$, where
\begin{align}\label{temp_Rcomp2}
   \tRni_{\varphi,1}(t) \doteq  \frac{1}{N} \int_0^t  \varphi(0,s)\indic{\xni(s)=0}\hen(\ren(s)) \poly{1}{\nunonebar_1(s)}ds,
\end{align}
is a local $\{\filtn_t\}$-martingale (note that the right-hand side of \eqref{temp_Rcomp} and the integrand of the display above only differ on a zero-measure subset of $[0,t]$). Moreover, since the number of idle servers at time $t$ is $N-\nunone_1(t)$, and by definition \eqref{def_poly} of $\mathfrak{P}_d$ and definition \eqref{def_Rcomp} of $\Bn_{\varphi,1}$, we have
\begin{align}\label{temp_Rcomp3}
 \sum_{i=1}^N \tRni_{\varphi,1} (t)  &=\frac{1}{N}\int_0^t\varphi(0,s)\hen(\ren(s))  \poly{1}{\nunonebar_1(s)}\sum_{i=1}^N \indic{\xni(s)=0}ds \notag\\
&=\int_0^t\varphi(0,s)\hen(\ren(s))  \left(1- \big(\nunonebar_1(s)\big)^d\right)ds\notag\\
&= \Bn_{\varphi,1}(t).
\end{align}
Then, since $\sum_{i=1}^N\Nni_{\varphi,1}=\Rn_{\varphi,1}-\Bn_{\varphi,1}=\Nn_{\varphi,1},$  it follows  that $\Nn_{\varphi,1}$ is a local $\{\filtn_t\}$-martingale. Moreover, by \eqref{remark_Quad} of Lemma \ref{remark_integral} in the setup described above,
\begin{align*}
  [\Rni_{\varphi,1}](t)= \int_0^t \varphi^2(0,s)\indic{\xni(u-)=0}d\arrivalni(u)=\Rni_{\varphi^2,1}(t).
 \end{align*}
Also, for $i=1,...,N$  and $i'\neq i$, $\RnX{i}_{\varphi,1}$ and $\RnX{i'}_{\varphi,1}$ are pure jump processes with no common jump times almost surely (see Lemma \ref{lem_condInd}), and hence, $[\Rni_{\varphi,\ell}]$ is locally bounded and $[ \RnX{i}_{\varphi,1},\RnX{i'}_{\varphi,1}] \equiv 0,$ almost surely. Together with the fact that $\Bn_{\varphi,1}=\Rn_{\varphi,1}-\Nn_{\varphi,1}$ is a continuous function with finite variation, this implies
\[ [\Nn_{\varphi,1}] = [\Rn_{\varphi,1}] = [\sum_{i=1}^N\Rni_{\varphi,1}] = \sum_{i=1}^N[\Rni_{\varphi,1}] =\sum_{i=1}^N\Rni_{\varphi^2,1}=\Rn_{\varphi^2,1}. \]

Similarly, for $\ell\geq2$, consider the setup of Lemma \ref{remark_integral} with $\xi$, $\{\cal G_t\}$ as above,
but with $\theta(s)$ replaced by
\[\langle \varphi(\cdot,s),\nuni_{\ell-1}(s-) - \nuni_{\ell}(s-)\rangle= \varphi(\ageni(s-),s)\indic{\xni(s-)=\ell-1} \]
and hence, $\zeta$ replaced by $\Rni_{\varphi,\ell}$. Since $\ageni$ and $\xni$ are $\{\filtn_t\}$-adapted, $\theta$ is bounded, $\{\filtn_t\}$-adapted, and left-continuous. Thus, using the fact that \eqref{Ecomp} is an  intensity of $\arrivalni$, and by the identity
\begin{align*}
    &\frac{1}{N} \varphi\left(\ageni(t-),t\right)\indic{\xni(t-)=\ell-1} \hen(\ren(t-)) \sum_{\ell^\prime=1}^\infty \indic{\xni(t-)=\ell^\prime-1} \poly{\nunonebar_{\ell^\prime-1}(t-)}{\nunonebar_{\ell^\prime}(t-)}\\
    &\hspace{2cm}=\frac{1}{N} \varphi\left(\ageni(t-),t\right)\indic{\xni(t-)=\ell-1} \hen(\ren(t-))  \poly{\nunonebar_{\ell-1}(t-)}{\nunonebar_{\ell}(t-)},
\end{align*}
which holds since  $\indic{\xni(t-)=\ell-1}\indic{\xni(t-)=\ell^\prime-1} \neq 0$ if and only if  $\ell^\prime =\ell$,
it follows from Lemma \ref{remark_integral} that the process
$\Nni_{\varphi,\ell} \doteq \Rni_{\varphi,\ell}  - \tRni_{\varphi,\ell},$ with
\begin{align*}
  \tRni_{\varphi,\ell}(t)\doteq  \frac{1}{N} \int_0^t  \varphi\left(\ageni(s),s\right)\indic{\xni(s)=\ell-1} \hen(\ren(s))  \poly{\nunonebar_{\ell-1}(s)}{\nunonebar_{\ell}(s)}ds,
\end{align*}
is a local $\{\filtn_t\}$-martingale. Again, by \eqref{def_Rcomp} and \eqref{nuni_alt},
\begin{align*}
 \tRni_{\varphi,\ell}(t)
  &
=\frac{1}{N}\int_0^t \hen(\ren(s)) \poly{\nunonebar_{\ell-1}(s)}{\nunonebar_{\ell}(s)}\sum_{i=1}^N \varphi(\ageni(s),s)\indic{\xni(s)=\ell-1}ds\\
  &= \Bn_{\varphi,\ell}(t).
\end{align*}
Using \eqref{def_Rni_decomp}, it follows that
$\sum_{i=1}^N\Nni_{\varphi,\ell}=\Rn_{\varphi,\ell}-\Bn_{\varphi,\ell}=\Nn_{\varphi,\ell},$
and hence, $\Nn_{\varphi,\ell}$ is also a local $\{\filtn_t\}$-martingale. Also, in  the setup above, we have
\[\theta^2(s)= \varphi^2(\ageni(s-),s)\indic{\xni(s-)=\ell-1}=\langle \varphi^2(\cdot,s),\nuni_{\ell-1}(s-) - \nuni_{\ell}(s-)\rangle, \]
and so \eqref{remark_Quad} of Lemma \ref{remark_integral} implies
\begin{align*}
  [\Rni_{\varphi,\ell}](t)=\int_0^t \langle \varphi^2(\cdot,s),\nuni_{\ell-1}(s-) - \nuni_{\ell}(s-) \rangle d\arrivalni(s)=\Rni_{\varphi^2,\ell}.
 \end{align*}
Finally, by the same reasoning as in the case $\ell=1$, for $\ell\geq 2$ we have
$[\Nn_{\varphi,\ell}] = [\Rn_{\varphi,\ell}] = \Rn_{\varphi^2,\ell}.$
This completes the proof.
\end{proof}


\section{Proof of Convergence Results}\label{sec_convergence}

In Section \ref{sec_tightness} we  prove  relative compactness of the state and auxiliary processes.  In Section \ref{sec_limit}, we first show that subsequential limits of the state processes satisfy the hydrodynamic equations, and then prove Theorem \ref{thm_convergence} and Corollary \ref{cor_PropofChaos} in Section \ref{sec_limitCharacterization}. For $H={\mathcal D}_{\varphi,\ell}, {\mathcal A}_{\varphi,\ell}, {\mathcal M}_{\varphi,\ell}, {\mathcal R}_{\varphi,\ell}$, ${\mathcal B}_{\varphi,\ell}$, and  ${\mathcal N}_{\varphi,\ell}$, let
\begin{equation}\label{scaling_DR}
\overline{H}^{(N)} (t) \doteq \frac{H^{(N)}(t)}{N}, \quad N \in \N, t \geq 0.
\end{equation}

\subsection{Relative Compactness}\label{sec_tightness}

The relative compactness results are  summarized in Theorem \ref{thm_tightness} of Section \ref{sec_tightState}. They are established  in Sections \ref{sec_tightDeparture}--\ref{sec_tightState} by verifying the well-known criteria of Kurtz and Jakubowski summarized in Section \ref{sec_tightCriteria}.

\subsubsection{Review of Relative Compactness Criteria}\label{sec_tightCriteria}

Recall from Section \ref{sec_introNotation} that $w^\prime(f,\cdot,\cdot)$ denotes the  modulus of continuity of a function $f$ in $\D$.  The first result follows from \cite[Theorem 3.7.2]{KurtzBook} and  \cite[Theorems 3.7.2 and 3.8.6 and Remark 3.8.7]{KurtzBook}.

\begin{proposition}[Kurtz's criteria]
\label{Kurtz}
A sequence of $\R$-valued c\`{a}dl\`{a}g processes $\{Y^{(N)}\}_{N\in\N}$ is relatively compact if and only if $\{Y^{(N)}\}_{N\in\N}$ satisfies the following:
\begin{description}
\item [K1.] For every rational $t\geq0$,
\begin{equation}
  \lim_{r\to\infty} \sup_{N}\Prob{|Y^{(N)}(t)|>r}=0;
\end{equation}
\end{description}
\begin{description}
  \item[K2a.] For every $\eta>0$ and $T>0$, there exists $\delta>0$ such that
  \begin{equation}
    \sup_N\Prob{w^\prime(Y^{(N)},\delta,T)\geq\eta}\leq \eta.
  \end{equation}
\end{description}
Moreover, $\{Y^{(N)}\}_{N\in\N}$ is relatively compact if  it satisfies K1 and the condition K2b:
\begin{description}
  \item[K2b.] For each $T\geq 0$, there exists $\beta>0$ such that
  \begin{equation}
    \lim_{\delta\to 0}\sup_{N}\Ept{\sup_{0\leq t\leq T}\left|Y^{(N)}(t+\delta)-Y^{(N)}(t)\right|^\beta}=0.
  \end{equation}
\end{description}
\end{proposition}

If $\mset=\supint$ or  $\mset=\supint\times\hc$, equipped with the Euclidean metric,
then   $\mathbb{M}_F(\mset)$   (defined in  Section  \ref{sec_introNotation}) is a separable metric space and hence, a completely regular topological space with metrizable compacts. Thus, the next result follows from \cite[Theorem 4.6]{Jak86}  and  \cite[Theorem 5.1, Chapter 1]{BillingsleyBook}. Recall that a family  $\mathbb{F}$ of real continuous functionals on $\mathbb{M}_F(\mset)$ is said to separate  points in $\mathbb{M}_F(\mset)$ if for  every distinct $\mu,\tilde\mu\in\mathbb{M}_F(\mset)$, there exists a function $f\in\mathbb{F}$ such that $f(\mu)\neq f(\tilde \mu)$.

\begin{proposition}[Jakubowski's criteria]\label{Jakob}
For $\mset=\supint$ or  $\mset=\supint\times\hc$,
A sequence $\{\pi^{(N)}\}_{N\in\N}$ of $\mathbb{D}_{\mathbb{M}_F(\mset)}[0,\infty)$-valued random elements is tight  if and only if
\begin{description}
  \item [J1.] (Compact containment condition) For each $T>0$ and $\eta>0$ there exists a compact set $\cal K_{T,\eta}\subset\mathbb{M}_F(\mset)$ such that
      \[\liminf_{N}\Prob{\pi_t^{(N)}\in \cal K_{T,\eta}\text{ for all }t\in[0,T] }>1-\eta.\]

  \item [J2.] There exists a family $\mathbb{F}$ of real continuous functionals on $\mathbb{M}_F(\mset)$ that (i) is closed under addition, and (ii) separates points in $\mathbb{M}_F(\mset)$, such that $\{\pi^{(N)}\}$ is $\mathbb{F}$-weakly tight, that is for every $F\in\mathbb{F}$, the sequence $\{F(\pi^{(N)}_s);s\ge 0\}_{N\in\N}$ is tight in $\D$.
\end{description}
\vspace{-0.1in}
In particular, $\{\pi^{(N)}\}_{N\in\N}$ is relatively compact if it satisfies J1 and J2.
\end{proposition}

\begin{remark}\label{apx_Jremark}
A set that satisfies  properties (i) and (ii) in condition J2 is
\[\mathbb{F}=\{F: \exists f\in \C^1_c(\mset) \text{ such that } F(\mu)=\langle f,\mu\rangle,\;\; \forall\mu\in\mathbb{M}_F(\mset)\},\]
\end{remark}

\subsubsection{Relative Compactness of Sequences of Departures and Auxiliary Processes}\label{sec_tightDeparture}

First, note that  using \eqref{apx_Qbound} and \eqref{asm_initial_Xarrivalmoment}, we have for $\varphi\in\tightphiset$ and $\ell \geq 1$,
\begin{equation}\label{Qn_bound}
 \limsup_{N\to\infty} \Ept{\Quenbar_{\varphi^2,\ell}(t)}\leq \|\varphi\|^2_\infty \limsup_{N\to\infty}\Ept{ \xnbar(0)+\arrivalnbar(t)} <\infty.
\end{equation}

\begin{lemma}\label{lem_Mconv}
Suppose Assumptions \ref{asm_E}, \ref{asm_G}.\ref{asm_mean}, and \ref{asm_initial}.\ref{asm_initial_ind} hold, and fix  $\varphi\in\tightphiset$, $\ell\geq1$. Then, $\Mn_{\varphi,\ell}$ is a square integrable martingale. Moreover, for $t \geq 0$,
\begin{equation}\label{Alimit_temp2}
  \limsup_{N\to\infty}\Prob{\left| \sup_{0\leq s\leq t}\Mnbar_{\varphi,\ell}(s)\right|> \epsilon}=0, \quad
\epsilon > 0,
\end{equation}
 and $\Mnbar_{\varphi,\ell}\Rightarrow 0$ in $\D$, as $N\to\infty$.
\end{lemma}
\begin{proof}
By Proposition \ref{prop_prelim_Qcomp}, $\Mn_{\varphi,\ell}$ is a local martingale with  $[\Mn_{\varphi,\ell}]=\Quen_{\varphi^2,\ell}$. By  \eqref{Qn_bound}  and  \cite[Theorem 7.35]{klebBook}, $\Mn_{\varphi,\ell}$ is a square integrable $\{\filtn_t\}$-martingale  and for all $t\geq0$, $\mathbb{E}[(\Mn_{\varphi,\ell}(t))^2]= \mathbb{E}[\Quen_{\varphi^2,\ell}(t)].$ By Doob's inequality and \eqref{scaling_DR}, for every $T\geq 0$ and $\epsilon >0$, $\mathbb{P} \{ |\sup_{t\in[0,T]} \Mnbar_{\varphi,\ell}(t)|>\epsilon \} \leq \mathbb{E}[\Quenbar_{\varphi^2,\ell}(T)]/N\epsilon^2$. Together with  \eqref{Qn_bound}, this implies  \eqref{Alimit_temp2}, which   in turn shows  that $\Mnbar_{\varphi,\ell}$ converges to zero  in $\D$ in probability, and hence  in distribution.
\end{proof}

To obtain further tightness results on the departure processes, we recall another many-server model, the so-called GI/GI/N queue, studied in \cite{KasRam11}. In a GI/GI/N queue, arriving jobs choose an idle server at random if there exists one, or if all servers are busy,  join a common queue and enter service in a FCFS manner when servers become free. Equivalently, one can view the GI/GI/N as a network of $N$ parallel servers in which  each server has its own queue and newly arriving jobs join  the queue that has the least residual work (see \cite[Section XII.1]{Asm03}). A fluid limit for the GI/GI/N queue in the same regime (i.e.,  when $N\to\infty$ and the arrival rate is proportional to $N$) is obtained in \cite{KasRam11}. Although the $SQ(d)$ and the GI/GI/N models have very different routing mechanisms (leading to completely different dynamics of the cumulative service entry process $\kn$) the total departure process $\dn_1$ and the measure-valued process $\nun_1$ that keeps track of  ages of all jobs in service in the $SQ(d)$ model share some common properties with their counterparts (denoted in \cite{KasRam11} by $\dn$ and $\nun$, respectively)  in the GI/GI/N model.    As a result, the very same techniques used in \cite{KasRam11} to prove certain tightness estimates for $\{\dn\}_{N\in\N}$ and $\{\nun\}_{N\in\N}$ in the GI/GI/N model can be applied to also  prove analogous tightness estimates for the particular processes $\{\dn_1\}_{N\in\N}$ and $\{ \nun_1\}_{N\in\N}$ in the $SQ(d)$ model. Thus, the latter are  summarized in Lemmas \ref{lem_tightDepart}--\ref{lem_GGNconv} below; outlines of the proofs with a more detailed connections to the corresponding result in the GI/GI/N queue are given in Appendix \ref{apx_ggn}.

The first result is used to prove Lemma \ref{lem_Qtight} below and, along with Lemma \ref{lem_nuContainment2},  to prove relative compactness of the sequence of state processes in Proposition \ref{prop_tightnu}.

\begin{lemma}\label{lem_tightDepart}
Suppose Assumptions \ref{asm_E}, \ref{asm_G}.\ref{asm_mean}, and \ref{asm_initial}.\ref{asm_initial_ind} hold. Then, for $\ell\geq 1$ and $\varphi\in\tightphiset$, $\{\Anbar_{\varphi,\ell}\}_{N\in\N}$, $\{\dnbar_{\ell}\}_{N\in\N}$ and  $\{\Quenbar_{\varphi,\ell}\}_{N\in\N}$  are relatively compact  in $\D$.
\end{lemma}

\begin{proof}
See Appendix \ref{apx_ggn}.
\end{proof}

It follows from  \eqref{apx_Qbound} that the linear functional $\Quen_{\cdot,\ell}(t):\varphi\mapsto\Quen_{\varphi,\ell}(t)$  can be identified with a random  finite nonnegative (Radon) measure on $\supint\times\R_+$ (see Section \ref{sec_introNotation}) and $\Quen_{\cdot,\ell} = \{\Quen_{\cdot,\ell}(t), t\geq 0\}$  can be viewed as an  $\mathbb{M}_F(\supint\times\R_+)$-valued process.

\begin{lemma}\label{lem_Qtight}
Suppose Assumptions \ref{asm_E}-\ref{asm_initial} hold. Then, for  $\ell\geq1$,  $\{\Quenbar_{\cdot,\ell}\}$ is relatively compact in $\mathbb{D}_{\mathbb{M}_F(\supint\times\R_+)}\hc$.
\end{lemma}

\begin{proof}
See Appendix \ref{apx_ggn}
\end{proof}

The next two results are on properties of $\nunbar_1$. Lemma \ref{lem_nuContainment} is used to prove relative compactness of the routing sequence in Lemma \ref{lem_Rtight}, while Lemma \ref{lem_nuContainment2} is used to prove relative compactness of $\{\nunbar_\ell\}$ in Lemma \ref{prop_tightnu}.

\begin{lemma}\label{lem_nuContainment}
Suppose Assumption \ref{asm_initial} holds. Then
\begin{equation}\label{533}
  \lim_{m\to\endsup}\sup_N\Ept{\nunbar_1(0,(m,\endsup))}=0,
\end{equation}
and if $\endsup<\infty,$
\begin{equation}\label{534}
  \lim_{m\to\endsup}\sup_N\Ept{\int_{[0,m)}\frac{\overline{G}(m)}{\overline{G}(x)}\nunbar_1(0,dx)}=0.
\end{equation}
\end{lemma}
\begin{proof}
See Appendix \ref{apx_ggn}
\end{proof}

\begin{lemma}\label{lem_nuContainment2}
If Assumption \ref{asm_initial} holds, then  $\{\nunbar_1 \}_{N\in\N}$  satisfies condition J1 with
\begin{equation}\label{compact2}
    \cal K_{\eta,T}=\{\mu\in\mathbb{M}_F(\supint): \nu([m,\endsup))\leq \eta\},
\end{equation}
for some $m < \endsup,$ if $\endsup < \infty$, and with
\begin{equation}\label{compact1}
    \cal K_{\eta,T}=\{\mu\in\mathbb{M}_F(\R_+): \mu(r(n)+T,\infty)\leq \frac{1}{n} \text{ for all }n<-\lceil\log\eta/\log2\rceil\}.
  \end{equation}
for some  sequence  $r(n)\uparrow\infty$ if $\endsup=\infty$.
\end{lemma}

\begin{proof}
See Appendix \ref{apx_ggn}.
\end{proof}

The next result, along with Lemma \ref{lem_Mconv}, is used in Section \ref{sec_tightState} to identify the limit of the weighted departure processes
(Proposition \ref{prop_Alimit}).

\begin{lemma}\label{lem_GGNconv}
    Suppose Assumptions \ref{asm_E}--\ref{asm_initial} hold. For every $\ell\geq 1$, suppose $\{\nunbar_\ell\}_{N\in\N}$ converges almost surely to some $\nu_\ell\in\mathbb{D}_{M_F\supint}\hc$ along a subsequence. Then,
      \begin{equation}\label{apx_Alimit2}
        \limsup_{N\to\infty}\Ept{ \left|\Anbar_{\varphi,\ell}(t)-\int_0^t\langle \varphi(\cdot,s)h(\cdot),\nu_\ell(s)\rangle ds\right|}=0.
      \end{equation}
\end{lemma}

\begin{proof}
See Appendix \ref{apx_ggn}.
\end{proof}

\subsubsection{Relative Compactness of Arrival and Routing Processes}\label{sec_tightnessAux}

We first establish relative compactness of the arrival process sequence.

\begin{lemma}\label{asm_E_lem}
    Suppose Assumption \ref{asm_E} holds. Then, $\{\arrivalnbar\}_{N\in\N}$ is relatively compact in $\D$, and
    \begin{equation}\label{asm_initial_arrival2moment}
      \limsup_{N\to\infty} \Ept{\arrivalnbar(t)^2}<\infty, \quad t\geq0.
    \end{equation}
\end{lemma}

\begin{proof}
Since $\arrivalnbar\to\lambda\id$ in $\D$ by Lemma \ref{asm_initial_remark},  $\{\arrivalnbar\}_{N\in\N}$ is relatively compact. Moreover, by  Assumption \ref{asm_E} there exists a delayed renewal process $\tilde{E}$ such that $\arrivaln (t) = \tilde{E} (Nt)$. Let $U_{\widetilde E}$ be the renewal measure associated with the inter-arrival distribution $G_{\widetilde{E}}$. Then,  basic calculations (see \cite[Eq.\ (2.3)]{DalMoh78}) show that
\[\Ept{\widetilde E(t)^2}\leq U_{\widetilde E}(t)+\int_0^t U_{\widetilde E}(t-s)dU_{\widetilde E}(s) \leq 2U_{\widetilde E}(t)^2.\]
Hence, by the elementary renewal theorem (see e.g. \cite[Proposition V.1.4]{Asm03}),
\[\limsup_{N\to\infty}\Ept{\arrivalnbar(t)^2}=\limsup_{N\to\infty}\frac{1}{N^2}\Ept{\widetilde E(Nt)^2}\leq 2t^2\limsup_{N\to\infty}\left(\frac{U_{\widetilde E}(Nt)}{Nt}\right)^2 = 2\lambda^2t^2, \]
which proves \eqref{asm_initial_arrival2moment}.
\end{proof}

Next, we focus on the sequence $\{\Rn_{\varphi,\ell}\}_{N\in\N}$. By \eqref{def_Rphi}, for $\ell\geq1$ and $0\leq s\leq t$,
\begin{equation}\label{Rbound}
    |\Rnbar_{\varphi,\ell}(t)-\Rnbar_{\varphi,\ell}(s)|  \leq
\|\varphi\|_\infty \left(\arrivalnbar(t)-\arrivalnbar(s)\right).
\end{equation}

\begin{lemma}\label{lem_tightR}
  Suppose Assumption \ref{asm_E} holds.  For $\varphi\in\tightphiset$ and $\ell \geq 1$,
\begin{equation}\label{RnBound}
 \limsup_{N\to\infty}\Ept{\Rnbar_{\varphi,\ell}(t)} \leq \|\varphi\|_\infty \limsup_{N\to\infty}\Ept{\arrivalnbar(t)}<\infty, \qquad  t\geq0,
\end{equation}
and  $\{\Rnbar_{\varphi,\ell}\}_{N\in\N}$ is relatively compact in $\D$.
\end{lemma}
\begin{proof}
Inequality \eqref{RnBound} follows from \eqref{Rbound} with $s=0$, and \eqref{asm_initial_arrivalmoment}. Moreover, \eqref{Rbound} shows that the modulus of continuity $w^\prime$ for $\Rnbar_{\varphi,\ell}$ is bounded by that of $\arrivalnbar$. Relative compactness of $\{\Rnbar_{\varphi,\ell}\}_{N\in\N}$ then follows from that  of  $\{\arrivalnbar\}$  proved in Lemma \ref{asm_E_lem}, and the necessity and sufficiency of Kurtz's criteria K1 and K2a stated in Proposition \ref{Kurtz}.
\end{proof}

Next, recall the definitions of $\Bn_{\varphi,\ell}$ and $\Nn_{\varphi,\ell}$ in \eqref{def_Rcomp} and \eqref{def_Rmgale}, respectively.

\begin{lemma}\label{lem_Nconv}
Suppose Assumptions \ref{asm_E},  \ref{asm_G}.\ref{asm_mean} and \ref{asm_initial}.\ref{asm_initial_ind} hold.  Then for  $\varphi\in\tightphiset$ and $\ell\geq1$, $\Nn_{\varphi,\ell}$ is a square integrable $\{\filtn_t\}$-martingale. Moreover, for  $t \geq 0$,
\begin{equation}\label{Nbound}
  \limsup_{N\to\infty}\Prob{\left|\sup_{0\leq s\leq t} \Nnbar_{\varphi,\ell}(s)\right|> \epsilon}=0, \quad \forall \epsilon > 0,
\end{equation}
and  $\Nnbar_{\varphi,\ell}\Rightarrow 0$ in $\D$, as $N\to\infty$.
\end{lemma}
\begin{proof}
By Proposition \ref{prop_prelim_Rcomp},  $\Nn_{\varphi,\ell}$ is a local martingale with  $[\Nn_{\varphi,\ell}]=\Rn_{\varphi^2,\ell}.$ Since $\Eptil{\Rn_{\varphi^2,\ell}(t)}<\infty$ by \eqref{RnBound},
by \cite[Theorem 7.35]{klebBook}, $\Nn_{\varphi,\ell}$ is a square integrable martingale,  and
$\mathbb{E}[(\Nn_{\varphi,\ell}(t))^2]$ $=$ $\mathbb{E}[\Rn_{\varphi^2,\ell}(t)].$
By Doob's inequality, for $\epsilon  > 0$,
$\mathbb{P}\{ |\sup_{t\in[0,T]} \Nnbar_{\varphi,\ell}(t)|>\epsilon \}$ $\leq$ $\mathbb{E}[\Rnbar_{\varphi^2,\ell}(T)]/N\epsilon^2.$
Together with \eqref{RnBound}, this implies \eqref{Nbound}, which implies $\Nnbar_{\varphi,\ell}$ converges to zero in distribution.
\end{proof}

The following result is a direct corollary of Lemmas \ref{lem_tightR} and \ref{lem_Nconv}.

\begin{corollary}\label{cor_tightB}
  Suppose Assumptions \ref{asm_E}, \ref{asm_G}.\ref{asm_mean}, and \ref{asm_initial}.\ref{asm_initial_ind} hold. Then, for $\varphi\in\tightphiset$ and $\ell\geq 1$,  $\{\Bnbar_{\varphi,\ell}\}_{N\in\N}$ is relatively compact in $\D$.
\end{corollary}

Analogous to $\Quen_{\cdot,\ell}$, due to  \eqref{Rbound}, the linear functional $\Rn_{\cdot,\ell}:\varphi\mapsto\Rn_{\varphi,\ell}$ can be identified with a random  finite nonnegative  measure on $\supint\times\R_+$, and hence, $\Rn_{\cdot,\ell} = \{\Rn_{\cdot,\ell}(t), t \geq 0\},$ can be viewed as an  $\mathbb{M}_F(\supint\times\R_+)$-valued process.

\begin{lemma}\label{lem_Rtight}
Suppose Assumptions \ref{asm_E}, \ref{asm_G}.\ref{asm_mean}, and \ref{asm_initial} hold. Then for every $\ell\geq1$, the sequence $\{\Rnbar_{\cdot,\ell}\}_{N\in\N}$ is relatively compact in $\mathbb{D}_{\mathbb{M}_F(\supint\times\R_+)}\hc$.
\end{lemma}

\begin{proof}
By definition, $\Rnbar_{\varphi,\ell}$ is a pure jump process and hence lies in $\mathbb{D}_{\mathbb{M}_F(\supint\times\R_+)}\hc$. Lemma \ref{lem_tightR} implies that for  $\varphi\in\tightphiset,$  $\{\Rnbar_{\varphi,\ell}\}_{N\in\N}$ is relatively compact (and therefore tight, by Prohorov's theorem) in $\D.$   It only remains to show that $\{\Rnbar_{\cdot,\ell}\}_{N\in\N}$ satisfies condition J1. We first claim that for $\ell\geq 1$,
\begin{equation}\label{Rtight_claim}
  \lim_{m\to\endsup}\sup_N\Ept{\Rnbar_{\f1_{(m,\endsup)},\ell}(T)} =0.
\end{equation}
Fix $m > 0$.  Substituting $\varphi(x,s)=\f1_{(m,\endsup)}(x)$ in \eqref{def_Rphi1}, we see that $\Rnbar_{\f1_{(m,\endsup)},1}\equiv 0$.  This proves \eqref{Rtight_claim} for $\ell = 1$.  For $\ell\geq 2$, substituting $\varphi(x,s)=\f1_{(m,\endsup)}(x)$ in \eqref{prelim_deltanu4_eq},  we have
\begin{align}\label{Rtight_temp1}
    \Rn_{\f1_{(m,\endsup)},\ell}(T)=\sum_{j\ge\j0}\;\sum_{j^\prime\geq 1}\indic{\arrivalTime_{j^\prime}\leq  T}  &\indic{\agen_j(\arrivalTime_{j^\prime})>m} \indic{\kninv_j<\arrivalTime_{j^\prime}\leq\deptn_j}\\
    &\indic{\custServSize_j(\arrivalTime_{j^\prime}-)= \ell-1}\indic{\stationn_{j^\prime}=\stationn_{j}}\notag.
\end{align}
For  every $j\geq\j0$, consider the (random) set
\[
\mathcal{J}_{j,\ell}^{(N)}(T)\doteq \left\{j^\prime\geq1 \; : \arrivalTime_{j^\prime}\leq T, \arrivalTime_{j^\prime}\in [\kninv_j,\deptn_j),\custServSize_j(\arrivalTime_{j^\prime}-)= \ell-1,\stationn_{j^\prime}=\stationn_{j}\right\},
\]
of jobs $j'$ that arrive  prior to $T$ while job $j$ is receiving service, and are routed to the same queue $\stationn_{j}$  as job $j$,  and this queue has length $\ell-1$ right before the arrival of $j'$.   During the time $(\kninv_j,\deptn_j)$ when job $j$ is in service at queue $\stationn_{j}$, there are no departures from the queue and so its length is non-decreasing.  Thus, there can be at most one job $j^\prime$ that arrives during that period and finds the queue length  equal to $\ell - 1$.
In other words, when $\mathcal{J}_{j,\ell}^{(N)}(T) \neq \emptyset$, $\mathcal{J}_{j,\ell}^{(N)}(T)  = \{j_*\}$ for some  $j_*= j_{*}(j,\ell,N,T).$  Also, for $j>\arrivaln(T)$, $\arrivalTime_{j} > T$, and thus,  $\mathcal{J}_{j,\ell}^{(N)}(T)=\emptyset.$  Therefore, since all departure and arrival times are almost surely distinct (i.e., $P(\tilde\Omega_T)=1$ by Corollary \ref{cor_distinctness}), almost surely,
\begin{align}\label{Rtight_temp2}
    \Rn_{\f1_{(m,\endsup)},\ell}(T) = \sum_{j\geq\j0} \sum_{j'\in\mathcal{J}_{j,\ell}^{(N)}(T)} \indic{\agen_j(\arrivalTime_{j'})>m}
    &=\sum_{j=\j0}^{\arrivaln(T)} \indic{\mathcal{J}_{j,\ell}^{(N)}(T)\neq\emptyset} \indic{\agen_j(\arrivalTime_{j_*})>m},
\end{align}
Now we consider two possible cases.  If $\endsup=\infty$ and a job $j$ has initial age $\agen_j(0)\leq m$, then  $\agen_j(t)\leq m+T$ for all $t\in[0,T]$. Therefore, \eqref{Rtight_temp2} implies
\begin{align}\label{Rtight_temp3}
    \Ept{\Rnbar_{\f1_{(m+T,\endsup)},\ell}(T)}    \leq \Ept{\frac{1}{N}\sum_{j\ge\j0} \indic{\agen_j(0)>m}} = \Ept{\nunbar_1(0,(m,\infty))}.
\end{align}
On the other hand, if $\endsup<\infty$, using \eqref{Rtight_temp2} and the fact that the age  process $\agen_j(\cdot)$ is non-decreasing and bounded by the  service time $\serviceTime_j$,  we have
\begin{align}\label{Rtight_temp4}
    \Rn_{\f1_{(m,\endsup)},\ell}(T)\leq& \sum_{j=\j0}^{0} \indic{\agen_j(0)>m}+\sum_{j=\j0}^{0} \indic{\agen_j(0)\leq m}\indic{\serviceTime_j>m}
    +\sum_{j=1}^{\arrivaln(T)} \indic{\serviceTime_j>m}.
\end{align}
Using the fact that $\j0$, $\agen_j(0)$, and $\nun_1(0)$ are $\filtn_0$-measurable, we then have
\begin{align*}
  \Ept{\sum_{j=\j0}^{0} \indic{\agen_j(0)\leq m}\indic{\serviceTime_j>m}} & =\Ept{\sum_{j=\j0}^{0} \indic{\agen_j(0)\leq m}\Ept{\indic{\serviceTime_j>m}|\filtn_0}} \\
  & =\Ept{\sum_{j=\j0}^{0} \indic{\agen_j(0)\leq m}\frac{\overline G(m)}{\overline G(\agen_j(0))}} \\
  & =\Ept{\int_{[0,m)}\frac{\overline G(m)}{\overline G(x)}\nun_1(0,dx)}.
\end{align*}
Hence, taking expectations in \eqref{Rtight_temp4} and  using  the independence of the initial conditions and the arrival process for the third term on the right-hand side, we have
\begin{equation}\label{Rtight_temp5}
  \Ept{\Rnbar_{\f1_{(m,\endsup)},\ell}(T)}\leq \Ept{\nunbar_1(0,(m,\endsup))}+\Ept{\int_{[0,m)}\frac{\overline G(m)}{\overline G(x)}\nunbar_1(0,dx)} + \overline G(m)\Ept{\arrivalnbar(T)}.
\end{equation}
Taking first the supremum over $N$ and then the limit as $m\to\endsup$ of both sides of \eqref{Rtight_temp3} and \eqref{Rtight_temp5}, the claim \eqref{Rtight_claim} for both cases  follows  by using \eqref{533} and \eqref{534} in Lemma \ref{lem_nuContainment}, the bound \eqref{asm_initial_arrivalmoment}, and the  elementary identity $\lim_{m\to\endsup}\overline G(m)=0$.

We now use the claim \eqref{Rtight_claim} to complete the proof of the lemma.
Fix $\eta>0$ and apply  \eqref{RnBound} with $\varphi=\f1$ and $t=T$, to conclude that the constant
\begin{equation}\label{temp_Rtight1}
  C(\eta,T)\doteq\frac{2}{\eta}\sup_N \Ept{\Rnbar_{\f1,\ell}(T)}
\end{equation}
is finite. Also, by \eqref{Rtight_claim}, there exists a sequence $\{m(n,\eta)\}$ with $\lim_{n \rightarrow \infty} m(n,\eta) = \endsup$ such that
\begin{equation}\label{temp_Rtight2}
  \sup_N\Ept{\Rnbar_{\f1_{(m(n,\eta),\endsup)},\ell}(T)}\leq \frac{\eta}{n2^{n+1}}.
\end{equation}
The subset $\mathcal{K}_{\eta,T}$ of $\mathbb{M}_F(\supint\times\R_+)$ defined as
\[  \mathcal{K}_{\eta,T}\doteq\left\{\mu\in\mathbb{M}_F(\supint\times\R_+)\; : \langle\f1,\mu \rangle\leq C(\eta,T),\;\mu((m(n,\eta),\endsup)\times \R_+)\leq \frac{1}{n}\;\forall n\in\N  \right\},\]
is compact because $\sup_{\mu\in\mathcal{K}_{\eta,T}} \;\mu(\supint\times\R_+)\leq C(\eta,T)$ and
\[\inf_{ \genfrac{}{}{0pt}{1}{C\subset \supint\times\R_+} {C \text{ compact}}} \;\sup_{\mu\in\mathcal{K}_{\eta,T}}\mu (C^c)\leq \inf_n \sup_{\mu\in\mathcal{K}_{\eta,T}}\mu\big( (m(n,\eta),\endsup)\times \R_+ \big)=0.\]
Finally, using \eqref{temp_Rtight1} and \eqref{temp_Rtight2}, for $W \doteq
\{\Rnbar_{\cdot,\ell}(t)\not\in\mathcal{K}_\eta \text{ for some }t\in[0,T] \}$, we have
\begin{align*}
  \Prob{W }&\leq \Prob{\Rnbar_{\f1,\ell}(T)\geq C(\eta,T)} + \sum_{n\in\N}\Prob{\Rnbar_{\f1_{(m(n,\eta),\endsup)},\ell}(T)>\frac{1}{n}}\\
  &\leq \sup_N\frac{\Ept{\Rnbar_{\f1,\ell}(T)}}{ C(\eta,T)}+\sum_{n\in\N}n\Ept{\Rnbar_{\f1_{(m(n,\eta),\endsup)},\ell}(T)}\\
  &\leq\eta.
\end{align*}
This shows that $\{\Rnbar_{\cdot,\ell}\}_{N\in\N}$ satisfies condition J1, and  completes the proof.
\end{proof}

\subsubsection{Relative Compactness of State Variables}\label{sec_tightState}

\begin{proposition}\label{prop_tightnu}
  Suppose Assumptions \ref{asm_E},  \ref{asm_G}.\ref{asm_mean}  and \ref{asm_initial}.\ref{asm_initial_ind} hold. Then, for every $\ell \geq 1$, the sequence $\{\nunbar_\ell\}_{N\in\N}$ is relatively compact in $\mathbb{D}_{\mathbb{M}_F\supint}[0,\infty)$.
\end{proposition}
\begin{proof}
By Lemma \ref{lem_nuContainment},  the sequence  $\{\nunbar_1\}_{N\in\N}$    satisfies condition J1 with the compact set $\cal K_{\eta,T}$ equal to either \eqref{compact1} or \eqref{compact2} depending on whether $\endsup$ is finite or infinite.  Moreover, for every interval $I$,  $t\geq 0$ and  $\ell\geq 2$, $\nunbar_{\ell}(t,I)\leq \nunbar_1(t,I)$, and hence,  condition J1 also holds for $\{\nunbar_\ell\}$ with the same set $\cal K_{\eta,T}$.

It remains to  prove that $\{\nun_\ell\}$ satisfies condition J2, for which it suffices, by Remark \ref{apx_Jremark}, to show that for  $f\in\fset$, $\{ \langle f,\nunbar_\ell(t)\rangle\}_{N \in\N}$ is tight in $\D$. It follows from Proposition \ref{prelim_dynamics_prop}  and Remark \ref{rem_dynamic_extension} that
\begin{align*}
  \langle f,\nunbar_\ell(t)\rangle =& \langle f,\nunbar_\ell(0)\rangle +\int_0^t \langle f^\prime,\nunbar_\ell(s)\rangle ds - \Quenbar_{f,\ell}(t)
  + f(0)\dnbar_{\ell+1}(t)+ \Rnbar_{f,\ell}(t).
\end{align*}
Relative compactness of $\{\langle f,\nunbar_\ell(0)\rangle\}_{N \in \N}$ follows from Assumption \ref{asm_initial}.\ref{asm_initial_nu}, relative compactness of $\{\dnbar_{\ell+1}\}_{N\in\N}$ and $\{\Quenbar_{f,\ell}\}_{N\in\N}$  follow from Lemma \ref{lem_tightDepart}, and relative compactness of $\{\Rnbar_{f,\ell}\}_{N\in\N}$  follows from Lemma \ref{lem_tightR}. Moreover, the fact that $\{\int_0^\cdot \langle f^\prime,\nunbar_\ell(s)\rangle ds\}_{N\in\N}$ satisfies both criteria K1 and K2b, and hence is relatively compact, follows from the bound $\left|\int_t^{t+\delta}\langle f^\prime,\nunbar_\ell(s)\rangle ds\right| \leq \delta \|f^\prime\|_\infty,$ which uses the fact that $\nunbar_{\ell}(s)$ is a sub-probability measure. Therefore, for all $\ell\geq 1$ and $f\in\C_b^1\hc$, $\{\langle f,\nunbar_\ell(t)\rangle\}_{N\in\N}$ is relatively compact  in $\D$  (and therefore, tight, by Prohorov's Theorem since $\D$ is Polish).
\end{proof}

We summarize all the results of this section in the following theorem.

\begin{theorem}\label{thm_tightness}
Suppose Assumptions \ref{asm_E}, \ref{asm_G}.\ref{asm_mean}, and \ref{asm_initial} hold. Then, for each $\ell\geq 1$, the sequence  $\{(\nunbar,\Quenbar_{\cdot,\ell},\Rnbar_{\cdot,\ell})\}_{N\in\N}$  is relatively compact in $\mathbb{D}_{\s}[0,\infty)\times\mathbb{D}_{\mathbb{M}_F(\supint\times\R_+)}\hc^2$.
\end{theorem}

\begin{proof}
By Proposition \ref{prop_tightnu}, for each $\ell\geq 1$, $\{\nunbar_\ell\}_{N\in\N}$ is relatively compact in $\mathbb{M}_F\supint$.  Therefore, by a diagonalization argument, every subsequence of $\{\nunbar\}_{N\in\N}=\{(\nunbar_\ell;\ell\geq1)\}_{N\in\N}$ has a further subsequence that is convergent, simultaneously for all $\ell\geq1$. This  means that the sequence $\{\nunbar\}_{N\in\N}\subset\s$ is relatively compact with respect to the norm $d_{\s}$ defined in \eqref{def_dS}.  Relative compactness of the other two components are proved in Lemmas \ref{lem_Qtight} and \ref{lem_Rtight}, respectively.
\end{proof}

\subsection{Characterization of Subsequential Limits}\label{sec_limit}

We now  show that any subsequential limit of the sequence $\{\nun\}_{N\in\N}$ is a  solution to the hydrodynamic equations.
Section \ref{sec_limitCharacterization} contains the proofs of Theorem \ref{thm_convergence}, and Corollary \ref{cor_PropofChaos}.

\subsubsection{Limit of the Departure Processes}\label{sec_limitDeparture}

\begin{proposition}\label{prop_Alimit}
Suppose Assumptions \ref{asm_E}--\ref{asm_initial} hold, and fix $\ell\geq 1$. If $(\Quenbar_{\cdot,\ell},\nunbar)$ converges to $(\Que_{\cdot,\ell},\nu)$ in $\mathbb{D}_{\mathbb{M}_F(\supint\times\hc)}\hc\times \mathbb{D}_{\s}\hc$ almost surely, then, for every $\varphi\in\tightphiset$, $\mathcal{D}_{\varphi,\ell}$ is continuous and for every $t\geq0$,
\begin{equation}\label{Alimit}
    \Que_{\varphi,\ell}(t)= \int_0^t\langle \varphi(\cdot,s)h(\cdot),\nu_\ell(s)\rangle ds,\quad\quad\text{a.s.}
\end{equation}
\end{proposition}
\begin{proof}
 Using \eqref{def_Mn}, the difference of the two sides of \eqref{Alimit} is equal to
\begin{equation}\label{Alimit_temp0}
\Que_{\varphi,\ell}(t)-\Quenbar_{\varphi,\ell}(t)+ \Mnbar_{\varphi,\ell}(t)
+\Anbar_{\varphi,\ell}(t)-\int_0^t\langle \varphi(\cdot,s)h(\cdot),\nu_\ell(s)\rangle ds,
\end{equation}
By the hypothesis of this proposition,  $\Quenbar_{\varphi,\ell}$ converges to $\Que_{\varphi,\ell}$  almost surely, as $N\to\infty$. Moreover, by Lemma \ref{lem_condInd}, only one departure can occur at each time, and hence the jump size of $\Quenbar_{\varphi,\ell}$ is bounded by $\|\varphi\|_\infty$/N. Therefore, by \cite[Theorem 13.4]{BillingsleyBook}, $\mathcal{D}_{\varphi,\ell}$ is continuous and the convergence also holds uniformly on compact sets, almost surely. Since $\Mnbar_{\varphi,\ell}(t)$  also converges to zero in probability by  Lemma \ref{lem_Mconv},  in view of \eqref{Alimit_temp0} to prove \eqref{Alimit} it suffices to show  that
\begin{equation}\label{temp_limitD2}
\limsup_{N\to\infty}\Ept{ \left|\Anbar_{\varphi,\ell}(t)-\int_0^t\langle \varphi(\cdot,s)h(\cdot),\nu_\ell(s)\rangle ds\right|}=0.
\end{equation}
However, \eqref{temp_limitD2} follows from Lemma \ref{lem_GGNconv}(b).
\end{proof}

\subsubsection{Limit of the Routing Processes}\label{sec_limitRouting}

We begin with a  preliminary result, which uses the bounds obtained in Lemma \ref{lem_renewal}.

\begin{lemma}\label{lem_hRN}
Suppose Assumption \ref{asm_E} holds.
\begin{enumerate}
  \item  \label{lem_hRN_Mean} For every $0\leq s\leq t$, the following limit holds in probability as $N\to\infty$:
  \begin{equation}\label{hRN}
   \frac{1}{N}\int_s^t\hen(\ren(u))du\to (t-s)\lambda;
  \end{equation}
  \item  \label{lem_hRN_Var} For every $t\geq0$,
   \begin{equation}\label{hRN_var}
   C_E(t)\doteq\limsup_{N\to\infty}\Ept{\left(\frac{1}{N}\int_0^t\hen(\ren(u))du\right)^2}<\infty.
  \end{equation}
\end{enumerate}
\end{lemma}
\begin{proof}
By Assumption \ref{asm_E}, $\arrivaln$ is a pure renewal process with inter-arrival distribution $\Gen$ and backward recurrence time $\ren$. Applying Lemma \ref{lem_renewal} with $\Pstar$, $\Gstar$, $\hstar$, $\rstar$ and $\rstar_0$ replaced by $\arrivaln$, $\Gen$, $\hen$, $\ren$, and $R^{(N)}$, respectively, it follows from \eqref{renewal_Bound} that  for every $t\geq0$,
\begin{align*}
    &\Prob{\left|\frac{1}{N}\int_0^t\hen(\ren(s))ds- \lambda t\right|>2\epsilon}\\
    &\hspace{1cm} \leq\frac{1}{\epsilon^2}\Ept{\left(\frac{1}{N}\int_0^t\hen(\ren(s))ds- \arrivalnbar(t)\right)^2} +\Prob{\big|\arrivalnbar(t)- \lambda t\big|>\epsilon}\\
    &\hspace{1cm} \leq\frac{3}{N\epsilon^2}\left(\frac{4}{N}+\Ept{\arrivalnbar(t)}\right) +\Prob{\big|\arrivalnbar(t)- \lambda t\big|>\epsilon}
\end{align*}
By \eqref{asm_initial_arrivalmoment}, $\Eptil{\arrivalnbar(t)}$ is finite, and $\arrivalnbar(t)$ converges in expectation, and hence in probability, to $\lambda t$ by Lemma \ref{asm_E_lem}. Hence, the right-hand side of display above converges to zero as $N\to\infty$. Since both sides of \eqref{hRN} are linear in $t$, \eqref{hRN} follows.

To establish \eqref{hRN_var}, by another application  of \eqref{renewal_Bound}, we have
\begin{align*}
    \Ept{\left(\frac{1}{N}\int_0^t\hen(\ren(s))ds \right)^2}\leq&\frac{2}{N^2}\Ept{\left(\int_0^t\hen(\ren(s))ds- \arrivaln(t)\right)^2} +2\Ept{\arrivalnbar(t)^2}\\
    \leq&\frac{24}{N^2}+\frac{6}{N}\Ept{\arrivalnbar(t)} +2\Ept{\arrivalnbar(t)^2}.
\end{align*}
Taking the limit superior as $N\to \infty$ of both sides of the inequality above and using  \eqref{asm_initial_arrivalmoment}, \eqref{hRN_var} follows with $C_E(t)=2 \limsup_{N\to\infty}\Ept{\arrivalnbar(t)^2}$,
which is finite by \eqref{asm_initial_arrival2moment}.
\end{proof}

\begin{proposition}\label{prop_Blimit}
Suppose  Assumptions \ref{asm_E},  \ref{asm_G}.\ref{asm_mean}, and \ref{asm_initial} hold,  fix $\ell\geq1$, and let
$\eta_\ell$ be  defined as in \eqref{Fluid_R}.
 If $(\Rnbar_{\cdot,\ell},\nunbar)$ converges to $(\mathcal{R}_{\cdot,\ell},\nu)$ in $\mathbb{D}_{\mathbb{M}_F(\supint\times\hc)}\hc\times \mathbb{D}_{\s}\hc$, almost surely as $N\to\infty,$ then for every $\varphi\in\tightphiset$, $\mathcal{R}_{\varphi,\ell}$ is continuous and satisfies for every $t\geq0$,
\begin{equation}\label{Blimit}
  \mathcal{R}_{\varphi,\ell}(t)=\int_0^t\langle\varphi(\cdot,s),\eta_\ell(s)\rangle ds,\quad\quad\text{a.s.}
\end{equation}
\end{proposition}
\begin{proof}
Fix $\varphi\in\tightphiset$, $\epsilon > 0$  and let ${\mathcal W} \doteq
 \{|\mathcal{R}_{\varphi,\ell}(t)-\int_0^t\langle\varphi(\cdot,s),\eta_\ell(s)\rangle ds|>3\epsilon\}$.  Using the relation $\Rnbar_{\varphi,\ell}=\Bnbar_{\varphi,\ell}+\Nnbar_{\varphi,\ell}$ from  \eqref{def_Rmgale}, we have
\begin{align}\label{Blimit_temp0}
  \Prob{{\mathcal W}}
  &\leq\Prob{\left|\mathcal{R}_{\varphi,\ell}(t)-\Rnbar_{\varphi,\ell}(t)\right|>\epsilon} +\Prob{\left|\Nnbar_{\varphi,\ell}(t)\right|>\epsilon}\notag \\
& \qquad   +\Prob{\left|\Bnbar_{\varphi,\ell}(t)-\int_0^t\langle\varphi(\cdot,s),\eta_\ell(s)\rangle ds\right|>\epsilon}.
\end{align}
By the hypothesis of this proposition, $\Rnbar_{\varphi,\ell}$ converges almost surely to $\mathcal{R}_{\varphi,\ell}$, in $\D$. Also, by \eqref{Rbound}, the jump sizes of $\Rnbar_{\varphi,\ell}$ are bounded by $\|\varphi\|_{\infty}$ times the jump sizes of $\arrivalnbar$, which are at most $1/N$. Therefore, the maximum jump size of $\Rnbar_{\varphi,\ell}$ converges to zero as $N\to\infty,$ and by  \cite[Theorem 13.4]{BillingsleyBook}, $\mathcal{R}_{\varphi,\ell}$ is almost surely continuous and $\Rnbar_{\varphi,\ell}$ converges to $\mathcal{R}_{\varphi,\ell}$, almost surely, uniformly on compact sets:
\begin{equation}\label{Blimit_temp1}
  \limsup_{N\to\infty}\Prob{\left| \mathcal{R}_{\varphi,\ell}(t)-\Rnbar_{\varphi,\ell}(t)\right|> \epsilon}=0, \qquad t \geq 0.
\end{equation}

Moreover, by \eqref{def_Rcomp} and \eqref{Fluid_R}, we can write
\begin{align}\label{Blimit_temp3}
  \Bnbar_{\varphi,\ell}(t)-\int_0^t\langle\varphi(\cdot,s),\eta_\ell(s)\rangle ds
                       =& \int_0^t\left(\frac{1}{N}\hen(\ren(u))-\lambda\right)f_{\varphi,\ell}(u)du\\
                        &+ \frac{1}{N}\int_0^t\hen(\ren(u))(f^{(N)}_{\varphi,\ell}(u)-f_{\varphi,\ell}(u))du, \notag
\end{align}
with
\[f_{\varphi,\ell}(u)\doteq \twopartdef{\varphi(0,u)(1- S_1(u)^d)}{\ell=1,}{  \left\langle \varphi(\cdot, u), \nu_{\ell-1}(u)-\nu_\ell(u)\right\rangle \poly{S_{\ell-1}(u)}{S_\ell(u)}}{\ell\geq2,}\]
and $f^{(N)}$ defined analogously, but with $S_\ell$ and $\nu_\ell$ replaced by $\nunonebar_\ell$ and $\nunbar$, respectively, for $\ell \geq 1$.  By the hypothesis of this proposition, almost surely, $\nunbar_\ell$ converges weakly to $\nu_\ell$, and hence, almost surely, $\nunonebar_\ell = \langle \f1,  \nunbar_\ell  \rangle$ and $\langle \varphi,\nunbar_{\ell-1}-\nunbar_\ell\rangle$ converge to $S_\ell = \langle \f1, \nu_\ell \rangle$ and $\langle \varphi, \nu_{\ell-1}-\nu_\ell\rangle$, respectively. By Lemma \ref{lem_condInd}, there is at most one arrival or departure at each time, and hence, the maximum jump sizes of $\nunonebar_\ell$ and $\langle \varphi,\nunbar_{\ell-1}-\nunbar_\ell\rangle$ are bounded by $1/N$ and $\|\varphi\|_\infty/N$, respectively. Therefore, by \cite[Theorem 13.4]{BillingsleyBook}, $S_\ell$ and $\langle \varphi, \nu_{\ell-1}-\nu_\ell\rangle$ are continuous.  Consequently, $f_{\varphi,\ell}$ is continuous, and  $f_{\varphi,\ell}^{(N)}$ converges to $f_{\varphi,\ell}$ uniformly on compact sets, almost surely.

Fix $t\geq0$ and $\delta>0$, and  let  $w_{f_{\varphi,\ell}}(\delta,t)$ be the modulus of continuity of $f_{\varphi,\ell}$, and define $\mathcal{P}=\{0=t_0<t_1<...<t_n=t\}$ to be a partition of size $\delta = \max_{j} |t_j - t_{j-1}|$. We have
\begin{align}\label{Blimit_temp31}
   \lim_{N \rightarrow \infty} \Prob{\sum_{j=0}^{n-1}f_{\varphi,\ell}(t_j)\left|\frac{1}{N}\int_{t_j}^{t_{j+1}}\hen(\ren(u))du- (t_{j+1}-t_{j})\lambda \right|>\epsilon} = 0,
\end{align}
and, by the Markov and Cauchy-Schwartz inequalities and \eqref{hRN_var},
\begin{align*}
&\Prob{w_{f_{\varphi,\ell}}(\delta,t)\sum_{j=0}^{n-1}\int_{t_j}^{t_{j+1}}\left|\frac{1}{N}\hen(\ren(u))-\lambda\right|du>\epsilon}\\
        &\hspace{1cm}\leq \frac{1}{\epsilon}\Ept{w_{f_{\varphi,\ell}}(\delta,t)\left(\frac{1}{N}\int_{0}^{t}\hen(\ren(u))du+\lambda t\right)}\\
          &\hspace{1cm}\leq \frac{ \sqrt{2}}{\epsilon} \Ept{w^2_{f_{\varphi,\ell}}(\delta,t)}^{\frac{1}{2}}\left(\Ept{\left(\frac{1}{N}\int_0^t\hen(\ren(u))du\right)^2} +\lambda^2t^2\right)^{\frac{1}{2}}.
\end{align*}
Taking the limit superior as $N \rightarrow \infty$ of both sides of the display above, using  \eqref{hRN_var} and
\eqref{Blimit_temp31}, we conclude that
\begin{align*}
    &\limsup_{N\to\infty}\Prob{\left|\int_0^t\left(\frac{1}{N}\hen(\ren(u))-\lambda\right)f_{\varphi,\ell}(u)du\right|>2\epsilon} \leq \frac{C(t)}{\epsilon}  \left( \Ept{w^2_{f_{\varphi,\ell}}(\delta,t)} \right)^{1/2},
\end{align*}
where $C(t)\doteq \sqrt{2}(C_E(t)+\lambda^2t^2)^{1/2}$. Since $f_{\varphi,\ell}$ is continuous on $[0,t]$,  $w^2_{f_{\varphi,\ell}}(\delta,t)$ is bounded and converges almost surely to zero as $\delta\to 0$. Sending $\delta\to0$ in the last display and applying the bounded convergence theorem, we see that
\begin{equation}\label{Blimit_temp33}
    \limsup_{N\to\infty}\Prob{\left|\int_0^t\left(\frac{1}{N}\hen(\ren(u))-\lambda\right)f_{\varphi,\ell}(u)du\right|>2\epsilon} =0.
\end{equation}

Similarly, applying the Markov and Cauchy-Schwartz inequalities, we have
\begin{align*}
  &\Prob{\left|\frac{1}{N}\int_0^t\hen(\ren(u))(f^{(N)}_{\varphi,\ell}(u)-f_{\varphi,\ell}(u))du\right|>\epsilon}\\
  &\hspace{1cm}\leq \frac{1}{\epsilon}\Ept{\sup_{u\in[0,t]}\Big|f^{(N)}_{\varphi,\ell}(u)-f_{\varphi,\ell}(u)\Big|^2}^{\frac{1}{2}}
  \Ept{\left(\frac{1}{N}\int_0^t\hen(\ren(u))du\right)^2}^{\frac{1}{2}}. \notag
\end{align*}
Together with  \eqref{hRN_var}, the almost sure uniform convergence of  $f^{(N)}_{\varphi,\ell}$ to  the bounded function $f^{\varphi,\ell}$, and the bounded convergence theorem,  this implies
\begin{equation}\label{Blimit_temp5}
\limsup_{N\to\infty}  \Prob{\left|\frac{1}{N}\int_0^t\hen(\ren(u))(f^{(N)}_{\varphi,\ell}(u)-f_{\varphi,\ell}(u))du\right|>\epsilon}=0.
\end{equation}
From \eqref{Blimit_temp3}, \eqref{Blimit_temp33} and \eqref{Blimit_temp5} it follows that
\begin{equation}\label{Blimit_temp6}
  \limsup_{N\to\infty} \Prob{\left|\Bnbar_{\varphi,\ell}(t)-\int_0^t\langle\varphi(\cdot,s),\eta_\ell(s)\rangle ds\right|>\epsilon} =0.
\end{equation}
Finally, \eqref{Blimit} follows on sending first $N\to\infty$ on the right-hand side of \eqref{Blimit_temp0}, next invoking \eqref{Blimit_temp1}, \eqref{Nbound} of Lemma \ref{lem_Nconv} and \eqref{Blimit_temp6}, and then sending  $\epsilon\downarrow0.$
\end{proof}

\subsubsection{Proofs of the Convergence Theorem  and Propagation of Chaos}

\label{sec_limitCharacterization}

\begin{proof}[Proof of Theorem \ref{thm_convergence}]
By Assumption \ref{asm_initial}, Lemma \ref{asm_E_lem}, and Theorem \ref{thm_tightness}, the sequence
\begin{equation}
\Yn\doteq\left(\nunbar(0), \arrivalnbar,\nunbar, \Quenbar_{\cdot,\ell}, \Rnbar_{\cdot,\ell};\ell\geq 1 \right), \quad N \in \N,
\end{equation}
is relatively compact in
\begin{equation}
  \Y\doteq\s\times\D\times\mathbb{D}_{\s}[0,\infty)\times \mathbb{D}_{\mathbb{M}_F(\supint\times\R_+)}\hc^{\N_0}\times \mathbb{D}_{\mathbb{M}_F(\supint\times\R_+)}\hc^{\N_0}.
\end{equation}
Therefore, for every subsequence  $\{ \Y^{N_k}\}$, there exists a  a further subsequence $\{N_{k_j}\}$, such that as $j \rightarrow \infty$, $Y^{N_{k_j}}$ converges in distribution to a random element
\begin{equation}
  Y\doteq (\nu(0),\lambda \id,\nu,\Que_{\cdot,\ell}, {\cal R}_{\cdot,\ell};\ell\geq1 ),
\end{equation}
that takes values in $\Y$.  It follows from the Skorokhod representation theorem that there exists a probability space that supports $\Y$-valued random elements $\tilde{\Y}_{N_{k_j}}$ and the $\Y$-valued random element $\widetilde Y$,  such that $\tilde{\Y}^{(N_{k_j})} \deq Y^{(N_{k_j})}$ for every $j$, $Y\deq\widetilde Y$, and as $j \rightarrow \infty$, $\tilde{\Y}^{(N_{k_j})} \to\widetilde Y$ almost surely in $\Y$.   With a slight abuse of notation, since we are only interested in distributional properties, we denote the subsequence $\{N_{k_j}\}$ just as $\{N\}$ and identify $\Ynt$ and $\widetilde Y$ with $\Yn$ and $Y$, respectively.  Using this convention, we have
\begin{equation}\label{convproof_conv}
  \Yn\to Y \quad\quad \text{ in } \Y, \quad \text{ a.s.}
\end{equation}

Now, we uniquely characterize the subsequential limit $Y$. Fix $\ell\geq1.$  For $f\in\C_b\supint$, by \eqref{convproof_conv}, $\langle f,\nunbar_\ell\rangle$  converge almost surely to $\langle f,\nu_\ell\rangle$ in $\D$. Since for every $N\in\N$, the maximum jump size of $\langle f,\nunbar_\ell\rangle$ is bounded by $\|f\|_\infty/N$ (due to Lemma \ref{lem_condInd}) the limit $\langle f, \nu_\ell\rangle $ is continuous, and hence $\nu$ is a continuous $\s$-valued process, almost surely. Next, let  ${\cal T}$ be a countable dense subset of $\R_+$ which contains $0$ (say the diadic numbers). For $t\in{\cal T}$, it follows from Proposition \ref{prop_Alimit}, with $\varphi = \f1$,  that the limit $\mathcal{D}_{\f1,\ell}$ of $\dnbar_\ell=\Quenbar_{\f1,\ell}$ takes the form
\begin{equation}\label{ConvProof_temp1}
  \mathcal{D}_{\f1,\ell} (t)=\int_0^t\langle h,\nu_\ell(s)\rangle ds<\infty.
\end{equation}
Therefore,  sending  $N\to\infty$ on both sides of \eqref{prelim_dymamics_balance} in  Proposition \ref{prop_Blimit},
for every $t \in {\cal T}$, the identity
\begin{equation}\label{ConvProof_Balance}
    \langle\f1,\nu_\ell(t)\rangle - \langle \f1,\nu_\ell(0)\rangle = D_{\ell+1}(t)+\int_0^t\langle \f1,\eta_\ell(s)\rangle ds -D_\ell(t),
\end{equation}
holds almost surely, where $\eta$ is defined by \eqref{Fluid_R}. Moreover, almost surely the relations \eqref{ConvProof_temp1} and \eqref{ConvProof_Balance} hold simultaneously for all $\ell\geq1$ and $t\in{\cal T}$ because ${\cal T}$ is countable, and therefore for all $t \geq 0$ since  both sides are continuous by Propositions \ref{prop_Alimit} and \ref{prop_Blimit}.

Furthermore, let $\mathcal{C}$ be a countable dense subset of $\phiset$, and fix $\varphi\in{\cal C}$ and $t\in{\cal T}$. By Proposition \ref{prelim_dynamics_prop}, for every $\ell\geq 1$ and $N\in\N$, $\Yn$ satisfies the equation \eqref{prelim_dymamics_eq}. Since $\varphi,$ $\varphi_x$ and $\varphi_s$ are all bounded continuous functions, $\langle \varphi(\cdot,t),\nunbar_\ell(t)\rangle$ and $\int_0^t\langle \varphi_s(\cdot,s)+\varphi_x(\cdot,s),\nunbar_\ell(s)\rangle ds$ converge almost surely to $\langle \varphi(\cdot,t),\nu_\ell(t)\rangle$ and $\int_0^t\langle \varphi_s(\cdot,s)+\varphi_x(\cdot,s),\nu_\ell(s)\rangle ds$, respectively. Also, as we have already shown, $\dnbar_{\ell+1}$ converges to $D_{\ell+1}$ almost surely in $\D$, and therefore, the associated sequence of Stieltjes integrals $\int_{[0,t]}\varphi(0,s)d\dnbar_{\ell+1}(s)$ converges almost surely to $\int_{[0,t]}\varphi(0,s)d D_{\ell+1}(s)$. Finally, by Proposition \ref{prop_Alimit}, the sequence  $\Quenbar_{\varphi,\ell}(t)$ converges to $\int_0^{t}\langle\varphi(\cdot,s)h(\cdot),\nu_\ell(s)\rangle ds$,  and  by Proposition \ref{prop_Blimit},  $\Rnbar_{\varphi,\ell}(t)$ converge to $\int_0^t\langle \varphi(\cdot,s),\eta_\ell(s)\rangle ds$.   Sending $N\to\infty$ on both sides of \eqref{prelim_dymamics_eq} we then have
\begin{align}\label{ConvProov_Dynamic}
  \langle \varphi(\cdot, t),\nu_\ell(t)\rangle  = & \langle \varphi(\cdot, 0),\nu_\ell(0)\rangle + \int_0^t \langle \varphi_x(\cdot, s)+\varphi_s(\cdot, s),\nu_\ell(s)\rangle ds \notag\\
  & - \int_0^t \langle \varphi(\cdot, s)h(\cdot),\nu_\ell(s)\rangle ds + \int_0^t\varphi(0,s)dD_{\ell+1}(s)\notag\\
  &+ \int_0^t \langle \varphi(\cdot, s),\eta_\ell(s)\rangle ds,
\end{align}
almost surely. The equation \eqref{ConvProov_Dynamic} above holds with probability one, simultaneously for all $\ell\geq1$, $\varphi\in{\cal C}$ and $t\in{\cal T}$, because both $\cal C$ and $\cal T$ are countable. Moreover, since both sides of the equation above are continuous functions of $t$, the identity holds simultaneously for all $t\geq 0$, and since  $\nu_\ell$, $\Que_{\cdot,\ell}$ and ${\cal R}_{\cdot,\ell}$ are finite Radon measures, the identity holds simultaneously for all $\varphi\in\phiset$ using the Dominated Convergence Theorem.

Consequently, it follows from \eqref{ConvProof_temp1},\eqref{ConvProof_Balance} and \eqref{ConvProov_Dynamic} and Proposition \ref{prop_SolutionForm} that $\nu$ is a solution to the hydrodynamic equations \eqref{Fluid_bound}-\eqref{Fluid_R} associated to $(\lambda,\nu(0))$, which is proved to be unique in Theorem \ref{thm_uniqueness}.  This provides a unique characterization of all subsequential limits of $\{ \nunbar\}$ and completes the proof.
\end{proof}

\begin{proof}[Proof of Corollary \ref{cor_PropofChaos}]
Since queues and servers are homogeneous and the routing algorithm is symmetric with respect to the queue indices, the queue lengths and age distributions remain exchangeable for all finite times $t\geq0$.
In particular,  for any permutation $\pi:\{1, \ldots, N\} \mapsto \{1, \ldots, N\}$,
\begin{equation}\label{temp_cor1}
    \left( \xni(t);i=1,...,N \right)\deq\left( \xnX{\pi(i)}(t);i=1,...,N \right).
\end{equation}
Recall that $\nunone_\ell(t)$  is the number of queues  of length of at least $\ell$, that is, with $\xni(t)\geq 1$. Therefore,
\begin{equation}\label{cortemp3}
  \Ept{\nunonebar_\ell(t)} = \frac{1}{N}\Ept{\sum_{i=1}^N \indic{\xni(t)\geq\ell}}=\frac{1}{N}\sum_{i=1}^N \Prob{\xni(t)\geq\ell}=\Prob{\xnX{1}(t)\geq\ell},
\end{equation}
where the last equality is due to \eqref{temp_cor1}. By Theorem \ref{thm_convergence}, and since the solution $\nu$ to the hydrodynamic equations is continuous, for every $t\geq0$, $\nunbar(t) \Rightarrow \nu(t)$ in $\s$.  Hence, by the continuous mapping theorem, $\nunonebar_\ell(t) \Rightarrow S_\ell(t)$. Since $\sup_N\nunonebar_\ell(t)$ is bounded by $1$,
$\{\nunonebar_\ell\}$ is uniformly integrable and so the convergence also holds in expectation.

To prove the second claim, note that by \eqref{temp_cor1},
\begin{align}\label{temp_cor2}
  \Ept{\prod_{m=1}^k \nunonebar_{\ell_m}(t)} =& \frac{1}{N^k}\Ept{\sum_{i_1=1}^N, \ldots, \sum_{i_k=1}^N \indic{\xnX{i_1}(t)\geq\ell_1}, \ldots, \indic{\xnX{i_k}(t)\geq\ell_k}}\\
  =& \frac{1}{N^k}\sum_{i_1=1}^N...\sum_{i_k=1}^N \Prob{\xnX{i_1}(t)\geq\ell_1,\ldots, \xnX{i_k}(t)\geq\ell_k}\notag\\
  =& \Prob{\xnX{1}(t)\geq \ell_1, \ldots,\xnX{k}(t)\geq \ell_k}. \notag
\end{align}
Since $\nunbar(t) \Rightarrow \nu(t)$ in $\s$, by the continuous mapping theorem, $\prod_{m=1}^n \nunonebar_{\ell_m}(t) \Rightarrow \prod_{m=1}^n S_{\ell_m}(t)$ and,  since $\sup_N \prod_{m=1}^n \nunonebar_{\ell_m}(t)$ is bounded by $1$, the convergence also holds in expectation. Taking the limit as $N\to\infty$ of both sides of \eqref{temp_cor2}, \eqref{propOfChaos} follows.
\end{proof}

\noindent
{\bf Acknowledgments. }
The first author was partially supported by NSF grant CMMI-1538706.  The second author  was partially supported by ARO grant W911NF-12-1-0222. Both authors would like to thank Microsoft Research New England for their hospitality during the Fall of 2014, when part of this research was conducted.

\appendix

\section{A Marked Point Process Representation}\label{secapx_mpp}
 For conciseness, throughout this section, we fix $N$ and suppress the superscript $(N)$, and also  assume that the arrival process $\arrivaln$ is a pure renewal process with inter-arrivals $u^{(N)}_n;n\geq1$. The extension to a delayed renewal process is straightforward.

\subsection{Construction of State and Auxiliary Processes}\label{secapx_Adapted}

Recall that $I_0$ defined in \eqref{def_I0} is the initial state of the network, $u_j,j=1,2,...$ are inter-arrival times, $\rqc_j=(\rqc_{j}(1),...,\rqc_{j}(d)),j=1,2,...$ are queue indices randomly chosen by job $j$ upon arrival and $v_j;j\in\Z,$ are service times, and define $Data$ to be the vector of all input data,
\begin{equation}\label{data}
    Data \doteq \left(I_0,u_j,\rqc_{j};j\geq1,v_j,j\in\Z\right),
\end{equation}
and $\filtData$ to be the associated $\sigma$-algebra:
\begin{equation}\label{def_filtData}
  \filtData\doteq\sigma(\initial)\vee\sigma(u_j,\rqc_{j};j=1,2,...)\vee\sigma(v_j;j\in\Z).
\end{equation}

\begin{lemma}\label{lem_construct}
  Suppose, almost surely, all inter-arrival times and service times are positive  and the arrival process is non-explosive. Then, we can construct $\arrivalTimenon_j,\kinv_j$ and $\dept_j$ for all jobs $j$, and the process $\arrivali$ and $\di$ on $\hc$ for all queues $i=1,...,N$, as measurable functions of $Data$. In particular, $\filt_{t}\subset\filtData$ for all $t\geq0$.
\end{lemma}

\begin{proof}
First, define $\arrivalTimenon_0=0$ and note that the arrival time $\arrivalTimenon_j$ of a job $j\geq0$ satisfies $\arrivalTimenon_j=\sum_{j'=1}^ju_{j'}.$ We define $\station_j$, $\kinv_j$, $\dept_j$  and processes $\xii,\arrivali$ and $\di$  on $(\arrivalTimenon_j,\arrivalTimenon_{j+1}]$ inductively. First, define $\arrivali(0)=0$ and $\di(0)=0$ for $i=1,...,N$, and recall that  the random variables $\xii(0)$,  $i=1,...,N,$ and $\station_j$, $\kinv_j$, $\dept_j$, $j\leq0,$ (initially in service) are measurable with respect to  $\sigma (I_0)$.  Now for  $j_*\geq0$, assume $\station_j$, $\kinv_j$, $\dept_j$, $j\leq j_*$, and $\xii,\arrivali$ and $\di$ on $[0,\arrivalTimenon_{j_*}]$ are defined as  measurable functions of $I_0$.  Define
\[\di_{j_*}(t)=\sum_{j=-X(0)+1}^{j_*}\indic{\dept_j\leq t}\indic{\station_j=i},\quad\quad t\geq0,\]
to be the cumulative departure process of jobs from queue $i$ with indices no greater than  $j_*$. By the induction hypothesis,  this is a measurable function of $Data$.  Since service times are positive almost surely,  the departure time $\dept_j$ of any job $j\geq j_*+1$ satisfies $\dept_j \geq \arrivalTimenon_j+v_j>\arrivalTimenon_{j_*+1}$,  and hence, $\d  = \di_{j_*}$ on $[0,\arrivalTimenon_{j_*+1}]$. Thus, $\di$ can be extended to $[0,\arrivalTimenon_{j_*+1}]$ as a measurable function of $Data$. Moreover, since $\arrival$ and hence $\arrivali$, $i=1,...,N$, are piecewise constant on $(\arrivalTimenon_{j_*},\arrivalTimenon_{j_*+1}),$ the length of  queue $i$ right before the arrival of the $\Xth{(j_*+1)}$ job can be obtained from the mass balance equation \eqref{mass_balance_i}:
\[\xii(\arrivalTimenon_{j_*+1}-)=\xii(0)+ \arrivali(\arrivalTimenon_{j_*+1}-)+\di(\arrivalTimenon_{j_*+1}-)=\xii(0)+ \arrivali(\arrivalTimenon_{j_*})+\di(\arrivalTimenon_{j_*+1}-),\]
and  so $\xii(\arrivalTimenon_{j_*+1}-)$  is also $\filtData$-measurable.  Therefore, both the queue index $\station_{j_*+1}$ to which the job $j_*+1$ is routed, as defined by \eqref{def_station},  and the process $\arrivali(t)=\sum_{j=1}^{j_*+1}\indic{\arrivalTimenon_j\leq t}\indic{\station_j=i}$ on $[0,\arrivalTimenon_{j_*+1}],$ are  measurable functions of $Data$. The job $j_*+1$  joins the back of the queue (if the queue is not empty), and enters service when the service requirement of all the jobs ahead of it in the same queue is completed, that is at time
\[\kinv_{j_*+1}=\arrivalTimenon_{j_*+1}+\sum_{j=-X(0)+1}^{j_*}\left(v_j-a_j(\arrivalTimenon_{j_*+1})\right) \indic{\station_j=\station_{j_*+1}},\]
and departs at time $\dept_{j_*+1}=\kinv_{j_*+1}+v_{j_*+1}.$ Using these relations and the $\filtData$-measurability of  the age processes $a_j, j < j_*,$ defined by \eqref{def_agen},  due to the induction hypothesis, it follows that $\kinv_{j_*+1}$ and $\dept_{j_*+1}$ are also measurable functions of $Data$. This completes the  induction argument.  Finally, note that since the arrival process is non-explosive, $\lim_{j\to\infty}\arrivalTimenon_j=\infty,$ and hence, the above construction holds  for the whole of $\hc$.
\end{proof}

\begin{remark}
   The assumptions of Lemma \ref{lem_construct} hold under our Assumptions \ref{asm_E} and \ref{asm_G}.\ref{asm_mean} because inter-arrival and service time distributions have densities $\gee$ and $g$, respectively, and the arrival process is non-explosive by Corollary \ref{cor_distinctness}.
\end{remark}

Next, we prove that the state and auxiliary processes defined in Section \ref{sec_model} are $\{\filt_t\}$-adapted.

\begin{proposition}\label{prop_adapted}
    Under the assumptions of Lemma \ref{lem_construct}, for every job $j,$ the arrival time $\arrivalTimenon_j$, service entry time $\kinv_j$ and departure time $\dept_j$  are $\{\filt_t\}$-stopping times, and the age and queue index processes $a_j(\cdot)$ and $\station_j(\cdot)$ are $\{\filt_t\}$-adapted. Also, the processes $\xii, i=1,...,N,$   $D_\ell,$ and $\nu_\ell,\ell\geq 1,$ are  $\{\filt_t\}$-adapted.
\end{proposition}

\begin{proof}
First, note that $\arrival(t)=\sum_{i=1}^N\arrivali(t)$ is $\filt_t$-adapted, and since $\xii(0)$ is $\filt_0$-measurable, the mass balance equation \eqref{mass_balance_i} for queue $i$ implies that $\xii$ is also adapted. Moreover, $\arrivali,$ $\di$ and $\xii$ are all right-continuous, and hence $\{\filt_t\}$-progressive.

Noting that  $\{\arrivalTimenon_j\leq t\} = \{\arrival(t)\geq j\},$ and   $\arrival(t)$ is $\filt_t$-measurable, we see that $\arrivalTimenon_j$ is an $\{\filt_t\}$-stopping time. Also, since there are not multiple arrivals at the same time, the queue $\station_j$ to which job $j$ is routed is the unique queue whose arrival process increases at $\arrivalTimenon_j$, that is,
\[
    \station_j=\text{argmax}_{i\in\{1,...,N\}} \left(\arrivali(\arrivalTimenon_j)-\arrivali(\arrivalTimenon_j-)\right).
\]
Since $\arrivali, i=1,...,N,$ are $\{\filt_t\}$-progressive, each argument of the argmax above is $\filt_{\arrivalTimenon_j}$-measurable, and hence, so is $\station_j$. Therefore, $\station_j(t)=\indicil{\arrivalTimenon_j\leq t}\station_j$ is $\{\filt_t\}$-adapted for $j\geq1$. For jobs initially in network, that is $j\leq0$, for all $t\geq0$, $\station_j(t)$ is equal to the $\filt_0$-measurable quantity $\station_j$, and hence is also $\{\filt_t\}$-adapted.

When a job $j\geq 1$ arrives at time $\arrivalTimenon_j$ and joins the queue $\station_j$, there are  $(X^{\station_j}(\arrivalTimenon_j)-1)$ jobs ahead of it in that queue. Hence, for job $j$ to have entered service by  time $t$, it must have arrived at a queue prior to $t$ and all the jobs ahead it in the same queue must have already departed,  or, in other words,
\[
    \{\kinv_j \leq t\}  =  \{\arrivalTimenon_j\leq t\}\cap\{D^{\station_j}(t)-D^{\station_j}(\arrivalTimenon_j)\geq X^{\station_j}(\arrivalTimenon_j)-1 \},\quad\quad j\geq1.
\]
Since $\arrivalTimenon_j$ is an $\{\filt_t\}$-stopping time, $\di$ and $\xii$, $i = 1, \ldots, N$, are $\{\filt_t\}$-progressive, and $\station_j$ is $\filt_{\arrivalTimenon_j}$-measurable, and $D^{\station_j}(\arrivalTimenon_j)$ and $X^{\station_j}(\arrivalTimenon_j)$ are $\filt_{\arrivalTimenon_j}$-measurable, it follows that  $D^{\station_j}(t)$, $D^{\station_j}(\arrivalTimenon_j)$ and $X^{\station_j}(\arrivalTimenon_j)$ are $\filt_t$-measurable on $\{\arrivalTimenon_j\leq t\}$. Therefore, $\{\kinv_j\leq t\}$ is $\filt_t$-measurable and  $\kinv_j$ is  an $\{\filt_t\}$-stopping time. Similarly,
\[
        \{\dept_j \leq t\}  =  \{\arrivalTimenon_j\leq t\}\cap\{D^{\station_j}(t)-D^{\station_j}(\arrivalTimenon_j)\geq X^{\station_j}(\arrivalTimenon_j) \},\quad\quad j\geq1,
\]
and thus, $\dept_j$ is also an  $\{\filt_t\}$-stopping time. For jobs initially in network ($j\leq0$), the queue index $\station_j$ and the number of jobs $p_j$ ahead of job $j$ in station $\station_j$ at time $0$ are $\filt_0$-measurable, and since $\{\kinv_j\geq t\}=\{D^{\station_j}(t)\geq p_j\}$ and $\{\dept_j\geq t\}=\{D^{\station_j}(t)\geq p_j+1\}$, $\kinv_j$ and $\dept_j$ for $j\leq 0$ are also $\{\filt_t\}$-stopping times. Consequently, the age process $\agen_j$ defined in \eqref{def_agen} is $\{\filt_t\}$-adapted for all $j$.

Finally, recall that $\custServSizenon_j(t)$ is the queue size observed by job $j$ at time $t\geq \arrivalTimenon_j$, defined in \eqref{def-custservsize}. Since $\station_j$ is $\filt_{\arrivalTimenon_j}$-measurable, $\indic{\arrivalTimenon_j\leq t}\custServSizenon_j(t)$ is $\{\filt_t\}$-adapted, and also $\{\filt_t\}$-progressive, as it is right-continuous. Therefore, since $\arrivalTimenon_j\leq\dept_j$, $\custDSnon_j=\xX{\station_j}(\dept_j-)$ is $\filt_{\dept_j}$-measurable and hence $D_\ell$ defined in \eqref{def_Dl} is $\{\filt_t\}$-adapted. Moreover, since $\dept_j$ and $\age_j$ are $\{\filt_t\}$-stopping times, $a_j(\cdot)$ is $\{\filt_t\}$-adapted and $\{\kinv_j\leq t\}\subset\{\arrivalTimenon_j\leq t\}$, $\nu_\ell$ defined in \eqref{nuln_alternative} is $\{\filt_t\}$-adapted.
\end{proof}

\subsection{Preliminary Independence Results}

We first list two elementary generic lemmas on conditioning in Section \ref{subsub-elem}.  These are then used
in Section \ref{subsub-indep2} to obtain some preliminary conditional independence results on the model.

\subsubsection{Elementary Lemmas}\label{subsub-elem}

Recall  that given a sigma algebra $\filt$ and a subset $A,$ the trace of $\filt$ on $A$ is defined as $\filt\wedge A=\{A\cap B: B\in\filt\}.$  The following lemma is used in the proof of Lemma \ref{lem_independence}.

\begin{lemma}\label{lem_condAux}
Given a probability space $(\Omega,\filt,\mathbb{P})$, suppose sub $\sigma$-algebras $\cal G_c,\cal G_f\subset\filt$ and a set $A\in\cal G_c\cap \cal G_f$ are such such that $\cal G_c\wedge A\subset \cal G_f\wedge A.$ Then, for every integrable random variable $X$,
\begin{equation}\label{cond_lem_1}
        \indic{A}\Ept{X|\cal G_c}=\indic{A}\Ept{\Ept{X|\cal G_f}|\cal G_c }.
\end{equation}
\end{lemma}

\begin{proof}
For  $B\in\cal G_c$, by definition, $A\cap B\in\cal G_c\wedge A$.  Since  $\cal G_c\wedge A\subset \cal G_f\wedge A,$ we have $A\cap B\in\cal G_f\wedge A$.  Hence,
\[
    \Ept{\indic{B}\indic{A}\Ept{X|\cal G_f}}=\Ept{\indic{B\cap A}\Ept{X|\cal G_f}}=\Ept{\Ept{\indic{B\cap A}X|\cal G_f}}=\Ept{\indic{B}\indic{A}X},
\]
where the second equality uses $A\cap B\in\cal G_f\wedge A \subset G_f$ and the last equality follows by the tower property of conditional expectation. Hence, we have shown $\Ept{\indic{A}X|\cal G_c}=\Ept{\indic{A}\Ept{X|\cal G_f}|\cal G_c }$. Since $A\in\cal G_c$,  \eqref{cond_lem_1} follows.
\end{proof}

Let  $W_1,W_2,...,W_n$ and $Y$ be $\R\cup\{\infty\}$-valued random variables defined on a probability space $(\Omega,\filt,\mathbb{P})$, and let $\cal G\subset\cal F$ be a $\sigma$-algebra such that $W_i,i=1,...,n,$ and $Y$ are conditionally independent given $\mathcal{G}$. For $i=1,...,n,$ define
\[ \overline F_i(a) = \Prob{W_i>a|\cal G},  \quad\quad  a \geq 0. \]
Also, define $T\doteq \min(W_1,...,W_n,Y)$ and let $Z$ be a discrete-valued random variable such that $\{Y<\min(W_1,...,W_n)\}=\{Z= z_0\}$ for some value $z_0$.  The following lemma is used in the proof of Lemma \ref{lem_condInd}, when $W_i$'s and $Y$ are replaced by certain arrival and departure times.

\begin{lemma}\label{lem_condauxind}
Suppose $W_1,...,W_n,Y,T$ and $Z$ are as described above. Then, on $\{Z=z_0\}$,  $W_i-T;i=1,...,n,$ are  conditionally independent given $\cal G$, $T$ and $Z$, that is, for every $b_1,...,b_n\geq0$,
\begin{equation}\label{condAuxInd}
   \indic{Z=z_0}\Prob{W_i>T+b_i;i=1,...,n\big| \mathcal{G},T,Z}=   \indic{Z=z_0}\prod_{i=1}^n\Prob{W_i>T+b_i\big| \mathcal{G},T,Z},
\end{equation}
and,  for $b \geq 0$,
\begin{equation}\label{condAuxDist}
     \indic{Z=z_0}\Prob{W_i>T+b\big| \mathcal{G},T,Z} = \indic{Z=z_0}\frac{\overline F_i(T+b)}{\overline F_i(T)}.
\end{equation}
\end{lemma}

\begin{proof}
First, we claim that for every $B\in\mathcal{B}\hc$,
\begin{align}\label{temp_cond}
    \Ept{\indic{T\in B}\indic{Z=z_0}\prod_{i=1}^n\indic{W_i>T+b_i}\Big|\mathcal{G}}=\Ept{\indic{T\in B}\indic{Z=z_0}\prod_{i=1}^n\frac{\overline F_i(T+b_i)}{\overline F_i(T)}\Big|\mathcal{G}}.
\end{align}
To prove the claim, note that since $T=Y$ on $\{Z=z_0\}$ and $W_i > Y +b_i, i = 1, \ldots, n$ implies $Z=z_0$,  we have
\[\{T\in B,Z=z_0,W_i>T+b_i;i=1,...,n\}=\{Y\in B,W_i>Y+b_i;i=1,...,n\},\]
Then  the conditional independence of $W_1,...,W_n$ and $Y$ implies the left-hand side of \eqref{temp_cond} is equal to
\begin{align*}
    \Prob{Y\in B,W_i>Y+b_i,i=1,...,n\big|\mathcal{G}}  =  \Ept{\indic{Y\in B} \prod_{i=1}^n\overline F_i(Y+b_i)\Big|\mathcal{G}}.
\end{align*}
Similarly, since $\{T\in B,Z=z_0\}=\{ Y = T \in B,W_i>Y,i=1,...,n\}$,  by the conditional independence of $W_1,...,W_n$ and $Y$ given $\cal G$, the right-hand side of \eqref{temp_cond} is equal to
\begin{align*}
    \Ept{\indic{Y\in B,W_i>Y,i=1,...,n}\prod_{i=1}^n\frac{\overline F_i(Y+b_i)}{\overline F_i(Y),}\Big|\mathcal{G}}
        =\Ept{\indic{Y\in B}  \prod_{i=1}^n \overline F_i(Y+b_i)\Big|\mathcal{G}},
\end{align*}
and the claim \eqref{temp_cond} follows. Moreover,  by definition of conditional expectations,
\eqref{temp_cond} implies
\begin{equation}\label{temp_con2}
    \indic{Z=z_0}\Prob{W_i>T+b_i;i=1,...,n\big| \mathcal{G},T,Z}=   \indic{Z=z_0}\prod_{i=1}^n\frac{\overline F_i(T+b_i)}{\overline F_i(T)}.
\end{equation}
Now, substituting $b_i=b$ and $b_{i'}=0,i'\neq i$, in \eqref{temp_con2} and observing that for all $i$, $W_i>T$ on $\{Z=z_0\}$,  \eqref{condAuxDist} follows. Finally, \eqref{condAuxInd} is obtained by substituting each term of the product on the right-hand side of \eqref{temp_con2} from the equation \eqref{condAuxDist}.
\end{proof}

\subsubsection{Some Conditional Independence Results of the Model}
\label{subsub-indep2}

Recall that by the non-idling condition, when $\xii(t)\geq 1$ for some $t\geq0$, there exists a job receiving service at queue $i$ at time $t$, which we denote by $J^i(t)$. Note that $\indicil{\xii(t)\geq1}J^i(t)$ is  well defined for all $t\geq0$. Also, recall from Section \ref{sec_mpp} the mark $z_k$ associated with the event time $\tau_k$; with $z_k = (\mathfrak{E},i)$ (respectively, $z_k = (\mathfrak{D},i)$) indicating that the event is the arrival of a job to queue $i$ (respectively, departure of a job from queue $i$).

\renewcommand{\theenumi}{\alph{enumi}}
\begin{lemma}\label{lem_independence}
  Suppose Assumptions \ref{asm_E},  \ref{asm_G}.\ref{asm_mean}, and \ref{asm_initial}.\ref{asm_initial_ind} hold, and fix $k\geq1$ and $i\in\{1,...,N\}$. Then, on $\{\eventmark_k=(\mathfrak{E},i)\}$, the next interarrival time $u_{\arrival(\eventtime_k)+1}$ and the service time $v_{\arrival(\eventtime_k)}$ are independent of $\filt_{\eventtime_k}$, and for every $b\geq0$,
  \begin{equation}\label{indep_UE}
    \indic{\eventmark_k=(\mathfrak{E},i)}\Prob{u_{\arrival(\eventtime_k)+1}>b\Big|\filt_{\eventtime_k}} = \indic{\eventmark_k=(\mathfrak{E},i)} \bGe(b),
  \end{equation}
  and
  \begin{equation}\label{indep_VE}
    \indic{\eventmark_k=(\mathfrak{E},i)}\Prob{v_{\arrival(\eventtime_k)}>b\Big|\filt_{\eventtime_k}} = \indic{\eventmark_k=(\mathfrak{E},i)} \overline G(b).
  \end{equation}
  Moreover, on $\{\eventmark_k=(\mathfrak{D},i),\xii(\eventtime_k)\geq1\}$, $v_{J^i(\eventtime_k)}$ is independent of $\filt_{\eventtime_k}$ and for every $b\geq0$,
  \begin{equation}\label{indep_VJ}
    \indic{\eventmark_k=(\mathfrak{D},i),\xii(\eventtime_k)\geq1}\Prob{v_{J^i(\eventtime_k)}>b\Big|\filt_{\eventtime_k}} =\indic{\eventmark_k=(\mathfrak{D},i),\xii(\eventtime_k)\geq1} \overline G(b).
  \end{equation}
\end{lemma}

We prove Lemma \ref{lem_independence} at the end of this section, after we prove the following intuitive auxiliary  results.   First, on the set $A_{k,i,n}$ where the $\Xth{k}$ event is the arrival of  job  $n$,
\begin{equation}\label{queue_Akn}
    A_{k,i,n}\doteq\{\eventmark_{k}=(\mathfrak{E},i),\arrival(\eventtime_{k})=n\},
\end{equation}
we claim the state of the network up to time $\eventtime_k$ depends only on the arrival times and random queue choices of the first $n$ jobs, and the service times of the first $n-1$ jobs, or equivalently, is $\filtData_{\leq n}$-measurable, where
\begin{equation}\label{def_filtDataEn}
  \filtData_{\leq n}\doteq\sigma(\initial)\vee\sigma(u_j,\rqc_{j};j=1,2,...,n)\vee\sigma(v_j;j\leq n-1).
\end{equation}
Similarly, on the set
\begin{equation}\label{queue_Bkij}
B_{k,i,m}\doteq\{\eventmark_{k}=(\mathfrak{D},i),\xii(\eventtime_k)\geq1,J^i(\eventtime_k)=m\},
\end{equation}
where the $k$th event is a departure from queue $i$ and an entry into service by job $m$, we claim the state of the network up to time $\eventtime_k$ does not  depend on the service time of job $m$, or equivalently, is  $\filtData_{-m}$-measurable, where
\begin{equation}\label{def_filtDataDj}
    \filtData_{-m}\doteq\sigma(\initial)\vee\sigma(u_j,\rqc_{j};j\geq 1)\vee\sigma(v_j;j\in\Z\backslash\{m\}),
\end{equation}

\begin{remark}\label{remark_ABmble}
    Note that by the representation \eqref{Ei_mpp_rep} of $\arrivalni$ and the relation $\arrivaln=\sum_{i=1}^N\arrivalni$, $A_{k,i,n}$ is $\filt_{\eventtime_k}$-measurable.  Also, since $J^i(t)$ is the job with the largest index that has  entered service at queue $i$ before time $t$, we have
    \[\indic{\xii(t)\ge 1}J^i(t)=\indic{\xii(t)\geq1} \max\left\{j: \kinv_j\le t, \station_j=i\right\}.\]
    Since $\kinv_j$ is an $\{\filt_t\}$-stopping time and $\station_j$ is $\filt_{\kinv_j}$-measurable by Lemma \ref{lem_construct} and $X^i$ is $\{{\mathcal F}_t\}$-adapted by Proposition \ref{prop_adapted}, the right-hand side above is $\{\filt_t\}$-adapted,  and hence $\{\filt_t\}$-progressive because it is right-continuous as a function of $t$. Therefore, $\indic{\xii(\eventtime_k)\geq1}J^i(\eventtime_k)$, and hence $B_{k,i,m}$, are ${\mathcal F}_{\eventtime_k}$-measurable.
\end{remark}

\begin{lemma}\label{lem_queue}
    Suppose the assumptions of Lemma \ref{lem_construct} hold. Then, for  $k,n\geq1$, $i=1,...,N$ and $m\in \Z$,
    \begin{enumerate}
        \item  $A_{k,i,n}$ is $\filtData_{\leq n}$-measurable and $\filt_{\eventtime_{k}}\wedge A_{k,i,n}\subset \filtData_{\leq n}\wedge A_{k,i,n}.$ \label{queue_E}

        \item $B_{k,i,m}$ is $\filtData_{-m}$-measurable and $\filt_{\eventtime_{k}}\wedge B_{k,i,m}\subset \filtData_{-m}\wedge B_{k,i,m}.$\label{queue_D}
        \end{enumerate}
\end{lemma}

\begin{proof}
For part \ref{queue_E}, consider an identical queuing network driven by the alternative data
\[\widetilde{Data}=(\tilde\initial,\tilde u_j,\tilde \rqc_{j};j\geq1,\tilde v_j;j\in\Z),\]
coupled with our original network so that $\tilde \initial=\initial,$ $\tilde u_j=u_j, \tilde \rqc_{j}=\rqc_{j};j=1,2,...,n$ and $\tilde v_j=v_j;j\leq n-1,$ but $\tilde u_j=1$ and $\tilde{\rqc}_{j}(d')= {\rqc}_{j}(d')                                                                                                                                                                                                                            =1$ for $j\geq n+1$ and $d'=1,...,d$ and $\tilde{v}_j=1$ for $j\geq n$.  Then clearly
\begin{align*}
  \tilde{\mathcal{F}}^{\widetilde{\text{Data}}}&\doteq\sigma(\tilde\initial,\tilde u_j,\tilde \rqc_{j};j\geq1 ,\tilde v_j;j\in\Z)= \sigma(\initial,\rqc_{j}, u_j;j=1,...,n,v_j;j\leq n-1) = \filtData_{\leq n},
\end{align*}
where we distinguish quantities associated to the alternative network from those of the original network by a tilde. Clearly, the dynamics of the two networks are clearly identical up to the arrival time $\arrivalTimenon_n=\tilde \arrivalTimenon_n$ of job $n$.   In particular, $\tilde A_{k,i,n} \doteq\{\tilde{\arrival}(\tilde{\eventtime}_{k})=n,\tilde{\eventmark}_{k}=(\mathfrak{E},i)\}=A_{k,i,n}$, and since $\tilde A_{k,i,n}$  is $\tilde{\filt}_{\tilde{\eventtime}_k}$-measurable, so is $A_{k,i,n}$. But by Lemma \ref{lem_construct}, $\tilde{\filt}_{\tilde{\eventtime}_k} \subset \tilde{\mathcal{F}}^{\widetilde{\text{Data}}} = \filtData_{\leq n}$, and hence $A_{k,i,n}$ is $\filtData_{\leq n}$-measurable.

Moreover, for all $t\geq0$ and $i=1,...,n$,
\[\arrivali(t\wedge \arrivalTimenon_n)=\tilde \arrivali(t\wedge\arrivalTimenon_n),\quad\quad \di(t\wedge \arrivalTimenon_n)=\tilde \di(t\wedge\arrivalTimenon_n),\]
and hence, by \cite[Theorem T3 of Section III.1]{BremaudBook} and definition \eqref{def_filt} of the filtration $\{\filt_t;t\geq0\}$, and because  $\arrivalTimenon_j$ is an $\{\filt_t\}$-stopping time by Lemma \ref{lem_construct}, we have $\filt_{\arrivalTimenon_n}=\tilde{\filt}_{\arrivalTimenon_n}$.  In particular, since $\arrivalTimenon_n=\eventtime_k$ on $A_{k,i,n}$,  another application of Lemma \ref{lem_construct} shows
\[\filt_{\eventtime_k}\wedge A_{k,i,n} = \filt_{\arrivalTimenon_n}\wedge A_{k,i,n}=\tilde{\filt}_{\arrivalTimenon_n}\wedge A_{k,i,n}\subset\filtData_{\leq n}\wedge A_{k,i,n}.\]

Finally for part \ref{queue_D}, consider an alternative network driven by another $\widetilde{Data}$ with $\tilde \initial=\initial,$ $\tilde u_j=u_j, \tilde \rqc_{j}=\rqc_{j};j\geq1$ and $\tilde v_j=v_j;j\in\Z\backslash\{m\},$ but with $\tilde v_m=1$. Having in mind that the two systems are identical up to the service entry time $\kinv_m$ of job $m$ and that $B_{k,i,m}$ is $\filt_{\eventtime_k}$-measurable by Remark \ref{remark_ABmble} and $\kinv_m$ is an $\{\filt_t\}$-stopping time, the result follows from exactly analogous arguments as in part \ref{queue_E}.
\end{proof}

\begin{proof}[Proof of Lemma \ref{lem_independence}]
To see \eqref{indep_UE}, we partition the set $\{\eventmark_k=(\mathfrak{E},i)\}$ based on  the value of $E(\eventtime_k)$,  the index of the job that arrived at time $\eventtime_k$. In other words, with $A_{k,i,n}$ as defined in \eqref{queue_Akn}, we have
\begin{equation}\label{temp_ind1}
        \{\eventmark_k=(\mathfrak{E},i)\} =\bigcup_{n\in\N}A_{k,i,n}.
\end{equation}
Lemma \ref{lem_queue}.\ref{queue_E} and the fact that $A_{k,i,n}$ is $\filt_{\eventtime_k}$-measurable (see Remark \ref{remark_ABmble}) show that the conditions of Lemma \ref{lem_condAux} are satisfied with $A=A_{k,i,n}$, $X=\indicil{u_{n+1}>b}$, $\cal G_c=\filt_{\eventtime_k}$ and $\cal G_f=\filtData_{\leq n}$.    Therefore, using Lemma \ref{lem_condAux}, we have
\begin{equation}\label{temp_ind2}
      \indic{A_{k,i,n}}\Prob{u_{n+1}>b|\filt_{\eventtime_k}}= \indic{A_{k,i,n}}\Ept{\Prob{u_{n+1}>b|\filtData_{\leq n}}\Big|\filt_{\eventtime_k}}.
\end{equation}
By Assumptions \ref{asm_E} and  \ref{asm_initial}.\ref{asm_initial_ind}, the inter-arrival time $u_{n+1}$ has complimentary CDF $\bGe$ and is independent of the initial conditions, all other inter-arrival times and all service times and hence, independent of $\filtData_{\leq n}$. Therefore, $\Probil{u_{n+1}>b|\filtData_{\leq n}}=\bGe(b)$, which together with \eqref{queue_Akn},  \eqref{temp_ind1} and \eqref{temp_ind2}, yields
\[ \indic{\eventmark_k=(\mathfrak{E},i)}\Prob{u_{E(\eventtime_k)+1}>b|\filt_{\eventtime_k}}= \sum_{n\in\N} \indic{A_{k,i,n}}\bGe(b) =\indic{\eventmark_k=(\mathfrak{E},i)}\bGe(b).\]
This proves \eqref{indep_UE}. Similarly, by another use of Lemma \ref{lem_condAux}, with the same $A$, $\cal G_c$, $\cal G_f$ as above and with $X=\indicil{v_n>b}$, in the second equality below, we have
\begin{align*}
      \indic{\eventmark_k=(\mathfrak{E},i)}\Prob{v_{E(\eventtime_k)}>b|\filt_{\eventtime_k}}= & \sum_{n\in\N}\indic{A_{k,i,n}}\Prob{v_{n}>b|\filt_{\eventtime_k}}\\
      = & \sum_{n\in\N} \indic{A_{k,i,n}}\Ept{\Prob{v_{n}>b|\filtData_{\leq n}}\Big|\filt_{\eventtime_k}}.
\end{align*}
By Assumptions \ref{asm_G}.\ref{asm_mean} and  \ref{asm_initial}.\ref{asm_initial_ind}, $v_n$ has the complimentary CDF $\overline G$ and is independent of $I_0$, the arrival process and the service times of jobs $j<n$, and hence, is independent of  $\filtData_{\leq n}$. This  yields $\Probil{v_{n}>b|\filtData_{\leq n}}=\overline G(b)$. When combined with the last display, \eqref{indep_VE} follows from another use of \eqref{temp_ind1}.

Finally to see \eqref{indep_VJ}, we partition the set $\{z_k=(\mathfrak{D}, i ),\xii(\eventtime_k)\geq1\}$ based on the value of $J^i(\eventtime_k)$:
\begin{equation}\label{temp_ind3}
        \{z_k=(\mathfrak{D}, i),\xii(\eventtime_k)\geq1\}=\bigcup_{m\in\Z}B_{k,i,m},
\end{equation}
where $B_{k,i,m}$, defined in \eqref{queue_Bkij},   is $\filt_{\eventtime_k}$-measurable by Remark  \ref{remark_ABmble}. (Note that $B_{k,i,m}$ is empty if there is no job with index $m$.) For every $m\in\Z$, the conditions of  Lemma \ref{lem_condAux} with $A=B_{k,i,m}$, $X=\indicil{v_{m}>b}$, $\cal G_c=\filt_{\eventtime_k}$ and $\cal G_f=\filtData_{-m}$ hold due to Lemma \ref{lem_queue}.\ref{queue_D}, and hence, using Lemma \ref{lem_condAux}, we have
\begin{equation}\label{temp_ind4}
    \indic{B_{k,i,m}}\Prob{v_{m}>b|\filt_{\eventtime_k}} = \indic{B_{k,i,m}}\Ept{\Prob{v_{m}>b|\filtData_{-m}}\Big|\filt_{\eventtime_k}}.
\end{equation}
By Assumption \ref{asm_G}.\ref{asm_mean}, $v_m$ has complimentary CDF $\overline G$ and is independent of $I_0$, the  arrival process, and all other service times, and hence, is independent of $\filtData_{-m}$, which implies $\Prob{v_{m}>b|\filtData_{-m}} =\overline G(b)$. When combined with \eqref{temp_ind3} and \eqref{temp_ind4}, we obtain \eqref{indep_VJ} since
  \[
        \indic{\eventmark_k=(\mathfrak{D},i),\xii(\eventtime_k\geq1)}\Prob{v_{J^i(\eventtime_k)}>b\Big|\filt_{\eventtime_k}} = \sum_{m\in\Z} \indic{B_{k,i,m}}\overline G(b) = \indic{\eventmark_k=(\mathfrak{D},i),\xii(\eventtime_k\geq1)} G(b).
    \]
\end{proof}

Next, define $\eventmark_{k,1}\in\{\mathfrak{E},\mathfrak{D}\}$ and $\eventmark_{k,2}\in\{1,...,N\}$ to be the two components of  $\eventmark_k$, that is, $\eventmark_k=(\eventmark_{k,1},\eventmark_{k,2})$. The next lemma shows that when the next event is an arrival, $\sigma_k^i$ is conditionally independent of the queue index to which the new arriving job is routed.

\begin{lemma}\label{lem_rqcInd}
    Suppose Assumptions \ref{asm_E} and \ref{asm_G}.\ref{asm_mean} hold. Then, for every $k\ge 1$,
    \begin{align}\label{rqc_ind}
        &\indic{\eventmark_{k,1}=\mathfrak{E}}\Prob{\sigma^i_{k}>\eventtime_k+b_i,i=1,...,N|\filt_{\eventtime_k}}\notag\\
        &\hspace{1cm}= \indic{\eventmark_{k,1}=\mathfrak{E}} \Prob{\sigma^i_{k}>\eventtime_k+b_i,i=1,...,N|\filt_{\eventtime_{k-1}},\eventtime_k,z_{k,1}}.
    \end{align}
\end{lemma}

\begin{proof}
    Define $A_{k,n}=\{\eventmark_{k,1}=\mathfrak{E}, \arrival(\eventtime_k)=n\}$, and note that $\{\eventmark_{k,1}=\mathfrak{E}\}=\cup_{n\in\N}A_{k,n}$. On $A_{k,n}$, $\eventmark_{k,2}$ is the queue index of the newly arrived job $n$, that is, $\eventmark_{k,2}=\station_{n}$,  where $\station_n$ is defined by  \eqref{def_station} and is a measurable function of $\rqc_n$ and $\xii(\eventtime_k),i=1,...,N$. Since $\xii(\eventtime_{k-1}),i=1,...,N,$ are  $\filt_{\eventtime_{k-1}}$-measurable by Lemma \ref{lem_construct}, we have \[\filt_{\eventtime_k}\wedge A_{n,k}=(\filt_{\eventtime_{k-1}}\vee\sigma(\eventtime_{k},\eventmark_{k}))\wedge A_{k,n}=(\filt_{\eventtime_{k-1}}\vee\sigma(\eventtime_{k},\eventmark_{k,1},\rqc_n))\wedge A_{k,n},\]
    and therefore,
    \begin{align}\label{temp_rqc}
        &\indic{A_{k,n}}\Prob{\sigma^i_{k}>\eventtime_k+b_i,i=1,...,N|\filt_{\eventtime_k}}\notag\\
        &\hspace{1cm}= \indic{A_{k,n}} \Prob{\sigma^i_{k}>\eventtime_k+b_i,i=1,...,N|\filt_{\eventtime_{k-1}},\eventtime_k,z_{k,1},\rqc_n}.
    \end{align}
    Moreover, for $i\in\busy_{k-1}$, $\sigma_{k}^i$ is the departure time of the job $J^i(\eventtime_{k-1})$ that is receiving service at queue $i$ at time $\eventtime_{k-1}$, and hence, $\sigma_{k}^i=\kinv_{J^i(\eventtime_{k-1})}+v_{J^i(\eventtime_{k-1})}$. By Remark \ref{remark_ABmble}, $J^i(\eventtime_{k-1})$ is $\filt_{\eventtime_{k-1}}$-measurable, and since $\kinv_j$ is a $\{\filt_t\}$-stopping time for all $j\in\Z$ and $\kinv_{J^i(\eventtime_{k-1})}\leq \eventtime_{k-1}$, $\kinv_{J^i(\eventtime_{k-1})}$ is $\filt_{\eventtime_{k-1}}$-measurable. Also, since service times are independent of the random queue choices, $v_{J^i(\eventtime_{k-1})}$ is conditionally independent of $\rqc_n$, given $\filt_{\eventtime_{k-1}}$. On the other hand, if $i\notin\busy_{k-1}$, $\sigma_k^i=\infty$. Consequently, for every $i=1,...,N$, $\sigma_k^i$ is conditionally independent of $\rqc_n$, given $\filt_{\eventtime_{k-1}}$, and therefore,  \eqref{temp_rqc} yields
    \begin{align*}
        \indic{A_{k,n}}\Prob{\sigma^i_{k}>\eventtime_k+b_i,i=1,...,N|\filt_{\eventtime_k}}
        = \indic{A_{k,n}}
        \Prob{\sigma^i_{k}>\eventtime_k+b_i,i=1,...,N|\filt_{\eventtime_{k-1}},\eventtime_k,z_{k,1}}.
    \end{align*}
    Equation \eqref{rqc_ind} follows on summing  both sides of the above display over $n\in\Z$.
\end{proof}

\subsection{Proof of Lemma \ref{lem_condInd} }\label{secapx_condind}

\begin{proof}[Proof of Lemma \ref{lem_condInd}]
We prove the lemma using induction on $k$. As the induction hypothesis, assume that for some $k \geq 1$,  $\Prob{\tilde\Omega_{k-1}}=1$,  the next arrival time $\xi_{k}$ and next departure times $\sigma^{i}_{k}$, $i=1,...,N,$ after $\tau_{k-1}$ are conditionally independent given $\filt_{\eventtime_{k-1}}$, and \eqref{cond_density_arrival} and \eqref{cond_density_departure} hold with $k$ replaced by $k-1$, that is,  for every $b>0,$
    \begin{equation}\label{cond_CDF_arrival}
        \Prob{\xi_{k}-\eventtime_{k-1}> b\big|\filt_{\eventtime_{k-1}}} = \frac{\bGe(\re(\eventtime_{k-1})+b)}{\bGe(\re(\eventtime_{k-1}))},
    \end{equation}
    and for every $i=1, \ldots, N$,
    \begin{equation}\label{cond_CDF_departure}
        \indic{i\in\busy_{k-1}}\Prob{\sigma^{i}_{k}-\eventtime_{k-1}> b\big|\filt_{\eventtime_{k-1}}} = \indic{i\in\busy_{k-1}}\frac{\overline G(\agei(\eventtime_{k-1})+b)}{\overline G(\agei(\eventtime_{k-1}))}.
    \end{equation}

First, we prove $\Probil{\tilde\Omega_k}=1.$ Since $\Probil{\tilde \Omega_{k-1}}=1$,
by \eqref{next_event} we have, almost surely,
\[
        \eventtime_k-\eventtime_{k-1}=\min\{\xi_{k}-\eventtime_{k-1}, \sigma^{i}_k-\eventtime_{k-1}; i=1,...,N\},
\]
By the induction hypothesis, given $\filt_{\eventtime_{k-1}}$, the random variables $\xi_k-\eventtime_{k-1}$ and $\sigma^i_k-\eventtime_{k-1}, i=1,...,N,$  are conditionally independent, and by \eqref{cond_CDF_arrival} and \eqref{cond_CDF_departure} and the fact that  $\Ge$ and $\overline G$ have densities $\gee$ and $g$ by Assumptions \ref{asm_E} and  \ref{asm_G}.\ref{asm_mean}, the random variables are all distinct and strictly positive, taking values in  $(0,\infty]$.  In particular,  this shows that  $\Prob{\eventtime_k>\eventtime_{k-1}|\filt_{\eventtime_{k-1}} }=1$, which implies  $\Prob{\tilde\Omega_k|\filt_{\eventtime_{k-1}} }=1$, and hence,  $\Prob{\tilde\Omega_k}=\Ept{\Prob{\tilde\Omega_k|\filt_{\eventtime_{k-1}}}}=1.$

Next, we establish \eqref{cond_density_arrival} and \eqref{cond_density_departure} and the conditional independence of $\xi_{k+1}$ and $\sigma_{k+1}^i$ given $\filt_{\eventtime_{k}}$. Since $\Probil{\tilde\Omega_k}=1$, by \eqref{next_event} we have, almost surely,
\begin{equation}\label{tau_k}
    \eventtime_k= \min(\xi_{k},\sigma_{k}^i;i=1,...,N).
\end{equation}
Given the  filtration $\{\filt_t\}$ satisfies \eqref{def_filt_alt},  by \cite[Theorem T30 of Section A.2]{BremaudBook},  $\filt_{\eventtime_k}=\sigma(I_0)\vee \sigma(\eventtime_{k^\prime},\eventmark_{k^\prime};1\leq k^\prime\leq k), $ and therefore $\filt_{\eventtime_k}= \filt_{\eventtime_{k-1}}\vee \sigma(\eventtime_k,\eventmark_k).$ We partition $\Omega$ based on the values of $\eventmark_k$ and the queue lengths right before the event time $\eventtime_k$, and show the conditional independence on each (measurable) subset of the partition. First, for every $i_*\in\{1,...,N\}$, let $A_{i_*}\doteq\{\eventmark_k=(\mathfrak{E},i_*)\}$ be the set of realizations on which the $\Xth{k}$ event (which was already shown to be almost surely distinctly defined) is an arrival to the queue $i_*$. On $A_{i*}$, $\eventtime_k=\xi_k$ and $\xi_{k+1}=\eventtime_k+u_{\arrival(\eventtime_k)+1}$, and hence using Lemma \ref{lem_independence}, for every $b_0,b_1,...,b_N\geq0$ we have,
\begin{align}\label{cond_Eind1}
  &\indic{A_{i*}}\Prob{\xi_{k+1}>\eventtime_k+b_0,\sigma^i_{k+1}>\eventtime_k+b_i,i=1,...,N|\filt_{\eventtime_k}}\notag\\
  &\hspace{1cm}= \indic{A_{i*}}\Prob{u_{\arrival(\eventtime_k)+1}>b_0,\sigma^i_{k+1}>\eventtime_k+b_i,i=1,...,N|\filt_{\eventtime_k}}\notag\\
  &\hspace{1cm}= \indic{A_{i*}}\bGe(b_0) \Prob{\sigma^i_{k+1}>\eventtime_k+b_i,i=1,...,N|\filt_{\eventtime_k}}.
\end{align}
In particular, substituting $b_0=b$ and $b_i=0,$ $i=1,...,N,$ in \eqref{cond_Eind1} and recalling that $\re(\eventtime_k)=0$ on $A_{i_*}$, we have
\begin{align}\label{cond_UE}
    \indic{A_{i_*}}\Prob{\xi_{k+1}>\eventtime_k+b\big|\filt_{\eventtime_{k}}} =\indic{A_{i_*}}\bGe(b)
    = \indic{A_{i_*}}\frac{\bGe(\re(\eventtime_k)+b)}{\bGe(\re(\eventtime_k))}.
\end{align}
We now turn to  the joint conditional distribution of $\sigma_k^i,i=1,...,N,$ given $\filt_{\eventtime_k}$. Consider $Y=\xi_k$,  $W_i=\sigma_k^i$, $i=1,...,N$, $T=\eventtime_k$, $Z=\eventmark_{k,1}$ and $z_0=(\mathfrak{E})$. By the induction hypothesis, $\xi_k$ and $\sigma_k^i,i=1,...,N,$ are conditionally independent given $\filt_{\eventtime_{k-1}}$, and  $Y, W_i, i = 1, \ldots, N, T, Z$ satisfy the setup of Lemma \ref{lem_condauxind}. Using \eqref{tau_k} and noting that  $\{\xi_k\leq\min(\sigma_k^i;i=1,...,N)\} =\{\eventmark_k=(\mathfrak{E},i)\text{ for some }i=1,...,N\}=\{\eventmark_{k,1}=\mathfrak{E}\},$ we have $T=\eventtime_k$, $Z=\eventmark_{k,1}$ and $z_0=\mathfrak{E}$. Therefore, and since $A_{i_*}\subset \{\eventmark_{k,1}=\mathfrak{E}\}$ and using \eqref{rqc_ind} of Lemma \ref{lem_renewal} in the second and the last equality and \eqref{condAuxInd} of Lemma \ref{lem_condauxind} in the third equality, we have
\begin{align}\label{temp_apx1}
  \indic{A_{i_*}}\Prob{\sigma^i_{k}>\eventtime_k+b_i,i=1,...,N|\filt_{\eventtime_k}}\notag
  & =  \indic{A_{i_*}} \indic{\eventmark_{k,1}=\mathfrak{E}}
  \Prob{\sigma^i_{k}>\eventtime_k+b_i,i=1,...,N|\filt_{\eventtime_{k-1}},\eventtime_k,\eventmark_k}\notag\\
  & =  \indic{A_{i_*}}\prod_{i=1}^N \Prob{\sigma^i_{k}>\eventtime_k+b_i|\filt_{\eventtime_{k-1}},\eventtime_k,\eventmark_{k,1}}\notag\\
  & =  \indic{A_{i_*}}\prod_{i=1}^N \Prob{\sigma^i_{k}>\eventtime_k+b_i|\filt_{\eventtime_{k}}}.
\end{align}
Moreover, if $i\in\busy_{k-1}$, the conditional distribution of $\sigma_i^k$ is given by \eqref{cond_CDF_departure}, and hence using \eqref{condAuxDist} of Lemma \ref{lem_condauxind} and the fact that on $A_{i_*},$ $\age^i(\cdot)$ grows linearly on $(\eventtime_{k-1},\eventtime_k],$ we have
\begin{align}\label{temp_apx2}
    \indic{A_{i_*}}\indic{i\in\busy_{k-1}}\Prob{\sigma^i_{k}>\eventtime_k+b
    \big|\filt_{\eventtime_{k}}}  & = \indic{A_{i_*}}\indic{i\in\busy_{k-1}}\Prob{\sigma^i_{k}-\eventtime_{k-1}>\eventtime_k-\eventtime_{k-1}+b
    \big|\filt_{\eventtime_{k-1}},\eventtime_k,\eventmark_k} \notag\\
    &=\indic{A_{i_*}}\indic{i\in\busy_{k-1}} \frac{\overline G(\age^i(\eventtime_{k-1})+\eventtime_k-\eventtime_{k-1}+b )}{\overline G(\age^i(\eventtime_{k-1})+\eventtime_k-\eventtime_{k-1})}\notag\\
    &=\indic{A_{i_*}}\indic{i\in\busy_{k-1}}\frac{\overline G(\age^i(\eventtime_{k})+b )}{\overline G(\age^i(\eventtime_{k}))}.
\end{align}

Now we further partition $A_{i_*}$ into two subsets,  $A_{i_*,1}\doteq\{\eventmark_k=(\mathfrak{E},i_*),X^{i_*}(\eventtime_{k-1})\geq1\}$  and $A_{i_*,2}\doteq\{\eventmark_k=(\mathfrak{E},i_*),X^{i_*}(\eventtime_{k-1})=0\}$, on which the queue $i_*$ is non-empty and empty, respectively, right before the arrival.  Since an arrival to a non-empty queue does not change the set of busy servers, $\busy_k=\busy_{k-1}$ and  $\sigma^i_{k+1}=\sigma^i_k$ for all $i=1,...,N$, on $A_{i_*,1}$. Substituting these identities in \eqref{temp_apx1} and \eqref{temp_apx2}, respectively, we have
\begin{equation}\label{cond_Eind2}
  \indic{A_{i_*,1}}\Prob{\sigma^i_{k+1}>\eventtime_k+b_i,i=1,...,N|\filt_{\eventtime_k}} =  \indic{A_{i_*,1}}\prod_{i=1}^N \Prob{\sigma^i_{k+1}>\eventtime_k+b_i|\filt_{\eventtime_{k}}},
\end{equation}
and for every $i=1,...,N,$
\begin{equation}\label{Cond_VE1}
    \indic{A_{i_*,1}}\indic{i\in\busy_{k}}\Prob{\sigma^i_{k+1}>\eventtime_k+b
    \big|\filt_{\eventtime_{k}}}  =\indic{A_{i_*,1}}\indic{i\in\busy_{k}}\frac{\overline G(\age^i(\eventtime_{k})+b )}{\overline G(\age^i(\eventtime_{k}))}.
\end{equation}
On the other hand, on $A_{i_*, 2}$, the arrival makes the previously idle server $i_*$ busy, that is, $\busy_{k}=\busy_{k-1}\cup\{i_*\}$, and also $\sigma_{k+1}^{i}=\sigma_{k}^{i}$ for all $i\neq i_*$  and $ \sigma_{k+1}^{i_*}=\eventtime_{k}+v_{J^{i_*}(\eventtime_k)}$. Therefore, using \eqref{indep_VE} of Lemma \ref{lem_independence} in the second equality and \eqref{temp_apx1} in the third equality, we have
\begin{align}\label{cond_Eind3}
  \indic{A_{i_*,2}}\Prob{\sigma^i_{k+1}>\eventtime_k+b_i;i=1,...,N|\filt_{\eventtime_k}}\notag
  & =\indic{A_{i_*,2}}\Prob{v_{J^{i_*}(\eventtime_k)}> b_{i_*},\sigma^i_{k}>\eventtime_k+b_i;i\neq i_*|\filt_{\eventtime_k}}\notag\\
 & =\indic{A_{i_*,2}}\overline G(b_{i_*})\Prob{\sigma^i_{k}>\eventtime_k+b_i;i\neq i_*|\filt_{\eventtime_k}}\notag\\
 & =\indic{A_{i_*,2}}\overline G(b_{i_*})\prod_{i=1,i\neq i_*}^N\Prob{\sigma^i_{k}>\eventtime_k+b_i|\filt_{\eventtime_k}}.
\end{align}
In particular, by substituting $b_{i_*}=b$ and $b_i=0$ for $i=1,...,N,i\neq i_*$ in \eqref{cond_Eind3} and since $i_*\in\busy_k$ on $A_{i_*,2}$ and  $\age^{i_*}(\tau_k)$, the age  of the job $J^{i_*}(\eventtime_k)$ that began receiving service in queue $i_*$ at time $\tau_k$, is zero, we have
\begin{align}\label{Cond_VE2}
  \indic{A_{i_*,2}}\indic{i_*\in\busy_{k}}\Prob{\sigma_{k+1}^{i*}>\eventtime_k+ b|\filt_{\eventtime_k}}& =\indic{A_{i_*,2}}\indic{i_*\in\busy_{k}}\overline G(b) \notag \\
  & =  \indic{A_{i_*,2}}\indic{i_*\in\busy_{k}}\frac{\overline G(\age^{i_*}(\eventtime_{k})+b )}{\overline G(\age^{i_*}(\eventtime_{k}))}.
\end{align}
Finally, on $A_{i_*,2}$, for every $i\neq i_*$, $\indicil{i\in\busy_{k}}=\indicil{i\in\busy_{k-1}}$ and $\sigma_{k+1}^{i}=\sigma_{k}^{i}$, and so
 \eqref{temp_apx2} implies
\begin{align}\label{cond_VE3}
    \indic{A_{i_*,2}}\indic{i\in\busy_{k}}\Prob{\sigma^i_{k+1}>\eventtime_k+b
    \big|\filt_{\eventtime_{k}}} =\indic{A_{i_*,2}}\indic{i\in\busy_{k}}\frac{\overline G(\age^i(\eventtime_{k})+b )}{\overline G(\age^i(\eventtime_{k}))}.
\end{align}

Now consider the set $B_{i_*}=\{z_k=(\mathfrak{D},i_*)\}$ of realizations on which the $\Xth{k}$ event is a departure from queue $i_*$. Consider $Y=\sigma^{i^*}_k$, $W_{i_*}=\xi_k$ and $W_i=\sigma_k^i$ for $i=1,...,N,i\neq i_*$, $T=\eventtime_k$, $Z=\eventmark_k$ and $z_0=(\mathfrak{D},i_*)$.  By the induction hypothesis, $\xi_k$ and $\sigma_k^i,i=1,...,N,$ are conditionally independent given $\filt_{\eventtime_{k-1}}$, and $Y, W_i, i = 1, \ldots, N, T, Z$ satisfy the setup of Lemma \ref{lem_condauxind}. Therefore, using \eqref{condAuxInd} of Lemma \ref{lem_condauxind} in the second equality below, we have
\begin{align}\label{temp_apx31}
  &\indic{B_{i_*}}\Prob{\xi_{k}\geq \eventtime_k+b_0,\sigma^i_{k}>\eventtime_k+b_i,i=1,...,N,i\neq i_*|\filt_{\eventtime_k}}\notag\\
  &\hspace{1cm}= \indic{B_{i_*}} \indic{\eventtime_k=(\mathfrak{D},i)}
  \Prob{\xi_k\geq \eventtime_k+b_0,\sigma^i_k>\eventtime_k+b_i,i=1,...,N,i\neq i_*|\filt_{\eventtime_{k-1}},\eventtime_k,\eventmark_k}\notag\\
  &\hspace{1cm}= \indic{B_{i_*}}\Prob{\xi_k>\eventtime_k+b_i|\filt_{\eventtime_k}}\prod_{i=1,i\neq i_*}^N \Prob{\sigma^i_{k}>\eventtime_k+b_i|\filt_{\eventtime_{k}}}.
\end{align}
Moreover, given $\filt_{\eventtime_{k-1}}$, the conditional distribution of $\xi_k$  and $\sigma_k^i$ for $i\in\busy_{k-1}\backslash\{i_*\}$ are given by \eqref{cond_CDF_arrival} and \eqref{cond_CDF_departure}, respectively. Therefore, using \eqref{condAuxDist} of Lemma \ref{lem_condauxind}, we have
\begin{align}
    \indic{B_{i_*}}\Prob{\xi_{k}>\eventtime_k+b\big|\filt_{\eventtime_{k}}}
     & =   \indic{B_{i_*}}\Prob{\xi_{k}-\eventtime_{k-1}>\eventtime_k-\eventtime_{k-1}+b\big|\filt_{\eventtime_{k}}}\notag\\
     & =   \indic{B_{i_*}}\frac{\bGe(\re(\eventtime_{k-1})+\eventtime_k-\eventtime_{k-1}+b )}{\bGe (\re(\eventtime_{k-1})+\eventtime_k-\eventtime_{k-1})} \notag  \\
& =  \indic{B_{i_*}}\frac{\bGe(\re(\eventtime_{k})+b )}{\bGe (\re(\eventtime_{k}))},
\label{cond_UD}
\end{align}
where the last equality uses the fact that on $B_{i_*}$, $\re(\cdot)$ grows linearly on  $(\eventtime_{k-1},\eventtime_k]$ (because  $\xi_{k+1}=\xi_k$) and for every $i\neq i_*$,
\begin{align}
    \indic{B_{i_*}}\indic{i\in\busy_{k-1}}\Prob{\sigma^i_{k}>\eventtime_k+b \big|\filt_{\eventtime_{k}}}
        & =\indic{B_{i_*}}\indic{i\in\busy_{k-1}} \frac{\overline G(\age^i(\eventtime_{k-1})+\eventtime_k-\eventtime_{k-1}+b )}{\overline G(\age^i(\eventtime_{k-1})+\eventtime_k-\eventtime_{k-1})} \notag \\
& = \indic{B_{i_*}}\indic{i\in\busy_{k}} \frac{\overline G(\age^i(\eventtime_{k})+b )}{\overline G(\age^i(\eventtime_{k}))},
\label{cond_VD1}
\end{align}
where the last equality uses the fact that  $i\in\busy_{k}$ and  $\sigma_{k+1}^i=\sigma_k^i$ if and only if $i\in\busy_{k-1}$ and $a^i(\cdot)$  grows linearly on $(\eventtime_{k-1},\eventtime_k]$.

The value of $\sigma_{k+1}^{i_*}$ depends on the length of queue ${i_*}$ right before the departure, and hence we further partition $B_{i_*}$ into two parts. On $B_{i_*,1}=\{z_k=(\mathfrak{D},i_*), X^{i_*}(\eventtime_{k-1})=1\}$, queue $i_*$ becomes empty after the departure at $\eventtime_k$, and hence $\sigma_{k+1}^{i_*}=\infty$ by definition. Therefore, using \eqref{temp_apx31}, we have
\begin{align}\label{cond_Dind1}
    &\indic{B_{i_*,1}}\Prob{\xi_{k+1}\geq \eventtime_k+b_0,\sigma^i_{k+1}>\eventtime_k+b_i;i=1,...,N|\filt_{\eventtime_k}}\notag\\
    &\hspace{1cm}=\indic{B_{i_*,1}}\Prob{\xi_{k}>\eventtime_k+ b_0,\sigma^i_{k}>\eventtime_k+b_i;i=1,...,N,i\neq i_*|\filt_{\eventtime_k}}\notag\\
    &\hspace{1cm}= \indic{B_{i_*,1}}\Prob{\xi_k>\eventtime_k+b_0|\filt_{\eventtime_k}}\Prob{\sigma_{k+1}^{i_*}>b_{i_*}}\prod_{i=1,i\neq i_*}^N \Prob{\sigma^i_{k}>\eventtime_k+b_i|\filt_{\eventtime_{k}}},
\end{align}
where we use the  trivial identity $\Probil{\sigma_{k+1}^{i_*}>b_{i_*}}$$=1$ in the last equality.
On the other hand, on $B_{i_*,2}=\{z_k=(\mathfrak{D},i_*), X^{i_*}(\eventtime_{k-1})\geq2\}$, a new job, namely $J^{i_*}(\eventtime_k)$,  enters service $i_*$ right after departure at $\eventtime_k$, and hence $\sigma_{k+1}^i = \sigma_k^i$ for $i \neq i_*$ and $\sigma_{k+1}^{i_*}=\eventtime_k+v_{J^{i_*}(\eventtime_k)}$ and $a^{i_*}(\eventtime_k)=0$. Therefore, using \eqref{indep_VJ} of Lemma \ref{lem_independence} in the second equality and  \eqref{temp_apx31} in the third equality below, we have
\begin{align}\label{cond_Dind2}
&\indic{B_{i_*,2}}\Prob{\xi_{k+1}\geq \eventtime_k+b_0,\sigma^i_{k+1}>\eventtime_k+b_i,i=1,...,N|\filt_{\eventtime_k}}\notag\\
&\hspace{1cm}=\indic{B_{i_*,2}}\Prob{v_{J^{i_*}(\eventtime_k)}>b_{i_*},\xi_{k}\geq \eventtime_k+b_0,\sigma^i_{k}>\eventtime_k+b_i,i=1,...,N,i\neq i_*|\filt_{\eventtime_k}}\notag\\
&\hspace{1cm}=\indic{B_{i_*,2}}\overline G(b_{i_*})\Prob{\xi_{k}\geq \eventtime_k+ b_0,\sigma^i_{k}>\eventtime_k+b_i,i=1,...,N,i\neq i_*|\filt_{\eventtime_k}}\notag\\
&\hspace{1cm}=\indic{B_{i_*,2}}\overline G(b_{i_*})\Prob{\xi_k>\eventtime_k+b_0|\filt_{\eventtime_k}}\prod_{i=1,i\neq i_*}^N \Prob{\sigma^i_{k}>\eventtime_k+b_i|\filt_{\eventtime_{k}}}.
\end{align}
In particular, substituting $b_{i_*}=b$ and $b_i=0$ for $i=0,1,...,N,i\neq i_*$ in \eqref{cond_Dind2} and using the facts that  $i_*\in\busy_k$ and $a^{i_*}(\eventtime_k)=0$ on $B_{i_*,2}$, we have
\begin{align}\label{cond_VD2}
  \indic{B_{i_*,2}}\indic{i_*\in\busy_k}\Prob{\sigma^{i_*}_{k+1}>\eventtime_k+b|\filt_{\eventtime_k}} & =\indic{B_{i_*,2}}\indic{i_*\in\busy_k}\overline G(b) \notag\\
  & = \indic{B_{i_*,2}} \indic{i_*\in\busy_k} \frac{\overline G(a^{i_*}(\eventtime_k)+b)}{\overline G(a^{i_*}(\eventtime_k))}.
\end{align}

To conclude the lemma, observe that
\begin{equation}\label{cond_sum}
  \Omega= \bigcup_{i_*=1}^N\left(A_{i_*}\cup B_{i_*}\right),\quad\quad A_{i_*}=A_{i_*,1}\cup A_{i_*,2}, \quad\quad B_{i_*}=B_{i_*,1}\cup B_{i_*,2},\quad\quad i_*=1,...,N.
\end{equation}
The conditional independence of $\xi_{k+1}$ and $\sigma_{k+1}^i,i=1,...,N$ follows from \eqref{cond_sum} and the conditional independence results on each set of the partition, namely, \eqref{cond_Eind1}, \eqref{cond_Eind2}, \eqref{cond_Eind3}, \eqref{cond_Dind1} and \eqref{cond_Dind2}. Moreover, the form \eqref{cond_density_arrival} for the conditional distribution of $\xi_{k+1}$ given $\filt_{\eventtime_{k}}$ follows from \eqref{cond_sum}, \eqref{cond_UE} and \eqref{cond_UD}. Finally, for every $i\in\busy_{k}$, the form \eqref{cond_density_departure} of the conditional distribution of $\sigma_{k+1}^i$ given $\filt_{\eventtime_{k}}$ follows from \eqref{cond_sum}, \eqref{Cond_VE1}, \eqref{Cond_VE2}, \eqref{cond_VE3}, \eqref{cond_VD1}  and \eqref{cond_VD2}.
\end{proof}

\section{Compensator of the weighted departure process}\label{sec_Qcomp}

In this section, we prove Proposition \ref{prop_prelim_Qcomp}. We first identify the intensity of the process $\dni$, defined in \eqref{def_Di}.

\begin{lemma}\label{prelim_Dcomp_lem}
Suppose Assumptions \ref{asm_E},  \ref{asm_G}.\ref{asm_mean} and \ref{asm_initial}.\ref{asm_initial_ind} hold. Then, for $i=1,...,N$,
an $\{\filtn_t\}$-intensity of $\dni$ is
\begin{equation}\label{def_Dintensity}
  \left\{\indic{\xni(t-)\geq 1}h\left(\ageni(t-)\right);t\geq0\right\}.
\end{equation}
\end{lemma}
\begin{remark}
Recall that although $\agei(t)$ is not defined if server $i$ is idle at time $t$, the quantity $\indicil{\xii(t)\geq 1}h(\agei(t))$ is always well defined.
\end{remark}
\begin{proof}[Proof of Lemma \ref{prelim_Dcomp_lem}.]
Fix $N$  and suppress the superscript $(N)$ for ease of notation.
Suppose that for every $k\geq0$, the following conditional density$f^{\mathfrak{D},i}_{k+1}$ exists:
\begin{equation}\label{prelim_gdk_eq}
\Prob{\eventtime_{k+1}-\eventtime_{k}\in A,\eventmark_{k+1}=(\mathfrak{D},i)\big|\filt_{\eventtime_k}}= \int_A f^{\mathfrak{D},i}_{k+1}(\omega,r) dr,\quad\quad \omega\in\Omega,\; A\in\mathcal{B}\hc.
\end{equation}
Then, by the representation \eqref{Di_mpp_rep} of $\di$ in terms of the marked point process $\MPn$ and  the non-explosiveness of the sequence $\{\eventtime_k\}$  proved in Corollary \ref{cor_distinctness}, it follows from \cite[Theorem III.T7, comment ($\beta$) below the theorem, and Eqn. (2.10) of the next remark]{BremaudBook} that the process
\begin{equation}\label{Di_int}
  \sum_{k=0}^\infty \frac{f^{\mathfrak{D},i}_{k+1}(t-\eventtime_k)}{\Prob{\eventtime_{k+1}>t|\filt_{\eventtime_k}}}\indic{\eventtime_k<t\leq \eventtime_{k+1}}.
\end{equation}
is an $\{\filt_t\}$-intensity of $\di$.

We now identify the process $f^{\mathfrak{D},i}_{k+1}$ that satisfies \eqref{prelim_gdk_eq}. Recall that $\xi_{k+1}$ and $\sigma_{k+1}^i$ are, respectively, the first arrival time and the first departure time from queue $i$ after the event time $\eventtime_k.$ Using \eqref{next_event_new}, the next event after $\eventtime_k$ is a departure from queue $i$, that is,  $\eventmark_{k+1}=(\mathfrak{D},i)$,  if the next departure from queue $i$ occurs before the next arrival and the next  departure  from every other queue, that is, $\sigma^i_{k+1}\leq \min( \xi_{k+1}, \sigma_{k+1}^{i'},i'=1,...,N,i'\neq i)$. If  server $i$ is idle at time $\eventtime_k$, that is,  $i\not\in\busy(\eventtime_k)$, $\sigma^i_{k+1}=\infty$ by definition, and the probability that the next event is a departure from queue $i$ is zero, and therefore,
\begin{equation}\label{gD_idle}
  f^{\mathfrak{D},i}_{k+1}\equiv0,\quad\quad \forall i\not\in \busy(\eventtime_k).
\end{equation}
On the other hand, for $i\in\busy(\eventtime_k)$, using the conditional independence, given $\filt_{\eventtime_k}$, of $\xi_{k+1}$ and $\sigma_{k+1}^i,i=1,...,N$,  established  in Lemma \ref{lem_condInd} and \eqref{cond_density_departure}, we have
  \begin{align*}
    &\indic{i\in\busy(\eventtime_k)}\Prob{\eventtime_{k+1}-\eventtime_{k}>t, \eventmark_{k+1}=(\mathfrak{D},i)|\filt_{\eventtime_k}}\notag\\
    &\hspace{2cm}=\indic{i\in\busy(\eventtime_k)}\Prob{\sigma^i_{k+1} -\eventtime_{k}>t, \xi_{k+1}\wedge  \min_{i'\neq i} \sigma^{i^\prime}_{k+1}>\sigma^i_{k+1} |\filt_{\eventtime_k}}\notag\\
  &\hspace{2cm}=\indic{i\in\busy(\eventtime_k)}\frac{1}{\overline G(\agei(\eventtime_k))}\int_t^\infty\Prob{\xi_{k+1}\wedge  \min_{i'\neq i} \sigma^{i^\prime}_{k+1}>\eventtime_k+s|\filt_{\eventtime_k}}g(s+ \agei(\eventtime_k))ds.
\end{align*}
Thus, the conditional density $f^{\mathfrak{D},i}_{k+1}$ exists and for every $t\geq \eventtime_k$,
\begin{equation}\label{gD_busy}
f^{\mathfrak{D},i}_{k+1}(t-\eventtime_k)=\indic{i\in\busy(\eventtime_k)}\frac{1}{\overline G(\agei(\eventtime_k))}\Prob{\xi_{k+1}\wedge  \min_{i'\neq i} \sigma^{i^\prime}_{k+1}>t|\filt_{\eventtime_k}}g(t-\eventtime_k+\agei(\eventtime_k)).
\end{equation}

Moreover, by \eqref{next_event_new}, and  Lemma \ref{lem_condInd}, for $i\in\busy_k$ we have
\begin{align}\label{prelim_GbarD}
  \Prob{\eventtime_{k+1}>t|\filt_{\eventtime_k}}& = \Prob{\xi_{k+1}\wedge \sigma^i_{k+1}\wedge \min_{i'\neq i}\sigma^{i'}_{k+1}>t |\filt_{\eventtime_k}} \notag\\
  &=   \Prob{\xi_{k+1}\wedge \min_{i'\neq i}\sigma^{i'}_{k+1}>t|\filt_{\eventtime_k}}\frac{\overline G(t-\eventtime_k+\agei(\eventtime_k)).}{\overline G(\agei(\eventtime_k))}.
\end{align}
Combining \eqref{gD_idle}--\eqref{prelim_GbarD} with  definition \eqref{def_busy}, the  $\{\filt_t\}$-intensity of $\di$ given in \eqref{Di_int} takes the form
\begin{align*}
\sum_{k=0}^\infty \indic{i\in\busy_k}h(t-\eventtime_k+\agei(\eventtime_k))\indic{\eventtime_k<t\leq \eventtime_{k+1}}
&= \sum_{k=0}^\infty \indic{\xii(\eventtime_k)\ge1}h(t-\eventtime_k+\agei(\eventtime_k))\indic{\eventtime_k<t\leq \eventtime_{k+1}}\\
&= \sum_{k=0}^\infty \indic{\xii(t-)\geq1}h\left(\age^i(t-)\right)\indic{\eventtime_k<t\leq \eventtime_{k+1}} \\
  &=  \indic{\xii(t-)\geq1}h\left(\age^i(t-)\right),
\end{align*}
where the second equality uses the fact that  for $t\in(\eventtime_k,\eventtime_{k+1}]$ and $i\in\busy_k$, $\xii(t-)=\xii(\eventtime_k)$ and  $\agei(t-)= t-\eventtime_k +\agei(\eventtime_k)$.  This  completes the proof.
\end{proof}

\begin{proof}[Proof of Proposition \ref{prop_prelim_Qcomp}]
Fix $N\in\N$, $\varphi\in\tightphiset$ and $\ell\geq1$, and recall that $\agen_j(\cdot)$, $v_j$ and  $\deptn_j$ are the age process, service time and  the departure time of  job $j$, and that $\stationn_j$ is the queue to which job $j$ is routed and $\agenX{k}(t)$ is the age of the job in service at queue $k$, if one exists. Thus, we have
\[\agenX{\stationn_j}(\deptn_j-)=\agen_j(\deptn_j)=v_j.\]
Using this and \eqref{def_CustXS}, the process $ \Quephinl(t)$ defined in \eqref{def_Qphi} can be rewritten as \begin{align}\label{Q_decompose}
\Quephinl(t)& =  \sum_{j=\j0}^{\infty}  \varphi\Big(\agenX{\stationn_j}(\deptn_j-),\deptn_{j}\Big) \indic{\deptn_j\leq t}\indic{\xnX{\stationn_j}(\deptn_j-)\geq\ell} \notag\\
& =\sum_{i=1}^N\Quephinli(t),
\end{align}
where
\[
   \Quephinli(t)\doteq \sum_{j=\j0}^{\infty} \varphi\Big(\ageni(\deptn_j-),\deptn_{j}\Big) \indic{\deptn_j\leq t}\indic{\xni(\deptn_j-)\geq\ell}\indic{\stationn_j=i}.
\]
By definition \eqref{def_Di} of $\dni$, the process $\Quephinli$ can be represented as the following integral with respect to the departure process $\dni$:
\begin{equation}\label{prelim_Qi}
\Quephinli(t)= \int_0^t \varphi\left(\ageni(s-),s\right) \indic{\xni(s-)\geq \ell} d\dni(s),\quad\quad t\geq0.
\end{equation}
Consider the setup of Lemma \ref{remark_integral} with $\xi, \{\cal G_t\}, \theta$ and $\zeta$ replaced by $\dni$,  $\{\filtn_t\}$, $\varphi(\ageni(s-),s)\indicil{\xni(s-)}$, and $\Quephinli$, respectively.
Since $\ageni$ and $\xni$ are right-continuous and $\{\filtn_t\}$-adapted (see Proposition \ref{prop_adapted}), $\theta$ is bounded, $\{\filtn_t\}$-adapted, and left-continuous. Hence,  by Lemmas \ref{prelim_Dcomp_lem} and \ref{remark_integral}, the process $\{\Mni_{\varphi,\ell}(t);t\geq0\}$ defined by
\begin{align}\label{temp_Dint2}
  \Mni_{\varphi,\ell}(t)& \doteq\Quephinli(t) - \int_0^t\varphi\left(\ageni(s-),s\right)  h(\ageni(s-))\indic{\xni(s-)\geq \ell}ds,\notag\\
                        & =\Quephinli(t) - \int_0^t\varphi\left(\ageni(s),s\right)  h(\ageni(s))\indic{\xni(s)\geq \ell}ds,
\end{align}
is a local $\{\filtn_t\}$-martingale. Moreover,  by \eqref{nun_nuni_alt} and \eqref{Def_Qcomp} we have
\begin{equation}\label{temp_Dint}
  \sum_{i=1}^N \int_0^t\varphi\left(\ageni(s),s\right)h(\ageni(s)) \indic{\xni(s)\geq \ell} ds  = \int_0^t\langle \varphi(\cdot,s)h(\cdot),\nun_\ell(s)\rangle ds = \An_{\varphi,\ell}(t).
\end{equation}
Summing up both sides of \eqref{temp_Dint2} over $i=1,...,N,$ and using \eqref{Q_decompose} and \eqref{temp_Dint}, we have $\sum_{i=1}^N \Mni_{\varphi,\ell}= \Quen_{\varphi,\ell} -\An_{\varphi,\ell} = \Mn_{\varphi,\ell},$ and hence, $\Mn_{\varphi,\ell}$ is also a local $\{\filtn_t\}$-martingale. Moreover, by \eqref{remark_Quad} of Lemma \ref{remark_integral} with the setup described above, for all $t\geq0,$
\begin{align*}
[\Quephinli](t)= \int_0^t \varphi^2\left(\ageni(u-),u\right) \indic{\xni(u-)\geq \ell} d\dni(u)= \Queni_{\varphi^2,\ell} (t).
 \end{align*}
Furthermore, for every $i=1,...,N$ and  $i'\neq i$, $\QuephinlX{i}$ and $\QuephinlX{i'}$ are pure jump processes with no common jump times almost surely (see Lemma \ref{lem_condInd}),  and hence $[ \QuephinlX{i},\QuephinlX{i'}]\equiv0, $ almost surely. Together with the fact that $\An_{\varphi,\ell}=\Quen_{\varphi,\ell}-\Mn_{\varphi,\ell}$ is a continuous function with finite variation, this implies
\[ [\Mn_{\varphi,\ell}] = [\Quen_{\varphi,\ell}] = [\sum_{i=1}^N\Queni_{\varphi,\ell}] = \sum_{i=1}^N[\Queni_{\varphi,\ell}] =\sum_{i=1}^N\Queni_{\varphi^2,\ell}=\Quen_{\varphi^2,\ell}, \]
which completes the proof.
\end{proof}

\section{A Bound for Renewal Processes}\label{apx_renewal}

Fix $\rstar_0\in\R_+$ and let $\Pstar (t)$ be a delayed renewal process with inter-arrival times $\{\ustar_n;n\geq1\}$ with common distribution $\Gstar$ and delay $\ustar_0$ with distribution $\Gstar_{r_0^*}$:

\begin{equation}\label{rhn_u0}
  \Prob{\ustar_0\leq x}= \Gstar_{\rstar_0}(x)\doteq \frac{\Gstar (x+\rstar_0)- \Gstar (\rstar_0)}{1-\Gstar (\rstar_0)}.
\end{equation}
Assume $\Gstar$ has a density, denote $\overline \Gstar \doteq 1-\Gstar$ and let $\hstar$ be the corresponding rate function.  Also, let $\rstar (t)$ denote the backward recurrence time of the renewal process $\Pstar$. By convention, $\rstar(t)=\rstar_0+t$ for $t<\ustar_0$, and in particular, $\rstar(0)=\rstar_0$.

\begin{lemma}\label{lem_renewal}
Given the quantities described above, for every $t\geq0,$
\begin{equation}\label{renewal_Ept}
  \Ept{\int_0^t \hstar(\rstar (s))ds}<\infty,
\end{equation}
and
\begin{equation}\label{renewal_Bound}
  \Ept{\left(\int_0^t \hstar (\rstar(s))ds - \Pstar (t)\right)^2}\leq 12+ 3 \Ept{\Pstar (t)}.
\end{equation}
\end{lemma}
\begin{proof}
Define the epoch times $\{t_j;j\geq0\}$ as $t_0=\ustar_0$ and $t_{j}=t_{j-1}+\ustar_j$ for $j\geq1$. Then, we have
  \begin{align}\label{hrn_temp0}
    \int_0^t \hstar (\rstar (s))ds =&    \int_0^{t_0}\hstar (\rstar (s))ds  +\sum_{n=1}^{\Pstar (t)} \int_{t_{n-1}}^{t_{n}} \hstar (\rstar (s))ds
    - \int_t^{t_{\Pstar (t)}}\hstar (\rstar (s))ds \notag\\
    =& \int_{\rstar_0}^{\rstar_0+\ustar_0}\hstar (v)dv +\sum_{n=1}^{\Pstar (t)} \int_{0}^{\ustar_n}\hstar (v)dv-\int_{\rstar(t)}^{\ustar_{\Pstar (t)}}\hstar (v)dv.
  \end{align}
Defining the random variables $y_n,n\in\N,s$ as $y_n \doteq \int_{0}^{\ustar_n} \hstar (v)dv,$ the second term on the right-hand side of \eqref{hrn_temp0} can be written as $\sum_{n=1}^{\Pstar (t)} y_n.$ Since the renewal times $\{\ustar_n;n\geq1\}$ are i.i.d.\, the sequence $\{y_n\}_{n\in\N}$ is also i.i.d.\  with
\begin{align*}
 \Ept{y_1}=&\int_0^\infty\left( \int_0^{s} \hstar (v)dv\right)\gstar(s)ds
                 = \int_0^\infty\frac{\gstar (s)}{\overline \Gstar (s)}\left(\int_s^\infty \gstar (v)dv\right)ds=1,
\end{align*}
and $\mathbb{E}[(y_1)^2]$ is equal to
\begin{align*}
\int_0^\infty\left( \int_0^s \hstar (v)dv\right)^2 \gstar (s)ds
                 =\int_0^\infty \left(\log(\overline \Gstar (s))\right)^2\gstar (s)ds
                =\int_0^1\left(\log(s)\right)^2ds
                =2.
\end{align*}
Thus, the mean and variance of $y_n$ are both equal to $1$. Now, define the discrete-time filtration $\{\cal G_n;n\geq 0\}$ by $\cal G_n = \sigma(\ustar_j; j=0,...,n)$. Note that $\ustar_n$, and hence $y_n$, are  $\cal G_n$-measurable. Also, since the inter-arrival times are independent, $y_{n+1},y_{n+2},...$ are independent of $\cal G_n$. Finally, the random variable $\Pstar (t)$ satisfies $\Ept{\Pstar (t)}\leq \Ustar (t)<\infty$.  where $\Ustar$ is the renewal measure corresponding to the distribution $\Gstar$. (The inequality can be replaced by an equality if $\Pstar$ is replaced with a pure renewal process $\tilde P^*$, see \cite[Theorem 2.4.(iii) of Section V]{Asm03}, and $P^*$ and $\tilde P^*$ can be coupled such that almost surely, $\Pstar(t)\leq \tilde P^*(t)$ for all $t\geq0$.) Moreover, $\Pstar (t)$ is an integrable $\{\cal G_n\}$-stopping time because $\{\Pstar (t)=n\}=\{t_{n-1}\leq t <t_n\}\in\cal G_n$
since both $t_{n-1}$ and $t_n$ are $\cal G_n$-measurable.
Hence, by Wald's lemma \cite[Proposition A 10.2]{Asm03},
\begin{equation}\label{hrn_second_E}
\Ept{\sum_{n=1}^{\Pstar (t)} \int_{0}^{\ustar_n}\hstar (v)dv}=\Ept{\sum_{n=1}^{\Pstar (t)}y_n}= \Ept{\Pstar (t)}\Ept{y_1} =\Ept{\Pstar (t)} <\infty,
\end{equation}
and
$\mathbb{E}[(\sum_{n=1}^{\Pstar (t)} \int_{0}^{\ustar_n}\hstar (v)dv- \Pstar (t))^2]$ is equal to
\begin{equation}
\label{hrn_second_var}
\Ept{\left(\sum_{n=1}^{\Pstar (t)}y_n - \Pstar (t)\right)^2} = \Ept{\Pstar (t)}\text{Var}(y_1) =\Ept{\Pstar (t)}.
\end{equation}

Now, for the first term on the right-hand side of \eqref{hrn_temp0}, using \eqref{rhn_u0}, we obtain
\begin{align}\label{hrn_first_E}
  \Ept{\int_{\rstar_0}^{\rstar_0+\ustar_0} \hstar (v)dv} & =\frac{1}{\overline \Gstar (\rstar_0)} \int_0^\infty\left(\int_{\rstar_0}^{\rstar_0+u} \hstar (v)dv\right)\gstar (\rstar_0+u)du\notag\\
            & =\frac{1}{\overline \Gstar (\rstar_0)}\int_0^\infty\left(\log(\overline \Gstar (\rstar_0))-\log(\overline \Gstar (\rstar_0+u))\right) \gstar(\rstar_0+u)du\notag\\
            & =\frac{\log(\overline \Gstar (\rstar_0))}{\overline \Gstar (\rstar_0)}\int_{\rstar_0}^\infty \gstar (u)du -\frac{1}{\overline \Gstar (\rstar_0)}\int_0^{\overline \Gstar (\rstar_0)}log(v)dv \notag\\
            &=1,
\end{align}
and $\mathbb{E}[(\int_{\rstar_0}^{\rstar_0+\ustar_0} \hstar (v)dv)^2]$ is equal to
\begin{align}\label{hrn_first_var}
& \frac{1}{\overline \Gstar (\rstar_0)} \int_0^\infty\left(\int_{\rstar_0}^{\rstar_0+u} \hstar (v)dv\right)^2\gstar(\rstar_0+u)du            \notag \\
& \qquad =\frac{1}{\overline \Gstar (\rstar_0)}\int_0^\infty\left(\log(\overline \Gstar (\rstar_0))-\log(\overline \Gstar (\rstar_0+u))\right)^2 \gstar (\rstar_0+u)du\notag\\
            & \qquad =\log(\overline \Gstar (\rstar_0))^2 -2\frac{\log(\overline \Gstar (\rstar_0))}{\overline \Gstar (\rstar_0)}\int_0^{\overline \Gstar (\rstar_0)}log(v)dv + \frac{1}{\overline \Gstar (\rstar_0)}\int_0^{\overline \Gstar (\rstar_0)}(\log(v))^2 dv \notag\\
            &\qquad =\log(\overline \Gstar (\rstar_0))^2 -2 \log(\overline \Gstar(\rstar_0))(\log(\overline \Gstar(\rstar_0))-1)+(\log(\overline \Gstar (\rstar_0))^2-2\log(\overline \Gstar (\rstar_0))+2)\notag\\
            &\qquad =2.
\end{align}
For the last term  on the right-hand side of \eqref{hrn_temp0}, since $\rstar(t)\geq0$,
\begin{equation}\label{hrn_third_E}
\Ept{\int_{\rstar (t)}^{\ustar_{\Pstar(t)}}\hstar (v)dv}\leq \Ept{\int_{0}^{\ustar_{\Pstar(t)}}\hstar (v)dv} = \Ept{y_{\Pstar (t)}}=1
\end{equation}
and
\begin{equation}\label{hrn_third_var}
\Ept{\left(\int_{\rstar (t)}^{\ustar_{\Pstar (t)}}\hstar (v)dv\right)^2}\leq \Ept{\left(\int_0^{\ustar_{\Pstar (t)}}\hstar (v)dv\right)^2} = \Ept{(y_{\Pstar(t)})^2}=2.
\end{equation}
Then \eqref{hrn_second_E} follows on
taking  expectations of both sides of \eqref{hrn_temp0} and using \eqref{hrn_second_E}, \eqref{hrn_first_E}, and \eqref{hrn_third_E}, while  \eqref{renewal_Bound} follows on again applying \eqref{hrn_temp0},  the elementary bound $(a+b+c)^2 \leq 3(a^2 + b^2 + c^2)$, and  invoking \eqref{hrn_second_var}, \eqref{hrn_first_var} and \eqref{hrn_third_var}.
\end{proof}

\section{Proofs of Lemmas in Section 6.1.2} \label{apx_ggn}

In this section, we explain how the results on $\overline{\nu}_1^N$ and associated quantities stated in Section \ref{sec_tightDeparture} can be deduced  from the results on  the GI/GI/N queue in \cite{KasRam11}. Consider a GI/GI/N model that has the same  cumulative arrival process $E^{(N)}$,  iid service times  distributed according to $G$, and initial number of jobs $\xn(0)$ and age measure $\nu_1^{(N)} (0)$ as in our model. We recall that  in a GI/GI/N queue, arriving jobs choose an idle server at random if one is busy, or if all servers are busy,  join a common queue and enter service in a FCFS manner when servers become free (see \cite{KasRam11} for a more detailed description). We will use the tilde notation to denote  quantities associated with the GI/GI/N model, to distinguish them from the SQ($d$) model. Specifically, for $t \geq 0$, let  $\tdn(t)$, $\tkn(t)$ and $\txn(t)$  denote the cumulative number of departures from the system by time $t$,    cumulative number of entries into service by time $t$,  and the number of jobs in the system at time $t$, respectively.  Also, analogous to the definition of $\nun_1(t)$  in \eqref{def_nuln}, let $\tnun (t)$ be a finite measure that has a Dirac delta mass at the ages of each job in service at time $t$,  that is,
\[  \tnun (t)  \doteq \sum_{j \in \tV^{(N)}(t)}  \delta_{\ta_j^{(N)} (t)}, \]
where $\tvn (t)$ is the set of indices of jobs in service at time $t$ and $\ta_j$ is defined exactly as in \eqref{def_agen}, but with $\alpha_j^{(N)}$ replaced by $\tilde{\alpha}_j^{(N)}$, the time of service entry of job j in the GI/GI/N model and $\beta_j^{(N)}$ replaced by $\tilde{\beta}_j^{(N)} \doteq \tilde{\alpha}_j^{(N)} + v_j$. (Note that this coincides with the definition in  \cite[Eqns. (2.7)-(2.8)]{KasRam11} once one notes that the FCFS nature of GI/GI/N ensures that $\tvn (t)\subset \{-\langle \f1, \nu_1(0)\rangle, \ldots, \tkn(t)\}$.)

Simple mass balance relations \cite[Eqns. (2.5)-(2.6)]{KasRam11} then show that
\begin{equation}\label{tkn-mass}
    \tkn (t)  + \langle \f1, \tnun (0) \rangle \leq \txn (0) + E^{(N)} (t),
\end{equation}
and it follows from \cite[Theorem 5.1 and Eqn. (5.1)]{KasRam11}  that for any $\varphi \in \mathbb{C}_c^{1,1}([0,\endsup) \times \R_+)$ and $t \in [0,\infty)$,
\begin{align}\label{apx_dymamics_eq}
  \langle\varphi(\cdot,t),\tnun (t)\rangle = &\langle\varphi(\cdot,0),\tnun(0)\rangle + \int_0^t \langle \varphi_s(\cdot,s)+\varphi_x(\cdot,s),\tnun(s)\rangle ds \notag\\
  &- \tQuen_{\varphi}(t)    +\int_{[0,t]}\varphi(0,s)d\tkn(s),
\end{align}
where $\tQuen$, which  is  analogous to the quantity  ${\mathcal Q}^{(N)}$ in  \cite[Eqn. (5.1)]{KasRam11}, satisfies
\[
\tQuen_{\varphi} (t) \doteq \sum_{j \in \Z: \tilde\beta_j^{(N)} \leq t}  \varphi (v_j, \tilde\beta_j^{(N)}).
\]
Also, for $\varphi \in  \tightphiset$, defining
\[
    \tilAn_{\varphi}  (t) \doteq  \int_0^t \langle \varphi (\cdot, s) h(\cdot), \tnun_s  \rangle  ds, \quad t \geq 0,
\]
it follows from  \cite[Lemma 5.4 and Corollary 5.5]{KasRam11} that $\tmn_{\varphi} \doteq \tQuen_{\varphi} - \tilAn_{\varphi}$ is a martingale with respect to a certain filtration $\{\tilde {\mathcal{F}}_t\}$ (defined in \cite[page 40]{KasRam11} and denoted by $\{\mathcal{F}_t\}$ there).

\begin{remark}\label{rem-ggnnew}
Assumptions 1 and 2 in \cite{KasRam11} follow from Assumptions \ref{asm_E}-\ref{asm_initial} of this paper. Specifically, Lemma \ref{asm_initial_remark} shows that  Assumption 1(1) in \cite{KasRam11} follows from Assumption \ref{asm_E}. Also,  given that we set $\txn(0)=\xn(0)$ and $\tnun(0)=\nun_1(0)$,  Assumptions 1(2) and 1(3) in \cite{KasRam11} follow from Assumptions \ref{asm_initial}.\ref{asm_initial_nu} and \ref{asm_initial}.\ref{asm_initial_X}, respectively. Finally, Assumption 2 in \cite{KasRam11} follows from  Assumption \ref{asm_G} of this paper.
\end{remark}

We start by providing the proof of Lemma \ref{lem_nuContainment}.

\begin{proof}[Proof of Lemma \ref{lem_nuContainment}]
Due to the coupling of the  initial conditions, $\tnun(0)=\nun_1(0)$, it follows from Remark \ref{rem-ggnnew} that if $\{\nun(0)\}$ satisfies Assumption \ref{asm_initial}, then  $\{\tnun(0)\}$ satisfies  Assumption 1(3) of \cite{KasRam11}.   Thus, it follows from \cite[Lemma 5.12]{KasRam11}  that (5.33) and (5.34) of \cite{KasRam11} hold, which correspond exactly to \eqref{533} and \eqref{534} herein due to the fact that  $\tnun(0)=\nun_1(0)$.
\end{proof}

To prove Lemma \ref{lem_tightDepart}, we need the following result.

\begin{lemma}\label{lem_DAbounds}
Suppose Assumptions \ref{asm_E}-\ref{asm_initial} hold, and fix $T>0$.
\renewcommand{\theenumi}{\alph{enumi}}
\begin{enumerate}
\item     For every non-negative function $f\in\mathbb{L}^1\supint$, there exists a constant $C<\infty$ such that
        \begin{equation}\label{apx57}
        \sup_N\Ept{\int_0^T\langle f,\nunbar_1(s)\rangle ds} \leq C \sup_{u\in\supint}\int_u^{(u+T)\wedge \endsup}f(x)dx.
        \end{equation}
\item For every $N\in\N$,
        \begin{equation}\label{apx581}
            \lim_{m\uparrow\endsup}\sup_N\Ept{\int_0^T\left(\int_{[m,\endsup)}h(x)\nunbar_1(s,dx)\right)ds}=0.
        \end{equation}
\item For every $N\in\N$ and  $\varphi\in\tightphiset$,
\begin{equation}\label{apx_Atight_K2}
    \lim_{\delta\to0}\sup_N\Ept{\sup_{t\in[0,T]} \left|\Anbar_{\varphi,1}(t+\delta)-\Anbar_{\varphi,1}(t)\right|}=0.
\end{equation}
\end{enumerate}
\renewcommand{\theenumi}{\arabic{enumi}}
\end{lemma}

\begin{proof}
For part a, the analogous result for the GI/GI/N queue follows from \cite[Proposition 5.7]{KasRam11}. Using \cite[Assumption 1]{KasRam11}  and the bound \eqref{tkn-mass} above, it  was shown in \cite[Equation (5.31)]{KasRam11} that  for every $f\in\mathbb{L}^1\supint$, $\varphi\in\tightphiset$, and $0 \leq r\leq t$,
\begin{equation}\label{531}
  \frac{1}{N}\left| \int_r^t \langle \varphi(\cdot,s)f(\cdot),\tnun(s)\rangle  \right| \leq \|\varphi\|_\infty \frac{1}{N}\left(\txn(0)+E^{(N)}(t)\right)\sup_{u\in\supint}\int_u^{u+t-r\wedge \endsup} |f(x)|dx.
\end{equation}
Setting $\varphi=\f1$, $r=0$, and $t=T$ in the bound above, and taking first expectations and then the limit superior of both sides of \eqref{531}, \eqref{apx57} holds with $\nunbar_1$ replace by $\frac{1}{N}\tnun$ and $C$ replaced by $\tilde C\doteq\limsup_{N}\frac{1}{N}\Ept{\txn(0)+E^{(N)}(t)},$ which is finite by \cite[Assumption 1]{KasRam11} and Remark \ref{rem-ggnnew}.  To prove \eqref{apx57} for the SQ($d$) model,  note that by  \eqref{KDERelation}, \eqref{prelim_dymamics_balance}  with $\ell  = 1$, \eqref{def_Rphi1},  and \eqref{apx_massTotal}, the service entry process $\kn$ satisfies the bound
\begin{equation}\label{kbound}
    \kn (t) +  \langle \f1, \nun_1 (0) \rangle \leq \xn (0) + E^{(N)} (t)
\end{equation}
which is analogous to \eqref{tkn-mass}. Therefore, by Assumption \ref{asm_initial} and Remark \ref{rem-ggnnew}, \eqref{531} (with $\tnun$ replaced by $\nun_1$) can be established for the SQ($d$) model, exactly as in the proof of \cite[Proposition 5.7]{KasRam11}. The bound \eqref{apx57} then follows, as above, with $C\doteq\limsup_N\Eptil{\xnbar(0)+\arrivalnbar(t)}$, which is finite by Assumption \ref{asm_initial}.

For part b, the analogous result for the GI/GI/N queue is proved in \cite[Lemma 5.8(1)]{KasRam11}. Using equations \eqref{533} and \eqref{534} above, with $\nunbar_1(0)$ replaced by $\frac{1}{N}\tnun(0)$, the results \cite[Corollary 5.5 and equation (5.29) of Lemma 5.6]{KasRam11}, and the bound \eqref{531} above, it has been shown that \eqref{apx581} holds with $\nunbar_1$ replaced by $\frac{1}{N}\tnun$; see \cite[proof of Lemma 5.8(1)]{KasRam11}. To prove the result for the SQ($d$) model, note that the analogue of \cite[Corollary 5.5]{KasRam11} is Proposition \ref{prop_prelim_Qcomp} and an analogue of \cite[Equation (5.29)]{KasRam11} can be proved with $\tilQuen_\varphi$ replaced by $\Quen_{\varphi,1}$, using the same argument as in \cite[Lemma 5.6]{KasRam11}.   Combining this with  equations \eqref{533} and \eqref{534} and part a. of this Lemma, \eqref{apx581} follows from the same argument as used in \cite[Equation (5.29)]{KasRam11}.

For part c, the analogous result for the GI/GI/N queue is proved in \cite[Lemma 5.8(2)]{KasRam11}. Using \cite[Equation (5.31)]{KasRam11}, \cite[Lemma 5.8(1)]{KasRam11},  \cite[Remark 3.1]{KasRam11} (which follows from \cite[Assumption 1]{KasRam11}), \eqref{apx_Atight_K2} is proved with $A^{(N)}_{ \varphi ,1}$ replaced by $ \frac{1}{N}\tilde A^{(N)}_{\varphi}$; see \cite[Lemma 5.8(2)]{KasRam11}. To prove a version of this result for the $SQ($d$)$ model, note that equivalences of \cite[Equation (5.31)]{KasRam11} and \cite[Lemma 5.8(1)]{KasRam11}  for the $SQ($d$)$  model are proved in parts a and b of this lemma, respectively, while the equivalence of \cite[Remark 3.1]{KasRam11} is proved in \eqref{asm_initial_Xarrivalmoment} of Lemma \ref{asm_initial_remark}. Equation \eqref{apx_Atight_K2} can then be proved using the exact same argument as in the proof of \cite[Lemma 5.8(2)]{KasRam11}.

\end{proof}

\begin{proof}[Proof of Lemma \ref{lem_tightDepart}]
Since $\Mn_{\varphi,\ell}$ is a martingale by Lemma \ref{lem_Mconv}, for  $t\geq0$, $ \Eptil{\Mn_{\varphi,\ell}(t)}=0$, and hence $\Eptil{\An_{\varphi,\ell}(t)}=\Eptil{\Quen_{\varphi,\ell}(t)}$. Therefore, by Fatou's lemma, we have
\begin{equation}\label{apx_Atight_K1}
  \limsup_{N\to\infty}\Ept{\left|\Anbar_{\varphi,\ell}(t)\right|}\leq
\limsup_{N\to\infty}\Ept{\Anbar_{|\varphi|,\ell}(t)} = \Ept{\Quenbar_{|\varphi|,\ell}(t)} <\infty,
\end{equation}
where the finiteness follows from \eqref{Qn_bound}, with $\varphi$ replaced by $\sqrt{|\varphi|}$. Kurtz's K1 and K2b conditions for relative compactness of the sequence  $\{\Anbar_{\varphi,1}\}_{N\in\N}$ follow from \eqref{apx_Atight_K1}, with $\ell=1$, and \eqref{apx_Atight_K2} of Lemma \ref{lem_DAbounds}, respectively.  In turn, this also implies that  Kurtz's conditions are satisfied for $\{\Anbar_{\varphi,\ell}\}_{N\in\N}$ for $\ell \geq 2$ because the fact that $\langle f,\nun_1\rangle\geq \langle f,\nun_\ell\rangle$ for every non-negative function $f$, implies that for $0\leq s\leq t$,  we have $|\An_{\varphi,\ell}(t)|\leq \An_{|\varphi|,1}(t)$, and $\left|\An_{\varphi,\ell}(t)-\An_{\varphi,\ell}(s)\right|\leq \left|\An_{|\varphi|,1}(t)-\An_{|\varphi|,1}(s)\right|.$ This proves the relative compactness of  $\{\Anbar_{\varphi,\ell}\}_{N\in\N}$ for  $\ell \geq 1$.   Now, by Lemma \ref{lem_Mconv},  the sequence $\{\Mnbar_{\varphi,\ell}\}_{N\in\N}$ converges weakly to zero and thus is relatively compact, and hence,  by \eqref{def_Mn}, the sequence $\{\Quenbar_{\varphi,\ell}\}$ is also relatively compact. Setting $\varphi=\f1$, this implies the relative compactness of $\{\dnbar_\ell\}_{N\in\N}$ for $\ell\geq1$.
\end{proof}

Next, to prove Lemma \ref{lem_Qtight}, we need the following result.

\begin{lemma}\label{lem_Dcontainment}
Suppose Assumptions \ref{asm_E}-\ref{asm_initial} hold.  Then, for $\eta>0$ and $T\geq 0$, there exist a constant $B(\eta)<\infty$ and a sequence $\{m(n,\eta)\}$ with $m(n,\eta)\to\endsup$ as $n\to \infty$ such that
\begin{equation}  \label{apx_compactProb}
    \Prob{\Quenbar_{\cdot,1}(t)\not\in\mathcal{K}_\eta \text{ for some }t\in[0,T]  }\leq \eta,
\end{equation}
for the compact subset $\mathcal{K}_\eta \subset \mathbb{M}_F(\supint\times\R_+)$ defined by
\begin{equation}\label{apx_compactSet}
  \mathcal{K}_\eta \doteq \left\{\mu\in\mathbb{M}_F(\supint\times\R_+)\; : \langle\f1,\mu \rangle\leq B(\eta),\;\mu((m(n,\eta),\endsup)\times R_+)\leq \frac{1}{n}\;\forall n\in\N  \right\}.
\end{equation}
\end{lemma}

\begin{proof}
For every $N\in\N$, $t\geq0$, the linear functional $\tilQuen_{\cdot}(t):\varphi\mapsto\tilQuen_{\varphi}(t)$  can be identified with a finite non-negative Radon measure on $\supint\times\R_+$  (see, e.g., \cite[p.\ 96]{KasRam11}, where this quantity is denoted by $\cal {Q}^{(N)}$). The analogous result for the GI/GI/N model (with $\Quenbar_{\cdot,1}$ replaced by $\frac{1}{N}\tilQuen_{\cdot}$) is proved in \cite[Lemma 5.13]{KasRam11} (with $\overline{\cal Q}^{(N)}$ there replaced by $\frac{1}{N}\tilQuen_{\cdot}$ here), using the results in \cite[Lemma 5.8(1)]{KasRam11}. As discussed in Lemma \ref{lem_DAbounds} above, an equivalent of \cite[Lemma 5.8]{KasRam11} holds for the SQ($d$) model.  Thus the result follows from an argument exactly analogous to the one in the proof of \cite[Lemma 5.13]{KasRam11}.
\end{proof}

\begin{proof}[Proof  of Lemma \ref{lem_Qtight}]
For $\ell \geq 1$, condition J2 of Proposition \ref{Jakob} holds by  Lemma \ref{lem_tightDepart}. For $\ell=1$, condition J1 follows from Lemma \ref{lem_Dcontainment} above. Furthermore, for $\ell\geq 2$,  by \eqref{def_Qphi},  for every non-negative measurable function $\varphi$ and $t\geq 0$, $\Quen_{\varphi,\ell}(t)\leq \Quen_{\varphi,1}(t)$.  Thus, the bound \eqref{apx_compactProb} holds  with the same compact set $\mathcal{K}_\eta$ defined in \eqref{apx_compactSet}, even when $\Quen_{\varphi,1}$ is  replaced by $\Quen_{\varphi,\ell}$, and so condition J1 holds for $\Quenbar_{\cdot,\ell}, \ell\geq2$.
\end{proof}

Finally, we provide proofs of Lemmas \ref{lem_nuContainment2} and \ref{lem_GGNconv}.

\begin{proof}[Proof of Lemma \ref{lem_nuContainment2}]

The analogous result for the GI/GI/N model is proved in \cite[Lemma 5.12]{KasRam11} with $\nunbar_1$ is substituted by $\frac{1}{N}\tnun$. The proof only uses relations (5.33) and (5.34) in \cite{KasRam11}, and the fact that the ages of jobs in service grow linearly in time.  The equivalent of \cite[Relations (5.33)-(5.34)]{KasRam11} for the SQ($d$) model (with $\nun$ replaced by $\nunbar_1$) was shown above in Lemma \ref{lem_nuContainment}, and thus an exactly analogous argument yields the corresponding result  for $\nunbar_1$.

\end{proof}

\begin{proof}[Proof of Lemma \ref{lem_GGNconv}]

We first prove the following estimate: for every $m<\endsup$ and every non-negative function $f\in\mathbb{L}^1_{\text{loc}}\supint$ with support in $[0,m]$, there exists $\tilde L(m,T)<\infty$ such that
\begin{equation}\label{apx517}
    \left|\int_0^T\langle f,\nu_1(s)\rangle ds\right|\leq \tilde L(m,T)\int_{\supint}f(x)dx.
\end{equation}
To prove the claim, note that  the analogous result for the GI/GI/N queue is proved in \cite[Lemma 5.16]{KasRam11}.  Specifically, using  \eqref{apx_dymamics_eq} and the bound \eqref{tkn-mass} above as well as Assumption 1 and Lemma 5.9 in \cite{KasRam11} (which shows that $\frac{1}{N}\tmn$ converges to zero as $N\to\infty$), it was shown in \cite[Equation (5.46)]{KasRam11} that  \eqref{apx517} holds with  $\nu_1$ replaced by any subsequential limit $\tilde \nu$ of $\frac{1}{N}\tnun$.  To obtain the result for the SQ($d$) model,  note that by equations \eqref{prelim_dymamics_eq}, \eqref{KDERelation} and \eqref{def_Rphi1},  \eqref{apx_dymamics_eq} is satisfied with $\tnun$ replaced by $\nun_1$, $\tilQuen_{\varphi}$ replaced by $\Quen_{\varphi,1}$ and $\tkn$ replaced by $\kn$. Therefore, the result can be deduced from  the bound \eqref{kbound}, Assumption \ref{asm_initial} and Remark \ref{rem-ggnnew}, and Lemma \ref{lem_Mconv}, which is analogous to \cite[Lemma 5.9]{KasRam11},  \eqref{apx517}, using the same argument as in the proof Lemma \cite[Lemma 5.16]{KasRam11}.

Now we prove \eqref{apx_Alimit2}. For $\ell=1$ the  equation \eqref{apx_Alimit2} can be proved in the same manner as the corresponding result \cite[Proposition 5.17]{KasRam11} for the GI/GI/N model, using Lemma \ref{lem_DAbounds}(b), in place of the analogous \cite[Lemma 5.8(1)]{KasRam11}. For $\ell\geq 2$, since $\nunbar(t)$  and $\nu(t)$ both lie in $\s$, for every $N\in\N$ and  $t\geq0$, $\langle f,\nun_{\ell}(t)\rangle \leq\langle f,\nun_1(t)\rangle$,  and $\langle f,\nu_{\ell}(t)\rangle \leq\langle f,\nu_{1}(t)\rangle$  for all $\ell\geq 1$ and every non-negative measurable function $f$. Thus, the bounds \eqref{apx581} and \eqref{apx_Atight_K2} hold with $\nunbar_1$ and $\Anbar_{\varphi,1}$ replaced by $\nunbar_\ell$ and $\Anbar_{\varphi,\ell}$ respectively, \eqref{apx_compactProb} holds with $\overline{\mathcal{D}}^{(N)}_{\cdot,1}$ replaced by $\overline{\mathcal{D}}^{(N)}_{\cdot,\ell}$, and \eqref{apx517} holds with $\nu_1$ replaced by $\nu_\ell$. Therefore, \eqref{apx_Alimit2} for $\ell\geq2$ follows from the same argument as for the case $\ell=1$.

\end{proof}

\bibliographystyle{plain}
\bibliography{reference}

\end{document}